\documentclass[11pt]{amsart}
\usepackage{amssymb}
\usepackage{verbatim}  
\usepackage{color}\usepackage{color}
\usepackage{tikz-cd}
\usepackage{enumerate}

\newcommand{\TryPackage}[3]{\IfFileExists{#1.sty}{\usepackage{#1}#2}{#3}}
\TryPackage{mathrsfs}{\renewcommand{\mathcal}{\mathscr}}{%
     \TryPackage{eucal}{}{}}
\newcommand*\wbar[1]{
  \hbox{ \kern0.01em%
    \vbox{%
      \hrule height 0.5pt  
      \kern0.25ex
      \hbox{%
        \kern-0.15em
        \ensuremath{#1}%
      }%
    }%
  \kern0.15em}%
}

\newcommand{\lto}{\longrightarrow}
\newcommand{\wh}{\widehat}
\newcommand{\wt}{\widetilde}
\newcommand{\sm}{\smallsetminus}

\newcommand{\al}{\alpha}
\newcommand{\be}{\beta}
\newcommand{\ga}{\gamma}

\newcommand{\ep}{\varepsilon}
\renewcommand{\th}{\theta}
\newcommand{\la}{\lambda}
\renewcommand{\phi}{\varphi}
\newcommand{\si}{\sigma}

\newcommand{\Ga}{\Gamma}
\newcommand{\La}{\Lambda}
\newcommand{\Om}{\Omega}

\newcommand{\Si}{\Sigma}
\newcommand{\ZZ}{{\mathbb Z}}
\newcommand{\CC}{{\mathbb C}}
\newcommand{\FF}{{\mathbb F}}

\newcommand{\RR}{{\mathbb R}}

\newcommand{\cE}{\mathcal E}
\newcommand{\cI}{\mathcal I}
\newcommand{\cL}{\mathcal L}

\newcommand{\writhe}{\operatorname{\it wr}}

\newcommand{\width}{\operatorname{width}}

\newcommand{\rel}{\operatorname{rel}}

\newcommand{\Aut}{\operatorname{Aut}}

\newcommand{\VG}{\operatorname{\it VG}}
\newcommand{\WG}{\operatorname{\it WG}}

\newcommand{\VB}{\operatorname{\it VB}}

\newcommand{\EG}{\operatorname{\it EG}}
\newcommand{\QG}{\operatorname{\it QG}}
\newcommand*{\weq}{%
      \mathrel{\vcenter{\offinterlineskip
      \hbox{$\sim$}\vskip0.4mm\hbox{{\tiny \, \!\!\!\! $w$}}}}}

\newtheorem{theorem}{Theorem}[section]
\newtheorem{lemma}[theorem]{Lemma}
\newtheorem{proposition}[theorem]{Proposition}

\newtheorem{corollary}[theorem]{Corollary}

\theoremstyle{definition}                                     
\newtheorem{definition}[theorem]{Definition}

\theoremstyle{remark}
\newtheorem{remark}[theorem]{Remark}
\newtheorem{example}[theorem]{Example}

\begin{document}

\title[Alexander invariants  for virtual knots]{Alexander invariants for virtual knots}
\author[Boden]{Hans U. Boden}
\address{Mathematics \& Statistics, McMaster University, Hamilton, Ontario}
\email{boden@mcmaster.ca}

\author[Dies]{Emily Dies}
\address{Mathematics \& Statistics, McMaster University, Hamilton, Ontario}
\email{diesej@mcmaster.ca}

\author[Gaudreau]{Anne Isabel Gaudreau}
\address{Mathematics \& Statistics, McMaster University, Hamilton, Ontario}
\email{gaudreai@mcmaster.ca}

\author[Gerlings]{Adam Gerlings}
\address{Mathematics \& Statistics, University of Calgary, Calgary, Alberta}
\email{agerling@ucalgary.ca}

\author[Harper]{Eric Harper}
\address{Mathematics \& Statistics, McMaster University, Hamilton, Ontario}
\email{eharper@math.mcmaster.ca}

\author[Nicas]{Andrew J. Nicas}
\address{Mathematics \& Statistics, McMaster University, Hamilton, Ontario}
\email{nicas@mcmaster.ca}

\subjclass[2010]{Primary: 57M25, Secondary: 57M27, 20C15}
\keywords{Virtual knots and links, virtual knot groups, Alexander invariants, twisted Alexander polynomials, virtual crossing number}

\date{\today}
\begin{abstract}
Given a virtual knot $K$, we introduce a new group-valued invariant  $\VG_K$  called the virtual knot group, and we use the elementary ideals of $\VG_K$ to define invariants of $K$ called the virtual Alexander invariants. For instance, associated to the $0$-th ideal is a polynomial $H_K(s,t,q)$ in three variables which we call the virtual Alexander polynomial, and we show that it is closely related to the generalized Alexander polynomial $G_K(s,t)$ introduced in \cite{Sawollek, KR, SW-Alexander}. We define a natural normalization of the virtual Alexander polynomial and show it satisfies a skein formula. We also introduce the twisted virtual Alexander polynomial associated to a virtual knot $K$ and a representation $\varrho \colon \VG_K \to GL_n(R)$, and we define a normalization of the twisted virtual Alexander polynomial. As applications we derive bounds on the virtual crossing numbers of virtual knots from the virtual Alexander polynomial and twisted virtual Alexander polynomial. 
\end{abstract}
\maketitle


\section*{Introduction} 

Given a virtual knot $K$, we define a finitely presented group, denoted $\VG_K$, which we call the virtual knot group of $K$. Although the presentation of $\VG_K$ depends on the diagram of the virtual knot, we show that the group depends only on the virtual knot $K$. Thus $\VG_K$ is an invariant of the virtual knot $K.$ The abelianization of $\VG_K$ is $\ZZ^3$, the free abelian group  on three generators, and  we study the Alexander invariant associated to the surjection $\VG_K \to \ZZ^3$. As an invariant of virtual knots, the Alexander invariant is quite useful in that it gives an obstruction to $K$ being classical. For instance, we define the virtual Alexander polynomial  $H_K(s,t,q) = \Delta^0_K(s,t,q)$, and if $K$ is a classical knot, then it follows that $H_K(s,t,q)$ must vanish. 

We establish a skein formula for the virtual Alexander polynomial  and relate it to  the generalized Alexander polynomial $G_K(s,t)$, which was defined by Sawollek \cite{Sawollek}, Kauffman and Radford  \cite{KR}, and  Silver and Williams \cite{SW-Alexander}. Proposition \ref{HdeterminesG} and Corollary \ref{GdeterminesH} show that (i) $G_K(s,t)=H_K(s,t,1)$  up to units in  $\ZZ[s^{\pm 1}, t^{\pm 1}],$ and (ii) $H_K(s,t,q) = G_K(sq^{-1}, tq)$ up to units in $\ZZ[s^{\pm 1}, t^{\pm 1}, q^{\pm 1}],$  and it follows that the virtual Alexander polynomial $H_K(s,t,q)$ determines the generalized Alexander polynomial $G_K(s,t)$ and vice versa.

The virtual polynomial $H_K(s,t,q)$ encodes information about the virtual crossing number $v(K)$, in particular the $q$-$\width$ of $H_K(s,t,q)$ gives a lower bound on $v(K)$. We give examples to show that the bound on $v(K)$ obtained from $H_K(s,t,q)$ is sometimes stronger than the bound obtained from the arrow polynomial of Dye and Kauffman \cite{DK, BDK}.  For a virtual knot $K$ expressed as the closure of a virtual braid $\be$, the virtual Alexander invariant can be computed in terms of the virtual Burau representation $\Psi(\be)$, and using this approach we show how to define the normalized virtual Alexander polynomial ${\wh H}_K(s,t,q)$. In Theorem \ref{better-v-bound}, we show that the normalized invariant ${\wh H}_K(s,t,q)$ gives even stronger bounds on the virtual crossing number $v(K)$, and we present several examples where one obtains sharp results by applying this theorem.

Our approach to the virtual Alexander invariants is to define and compute them in terms of the virtual knot group $\VG_K,$ which is a new virtual knot invariant. Given any diagram of a virtual knot $K$, one can write down a presentation for $\VG_K$ in much the same way as is done for the classical knot group $G_K$. The generators of $\VG_K$ are given by all the short arcs of the diagram along with two auxiliary generators $s$ and $q$. The generator $s$ is used to relate adjacent arcs across over-crossings and the generator $q$ is used to relate adjacent arcs across virtual crossings. In order to obtain invariance under the generalized Reidemeister moves, it is necessary that the variables $s$ and $q$ commute. 

Although one can use biquandles (or virtual quandles) to define Alexander invariants of virtual knots (cf. \cite{KR} and \cite{FKM}), there are several advantages to our approach.
 One is that the virtual Alexander invariants can be computed using any presentation of the virtual knot group $\VG_K.$ Another is that it provides a common framework for Alexander invariants of virtual knots and links, and 
one could use the same ideas to develop Alexander invariants for many of the other group-valued invariants of virtual and welded knots (cf. \cite{BB}). Finally, our approach can be also used to define twisted virtual Alexander invariants associated to representations $\varrho \colon \VG_K \to GL_r(R)$ just as in the classical case.

Recall that twisted Alexander polynomials for knots in $S^3$ were first introduced by Lin \cite{Lin} using free Seifert surfaces, and  
Wada then extended the definition to finitely presented groups using the Fox differential calculus in \cite{Wada}. The resulting construction produces invariants of knots, links, and 3-manifolds that are remarkably strong and yet also highly computable. For instance, the twisted Alexander polynomials have been shown to detect the unknot and the unlink in \cite{SW-06, Friedl-Vidussi-unlink}, and to detect fibered 3-manifolds in \cite{Agol, Friedl-Vidussi}.  
The twisted Alexander polynomials have been applied to study slice  knots \cite{Kirk-Livingston-Twisted, Herald-Kirk-Livingston}, periodic knots  \cite{Hillman-Livingston-Naik}, and an interesting partial ordering on knots \cite{Kitano-Suzuki-Wada, Kitano-Suzuki, Horie-Kitano-Matsumoto-Suzuki}.     

In Section \ref{section-7}, we extend the notion of twisted Alexander polynomials to virtual knots and links, addressing a question posed in \cite{CHN}, and we show that the  resulting invariants often give stronger bounds on the virtual crossing number $v(K)$, see Theorem \ref{twist-v-bound}.  In Section \ref{section-8}, we show how to compute the twisted invariants in terms of braids by using the twisted virtual Burau representation. In Section \ref{section-9}, we examine how the invariants change under the virtual Markov moves and define a preferred normalization of the twisted virtual Alexander polynomial that has less indeterminacy and gives an improved bound on the virtual crossing number, see 
Theorem \ref{better-twist-v-bound}. 
 
Jeremy Green has classified virtual knots up to four crossings, and in this paper we will use his notation to refer to specific virtual knots in the table at \cite{green}.
 
\section{A brief introduction to virtual knots} \label{section-1}
In this section we give a quick introduction to the theory of virtual knots.

Virtual knot theory was first introduced by Kauffman in \cite{KVKT}. There are several equivalent approaches, and we define an oriented virtual link to be an equivalence class of oriented virtual link diagrams. A virtual link diagram is an immersion of one or several circles in the plane with only double points, such that each double point is either classical (indicated by over- and under-crossings) or virtual (indicated by a circle). An oriented virtual link diagram is a virtual link diagram endowed with an orientation for each component, and two virtual link diagrams are said to be equivalent if they can be related by planar isotopies and a series of \emph{generalized Reidemeister moves} ($r1$)--($r3$) and ($v1$)--($v4$) depicted in Figure \ref{VRM}. It was proved in \cite{GPV} that if two classical knot diagrams are equivalent under the generalized Reidemeister moves, then they are equivalent under the classical Reidemeister moves, and consequently this shows that the theory of classical knots embeds into the theory of virtual knots.

\begin{figure}[ht]
\centering
\def\svgwidth{350pt}
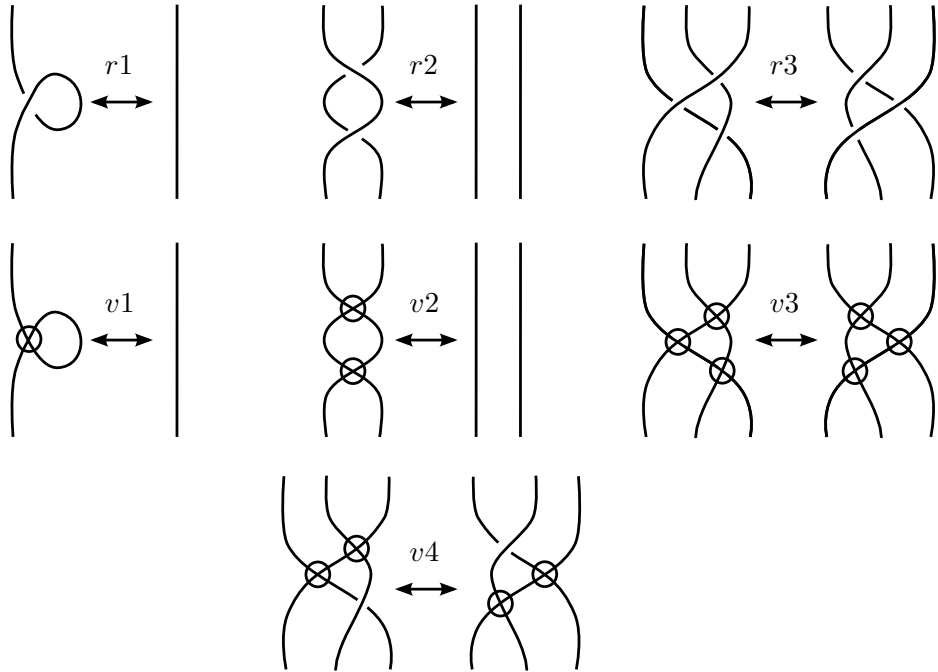
\caption{The generalized Reidemeister moves ($r1$)--($r3$) and ($v1$)--($v4$).}
\label{VRM}
\end{figure}

Virtual knots can also be described quite naturally as equivalence classes of Gauss diagrams. 
Given a knot $K$, its Gauss diagram consists of a circle, which represents points on $K$, along with directed chords from each over-crossing of $K$ to the corresponding under-crossing. The chords are given a sign $(\pm)$ according to whether the crossing is positive or negative. The Reidemeister moves can be translated into moves between Gauss diagrams, and in this way one can regard a classical knot as an equivalence class of Gauss diagrams. Every classical knot corresponds to some Gauss diagram, but not all Gauss diagrams are associated with a classical knot. This deficiency disappears by passing to the larger category of virtual knots, which admit the alternative definition as equivalence classes of Gauss diagrams. 

It is interesting that the Gauss diagram of a virtual knot does not explicitly indicate where to draw the virtual crossings; they are just a by-product of attempting to draw the virtual knot from its Gauss diagram.
For this reason, it is often useful to ignore the virtual crossings for the purpose of defining invariants and indeed
many invariants of classical knots can be extended to the virtual setting by means of this simple strategy.
Notable examples include the Jones polynomial and the knot group $G_K$.     
 
Virtual knots can be described geometrically as knots in thickened surfaces. Let $\Si_g$ be an oriented surface of genus $g$ and set $I = [0,1]$, and consider a link $L \subset \Si_g \times I$. A diagram of $L$ is obtained by projection onto a plane, keeping track of over- and under-crossings. In the case $g=0$, there is essentially no loss of information in projecting, and this corresponds to the case of a classical link. However, if $g \geq 1$, then projecting $L$ from the thickened surface to the plane may introduce additional self-intersections. These crossings correspond to virtual crossings, and in this way one obtains a virtual knot diagram from the projection of $L$ to the plane. 

In \cite{CKS}, the authors established a one-to-one correspondence between virtual links and stable equivalence classes of projections of links onto compact, oriented surfaces. (We refer the reader to \cite{CKS} for the definition of stable equivalence.) Thus,  
the following result, which is due to Kuperberg \cite{Kuperberg}, provides  a geometric interpretation of virtual knots and links in terms of knots in thickened surfaces.
 
\begin{theorem}
Every stable equivalence class of links in thickened surfaces has a unique irreducible representative.
\end{theorem}

Using Kuperberg's theorem, we can define the genus $g(K)$ for any virtual knot $K$ to be the genus of the surface of its unique irreducible representative. Obviously $g(K)=0$ if and only if $K$ is classical. Another closely related invariant of virtual knots is its virtual crossing number, which is defined as follows. 
Given a virtual knot or link diagram $D$, let $v(D)$ denote the number of virtual crossings of $D$. For a virtual knot or link $K$,  the virtual crossing number $v(K)$ is defined to be the minimum of $v(D)$ over all diagrams $D$ representing $K$. Clearly $K$ is classical if and only if $v(K)=0.$ In general, we have the inequality $g(K) \leq v(K).$

\begin{figure}[ht]
\centering
\def\svgwidth{8cm}
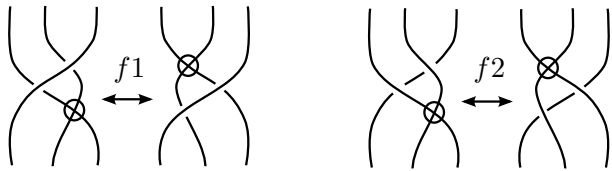
\caption{The two forbidden moves $(f1)$ and $(f2)$.}
\label{FRM}
\end{figure}    

Two virtual knots or links are said to be \emph{welded equivalent} if one can be obtained from the other by generalized Reidemeister moves plus the first forbidden move ($f1$) depicted in Figure \ref{FRM}. It was proved independently by Kanenobu and Nelson that  by allowing both forbidden moves $(f1)$ and $(f2)$, every virtual knot diagram can be shown to be equivalent to the unknot \cite{Kanenobu, Nelson-for}. Allowing only the first forbidden move $(f1)$, also called the forbidden overpass, leads to a nontrivial theory. A \emph{welded knot} is an equivalence class of virtual knot diagrams under the moves $(r1)$--$(r3), (v1)$--$(v4),$ and $(f1).$  
Given two virtual knots $K_1$ and $K_2$, we write $K_1 \weq K_2$ if $K_1$ is welded equivalent to $K_2$. 
In terms of Gauss diagrams, the first forbidden move corresponds to exchanging two adjacent arrow feet without changing their signs or arrowheads \cite{State}.  Therefore, a welded knot can also be viewed as an equivalence class of Gauss diagrams. In \cite{Rourke}, it is proved that two classical knots are welded equivalent if and only if they are equivalent as classical knots, and this shows that the theory of classical knots embeds into the theory of welded knots.

Many invariants in knot theory extend in a natural way to virtual knots, including the Jones polynomial and the knot group $G_K$. Since it is central to many of our later constructions, we review the definition of $G_K$. Recall that  for a classical knot $K \subset \RR^3$, the knot group $G_K = \pi_1(\RR^3 \sm K, x_0)$ is the fundamental group of the complement of the knot, relative to some basepoint $x_0 \in \RR^3$. The Wirtinger presentation for $G_K$, which depends on a diagram of $K$, has one generator for each arc and one relation for each crossing, depending on the sign of the crossing as indicated Figure \ref{KGR}.   

\begin{figure}[ht]
\centering
\includegraphics[scale=0.90]{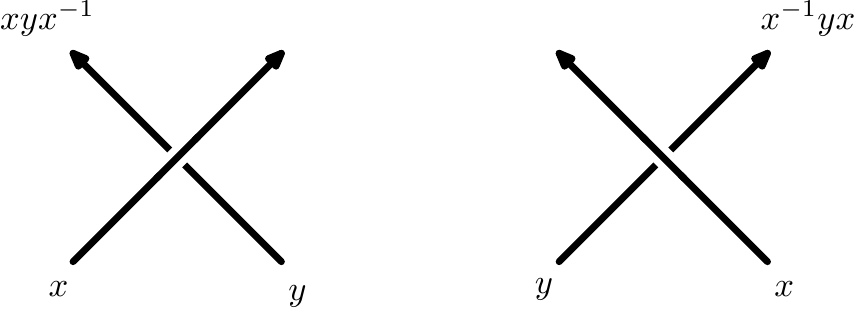}
\caption{Relations in the knot group $G_K$ at a classical crossing.}
\label{KGR}
\end{figure}

Even though virtual knots are not embeddings in $\RR^3$, one can apply the same procedure and associate to each virtual knot $K$ an abstract group $G_K$. The group $G_K$ has one generator for each \emph{long arc}, which is an arc that starts and ends at under-crossings, and one relation for each classical crossing, just as in   Figure \ref{KGR}. In this construction, one ignores the virtual crossings of $K$. The resulting group, denoted $G_K$, was first introduced by Kauffman in \cite{KVKT}, and it is given by the (upper) Wirtinger presentation.
If one regards $K$ as a knot in the thickened surface $\Si_g \times I$, then $G_K$ is the fundamental group of the complement of $K$ in $\Si_g \times I/\sim$,  where $x \sim y$ for all $x,y \in \Si_g \times \{0\}$; for a proof of this fact see \cite[Proposition 5.1]{Kamada-Kamada} by N.~Kamada and S.~Kamada. Using an analogous construction, one can define the \emph{lower} Wirtinger presentation, and geometrically it is the fundamental group of the complement of $K$ in $\Si_g \times I/\sim$,  where $x \sim y$ for all $x,y \in \Si_g \times \{1\}.$ It turns out that the lower group is the knot group of the knot obtained from $K$ by changing all the classical crossings. For classical knots, the upper and lower presentations give the same group, but for virtual knots this is no longer true. The first example of this was provided in \cite{GPV}. 

\begin{example} \label{virtual trefoil}
If $K$ is the knot depicted in Figure \ref{2-1}, then it has knot group $G_K \cong \ZZ.$ This follows easily from the relation $b=bab^{-1}$ at the first crossing. Since $K$ is a nontrivial virtual knot, this example shows that the knot group $G_K$ does not detect the unknot among virtual knots. In the next section, we introduce the virtual knot group $\VG_K$ and show it is nontrivial for this knot.
\end{example}

\begin{figure}[ht]
\centering
\includegraphics[scale=1.10]{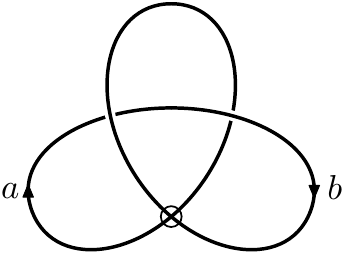}
\caption{A virtual knot with trivial knot group.}
\label{2-1}
\end{figure}

The knot group does not change under applying the forbidden move ($f1$), and thus $G_K$ is an invariant of the welded equivalence class of $K$.  That the knot in Example \ref{virtual trefoil} is welded trivial can be seen quite easily by considering its Gauss diagram. This can also be established by applying moves ($r1$)--($r3$), ($v1$)--($v4$) and ($f1$) to the virtual knot diagram in Figure \ref{2-1}, see \cite[Example 2.1]{Satoh}.  

 
\section{The virtual knot group} \label{section-2}

In this section we define the virtual knot group $\VG_K$ and show it is invariant under the generalized Reidemeister moves. We note that several authors have studied other group-valued invariants of virtual knots, and at the end of this section we present a commutative diamond relating $\VG_K$ to the extended group $\EG_K$ of Silver and Williams \cite{SW-Crowell} and the quandle knot group $\QG_K$ of Manturov \cite{M02} and Bardakov and Bellingeri \cite{BB}. Note that all the groups surject to $G_K$, and note further that although some authors refer to $G_K$ as the ``virtual knot group", we shall reserve that terminology exclusively for $\VG_K$ (cf. \cite{SW-VKG}).

Let $K$ be a virtual knot or link represented by a virtual knot diagram.
A \emph{short arc} is an arc that begins from one virtual or classical crossing  and terminates at the next one. In contrast to the classical knot group, short arcs terminate at under-crossings, over-crossings, and virtual crossings. The \emph{virtual knot group} $\VG_K$ has one generator for each short arc, along with two additional generators $s$ and $q$, and it has two relations for  each crossing, depending on the sign of the crossing as indicated in Figure \ref{Relations}, along with the commutation relation $[s,q]=1.$  

\begin{figure}[ht]
\centering
\includegraphics[scale=0.90]{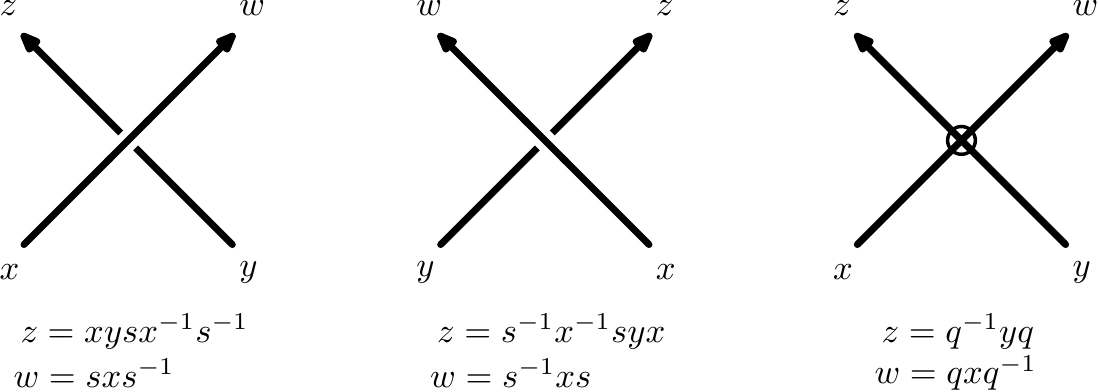}
\caption{}
\label{Relations}
\end{figure}

We will see that $\VG_K$ is independent under the generalized Reidemeister moves, but before doing that, we present a simple example illustrating how to compute $\VG_K$ from a diagram for $K$.

\begin{example} \label{virtual trefoil2}
We compute the virtual knot group for the knot depicted in Figure \ref{VirtualTrefoil}. Labeling the short arcs $a,b,c,d,e,f$ as in Figure \ref{VirtualTrefoil}, we obtain the following presentation.
\begin{eqnarray*}
\VG_K &=& 
\langle a,b,c,d,e,f,s,q \mid  [s,q]=1, b=e^{-1}ad, c=s^{-1}bs, d=q c q^{-1}, \\
&&  \hspace*{3.4cm} e=s^{-1}d s, f=c^{-1}eb, a=q^{-1}fq\rangle \\
&=& \langle a,b,s,q \mid   [s,q]=1,
a=q^{-1}s^{-1}b^{-1}qs^{-1}bs^2q^{-1}bq, \\
&&  \hspace*{2.0cm} b=qs^{-2}b^{-1}s^2 q^{-1} a qs^{-1}bsq^{-1}\rangle, \\
&=& \langle b,s,q \mid   [s,q]=1,
 b=qs^{-2}b^{-1}s q^{-2} b^{-1}qs^{-1}bs^2q^{-1}b q^2s^{-1}bsq^{-1}  \rangle.
\end{eqnarray*}
This group is nontrivial and can be used to show that $K$ is nontrivial as a virtual knot, see Example \ref{virtual trefoil3}.
\end{example}

\begin{figure}[ht]
\centering
\includegraphics[scale=1.10]{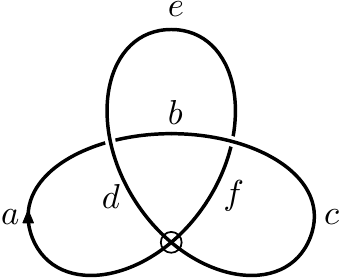}
\caption{The virtual trefoil with labellings of the short arcs.}
\label{VirtualTrefoil}
\end{figure} 

To verify that $\VG_K$ is an invariant of virtual knots, begin by considering the knot diagram of a virtual knot in a small neighborhood where the Reidemeister move is taking place.  Label arcs according to the definition of $\VG_K$, see Figure \ref{Relations}, and check to see that the knot diagram gives the same set of generators and relations before and after the move.  As an example, we verify in Figure \ref{InvarianceOfTheMixedMove} that the mixed Reidemeister move $(v4)$ holds.  The arcs are assumed to be oriented upwards in the figure, and note that the condition that $s$ and $q$ commute is necessary for invariance of ($v4$). However, it's worth mentioning that the moves ($r1$)--($r3$) and ($v1$)--($v3$) are in fact invariant without this condition.   

\begin{figure}[ht]
\centering
\includegraphics[scale=0.90]{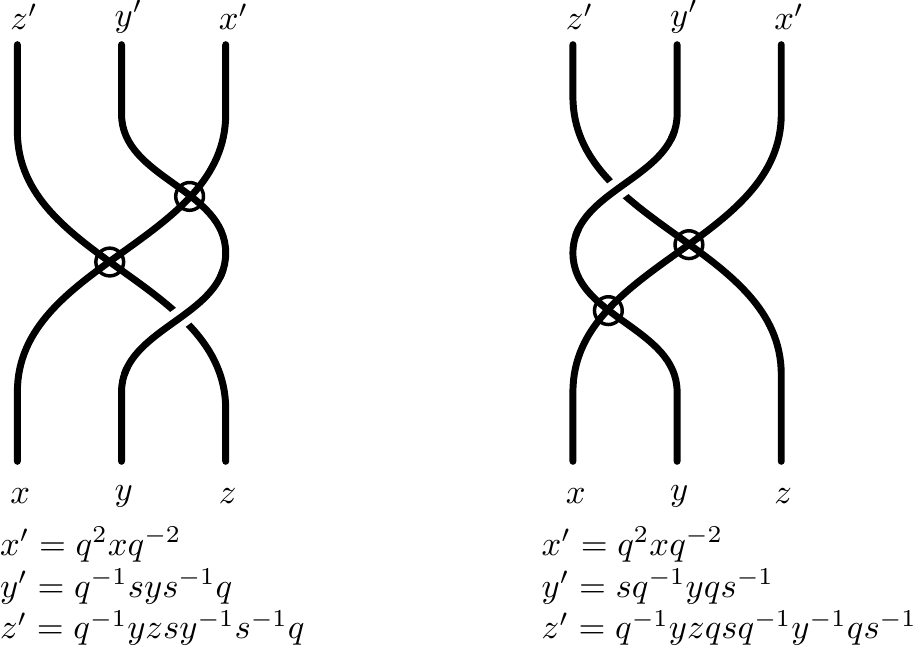}
\caption{Invariance of $\VG_K$ under ($v4$).}
\label{InvarianceOfTheMixedMove}
\end{figure} 

The virtual knot group $\VG_K$ has many interesting quotients.  By setting $s=1=q$, we recover the classical knot group $G_K$ as defined in the previous section.  One can also show that after setting $s=q$, the forbidden move ($f1$) holds. 
Thus the quotient of $\VG_K$ by the normal subgroup generated by $qs^{-1}$ is invariant under welded equivalence, and the \emph{welded knot group} is the quotient $$\WG_K=\VG_K / \langle\!\langle qs^{-1} \rangle\!\rangle,$$
where $\langle\!\langle \; \cdot \; \rangle\!\rangle$ denotes the normal closure.

\begin{example} 
The trivial knot $\bigcirc$   has virtual and welded knot groups given by
\begin{eqnarray*}
\VG_\bigcirc &=& \langle a,s,q \mid  [s,q]=1\rangle \cong \ZZ*\ZZ^2, \\
\WG_\bigcirc &=& \langle a,s \rangle \cong \ZZ*\ZZ.
\end{eqnarray*}
We say that $K$ has trivial virtual knot group if $\VG_K \cong \ZZ*\ZZ^2$, and that it has trivial welded knot group if $\WG_K \cong \ZZ* \ZZ$. There exist nontrivial virtual knots with trivial virtual knot group, for example the Kishino knot depicted in Figure \ref{Kishino}.
(In \cite{KS04}, Kishino and Satoh use the 3-strand bracket polynomial to show this knot is nontrivial.) Therefore $\VG_K$ does not detect the trivial knot among virtual knots. 
It can be shown that
$\WG_K \cong G_K* \ZZ$.  Hence $\WG_K$ and $G_K$ carry the same information about $K$.
It is an open problem whether $G_K$ (equivalently, $\WG_K$) detects the trivial knot among welded knots.
\end{example}

\begin{figure}[h]
\centering
\includegraphics[scale=1.60]{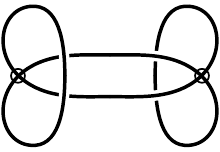}
\caption{The Kishino knot has trivial virtual knot group, trivial Jones polynomial, and trivial generalized Alexander polynomial.}
\label{Kishino}
\end{figure}

We now explain the relationships between the virtual knot group $\VG_K$ and various other groups that arise in virtual knot theory. For instance, under setting $q=1$ in $\VG_K$, one obtains the extended group $\EG_K$, which was introduced by Silver and Williams in \cite{SW-Crowell}, where it is denoted $\wt{\pi}_K$.  This group is related to the Alexander group, which Silver and Williams used to give a group-theoretic approach to the generalized Alexander polynomial, see \cite{SW-Alexander} and \cite{SW-03}. 

By setting $s=1$ in $\VG_K$, one obtains the quandle knot group $\QG_K$, which was introduced by Manturov \cite{M02} and by Bardakov and Bellingeri \cite{BB}.
Finally, as we have already mentioned, if one sets $q=s$ in $\VG_K$, one obtains the welded knot group $\WG_K$. 
Each of these groups has the classical knot group $G_K$ as a natural quotient, and the various relationships are summarized by the commutative diagram below.

Note that if $K$ is a classical knot or link, then none of the relations in the Wirtinger presentation for $\VG_K$ involve $q$. Thus, no information is lost in setting $q=1,$ and this shows that $\VG_K$ is determined by the extended group $\EG_K.$ As we shall see in the next section, the same is true at the level of Alexander invariants, namely that for a classical knot or link $K$, the virtual Alexander invariants are determined by the classical Alexander invariants. 

\begin{center}
\begin{tikzcd}[column sep=large] 
&\VG_K 
\arrow{ldd}{s=1}  \arrow{dd}{q=s}  \arrow{rdd}{q=1} \\ \\
\QG_K   \arrow{rdd}{q=1} & \WG_K \arrow{dd}{s=1} &\EG_K \arrow{ldd}{s=1}  \\ \\ 
&G_K
\end{tikzcd} \\
{A commutative diamond representing the natural surjections between the virtual, quandle, welded, extended, and classical knot groups.}
\end{center}

\section{The virtual Alexander invariant} \label{section-3}

In this section, we introduce the virtual Alexander invariants of a virtual knot $K$, and they are the Alexander invariants associated with the virtual knot group $\VG_K$. Our approach is motivated by the case of classical knots, where one obtains powerful  invariants of a knot $K$ by studying the elementary ideal theory of the knot group $G_K$. 
 
Given a virtual knot $K$, its virtual knot group $\VG_K$ has abelianization $\ZZ^3$, and we use $s,t,q$ to denote the three generators. Let $\VG_K'$ and $\VG_K''$ be first and second commutator subgroups of $\VG_K$,  and note that $\VG_K'$ is the kernel of $\VG_K \to \ZZ^3.$  The quotient $\VG_K'/ \VG_K''$ can be regarded as a module over the group-ring of $\ZZ^3,$ which is isomorphic to the ring of Laurent polynomials $\ZZ[s^{\pm 1}, t^{\pm 1}, q^{\pm 1}]$. We set $\cL=\ZZ[s^{\pm 1}, t^{\pm 1}, q^{\pm 1}]$ and we refer to the module associated to the abelianization $\VG_K \to \ZZ^3$
as the \emph{virtual Alexander module}.

One can extract a number of invariants of the virtual knot $K$ from its virtual Alexander module. For instance, if $M$ is any $m \times n$ presentation matrix for the virtual Alexander module,  then the $\ell$-th elementary ideals $\cE_\ell$ are defined as follows (see p.~101 of \cite{CF}). If $0 < n- \ell \leq m,$ we define $\cE_\ell$ to be the ideal of $\cL$ generated by  all $(n-\ell) \times (n-\ell)$ minors of $M$, otherwise we set $\cE_\ell  = 0$ if $n-\ell >m$ and $\cE_\ell = \cL$ if $ n-\ell \leq 0$. The elementary ideals are independent of the presentation matrix $M$, and they can be used to define invariants of the virtual knot $K$. However, the elementary ideals $\cE_\ell$ are not generally principal, and so we shall instead consider $\ell$-th principal elementary ideals $\cI_\ell$, which are defined to be the smallest principal ideals containing $\cE_\ell$. Thus $\cI_\ell$ is generated by the $\gcd$ of the $(n-\ell) \times (n-\ell)$ minors of $M$.
This is the basis for our definition of the virtual Alexander polynomial.

\begin{definition} Given a virtual knot or link $K$, let $\Delta^\ell_{K}$ be the Laurent polynomial given by the $\gcd$ of the $(n-\ell) \times (n-\ell)$ minors of $M$. It is a generator of the $\ell$-th principal elementary ideal $\cI_\ell,$
and it is well-defined up to units in $\cL$, namely up to multiplication by $\pm s^i t^j q^k$ for $i,j,k \in \ZZ.$ 
We call $\Delta^\ell_{K}$ the {\it $\ell$-th virtual Alexander polynomial of $K$}. The virtual Alexander polynomial $\Delta^0_{K}$  will often be denoted $H_K(s,t,q)$. 
 \end{definition}
 
We adopt notation to suppress the inherent indeterminacy of the virtual Alexander polynomial, namely we write $H_K(s,t,q) = f$ whenever $f \in \cL$ is a Laurent polynomial such that $H_K(s,t,q) = \pm s^i t^j q^k \cdot f$ for some $i,j,k \in \ZZ.$

\begin{remark} In Section \ref{section-5} we will show that there is a normalization of the virtual Alexander polynomial to make it well-defined up to multiplication by $(st)^i$. 
\end{remark}

\begin{example} \label{virtual trefoil3}
Using the presentation for $\VG_K$ from Example \ref{virtual trefoil2}, one can easily compute that the virtual Alexander polynomial for the virtual trefoil knot $K$ is given by
\begin{eqnarray*}
H_K(s,t,q) &=& s^{-1} + q^2s^{-2}t + qt^2 -q^2s^{-1}t^2 -t -qs^{-2}\\
&=&s +q^2t+qs^2t^2 -q^2st^2-s^2t-q\\
&=&(s-q)(1-st)(1-qt).\end{eqnarray*}
\end{example}

It is a basic fact that the elementary ideals are independent of the presentation matrix. Given any presentation 
\begin{equation}\label{present}
\VG_K = \langle a_1, \ldots a_n, s,q \mid  r_1, \ldots, r_n, [s,q]  \rangle,
\end{equation} 
one obtains a presentation matrix for the virtual Alexander module $\VG_K'/\VG_K''$ by taking Fox derivatives of the relations in \eqref{present} with respect to the generators and substituting $a_i=t$ for $i=1,\ldots, n.$ The resulting matrix is an $(n+2) \times (n+1)$ matrix of Laurent polynomials $\cL$ which we will denote by $M$. The $0$-th elementary ideal $\cE_0$ is the ideal generated by all $(n+1) \times (n+1)$ minors of $M$, and it is not necessarily principal. However, one can nevertheless compute the virtual Alexander polynomial entirely in terms of the virtual Alexander matrix $A$,  which is the $n \times n$ Jacobian matrix 
$$A=\left(\left.\frac{\partial r_i}{\partial a_j}\right|_{a_1, \ldots, a_n=t} \right)$$
of Fox derivatives of the relations $r_i$ with respect to the generators $a_j$. We take a moment to explain this rather delicate point here.

The virtual Alexander matrix $A$ appears in the upper left hand corner of the presentation matrix $M$. To be specific, let
$\left(\frac{\partial r^{}}{\partial s_{}}\right) $ and $\left(\frac{\partial r}{\partial q} \right)$ denote the column vectors of Fox derivates of the relations $r_1, \ldots, r_n$ with respect to $s$ and $q$, so  $\left(\frac{\partial r^{}}{\partial s_{}}\right) $ has $i$-th entry equal to $ \left. \frac{\partial r_i}{\partial s}\right|_{a_1, \ldots, a_n=t}$ and $\left(\frac{\partial r}{\partial q}\right) $ has $i$-th entry equal to $ \left. \frac{\partial r_i}{\partial s}\right|_{a_1, \ldots, a_n=t}$. Then
$$M = \begin{bmatrix} A & \left(\frac{\partial r^{}}{\partial s_{}} \right) & \left(\frac{\partial r}{\partial q}\right) \\ {\bf 0} & 1-q & s-1 \end{bmatrix},$$
where the last row is obtained by Fox differentiation of the commutator $[s,q].$

Computing the $(n+1) \times (n+1)$ minors of $M$ involves choosing one column to eliminate and taking the determinant. If either of the last two columns is chosen, one obtains minors equal to  $(s-1) \det A$ and $(1-q) \det A$, respectively. If any one of the other columns is chosen, we claim that the resulting minor equals $\pm (t-1) \det A.$

The {\it fundamental identities}  \cite[(2.3)]{Fox1953} for the Fox derivatives of the relators $r_i$
are equations in the integral group-ring of the free group $F_{n+2} = \langle a_1,\ldots, a_n, s, q\rangle$, namely
\begin{equation}\label{fundamental-Fox}
\sum^n_{j=1}  \frac{\partial r_i}{\partial a_j} (a_j - 1)  +  \frac{\partial r_i}{\partial s} (s - 1) +  \frac{\partial r_i}{\partial q} (q - 1) = r_i -1, \qquad \text{$i=1,\ldots, n$.}
\end{equation}

Let $A_j$ denote the $j$-th column of $A$. 
The map $a_i \mapsto t$, $s \mapsto s$ and $q \mapsto q$ extends to a ring
homomorphism  $\ZZ F_{n+2}  \rightarrow \cL$.  Applying this homomorphism
to \eqref{fundamental-Fox} yields the following linear relation among the columns of $M$.
\begin{equation}\label{fundamental-Fox-two}
\sum^n_{j=1} A_j(t-1) + \left( \frac{\partial r^{}}{\partial s_{}} \right)(s-1) + \left( \frac{\partial r}{\partial q} \right)(q-1) ={\bf 0}.
\end{equation}

Let $\wh{M}_k$ denote the result of removing the $k$-th column of $M$ and let $\wh{A}_k$ denote the result of removing the $k$-th column of $A$. Then
$$\wh{M}_k = \begin{bmatrix} \wh{A}_k & \left( \frac{\partial r^{}}{\partial s_{}} \right) & \left(\frac{\partial r}{\partial q}\right) \\ {\bf 0} & 1-q & s-1 \end{bmatrix}.$$
We calculate the determinant of $\wh{M}_k$ by expansion along its last row.
\begin{eqnarray*}
\det \wh{M}_k &=& -(1-q) \det\left[ \wh{A}_k  \,\,\, \left( \tfrac{\partial r}{\partial q}\right) \right] ~+~  (s-1) \det\left[ \wh{A}_k  \,\,\, \left( \tfrac{\partial r^{}}{\partial s_{}}\right) \right] \\
                      &=&   \det\left[ \wh{A}_k  \quad (s-1)\left( \tfrac{\partial r^{}}{\partial s_{}}\right)  + (q-1)\left( \tfrac{\partial r}{\partial q}\right) \right] \\
                      &=&   \det\left[ \wh{A}_k  \quad \sum^n_{j=1} (1-t) A_j  \right] \qquad  \text{by \eqref{fundamental-Fox-two}} \\
                      &=&  \sum^n_{j=1} (1-t) \det\left[ \wh{A}_k \,\,\, A_j  \right]
                      =  (-1)^{n-k+1}(t-1) \det A.
\end{eqnarray*}
This shows that $\cE_0$ is generated by $(s-1) \det A, (t-1) \det A,$ and $(q-1) \det A.$ Since  the ring $\cL=\ZZ[s^{\pm 1}, t^{\pm 1}, q^{\pm 1}]$ of Laurent polynomials is a unique factorization domain, it is also a $\gcd$ domain, and consequently the principal generator of $\cI_0$ is the $\gcd$ of the generators of $\cE_0,$ which in this case one can see directly equals $\det A.$ Thus it follows that $H_K(s,t,q) = \det A.$
 
A similar argument shows that the higher Alexander invariants $\Delta^\ell_{K}$ are given by the $\gcd$ of the $(n-\ell) \times (n-\ell)$ minors of $A$.

Since the virtual Alexander polynomial $H_K(s,t,q)$ is only well-defined up to multiplication by $\pm s^i t^j q^k, $ it is convenient to work with polynomials instead of Laurent polynomials, so we will often write $H_K(s,t,q)$ without negative powers of $s, t, q$ by multiplying through by appropriate powers of $s,t$ or $q$. We define the $q$-width of $H_K(s,t,q)$ to be the difference between the maximal $q$-degree and the minimal $q$-degree. 
 
The next result shows that the $q$-width of $H_K(s,t,q)$ gives a lower bound on the virtual crossing number $v(K)$ of $K$ which is defined in the introduction.
   
\begin{theorem} \label{v-bound}
If $K$ is a virtual knot or link and $H_K(s,t,q)$ is its virtual Alexander polynomial, then  
$$\text{$q$-$\width$} \; H_K(s,t,q) \leq 2 v(K).$$  
\end{theorem}
\begin{proof}
Suppose $D$ is a virtual knot diagram of $K$ with $v(D) = m$. Then the virtual Alexander matrix has $2m$ rows with a $q$ or $q^{-1}$ entry. It follows that $\det A$ has  $q$-width at most $2m.$
\end{proof}

We give some examples to illustrate the utility of the bound on $v(K)$ coming from the virtual Alexander polynomial.
In section \ref{section-7}, we will show how to improve these bounds using the twisted Alexander polynomial.
\begin{example}
Let $K$ be the virtual knot depicted in Figure \ref{3-7}. An easy computation shows that 
$$H_K(s,t,q) =  
(q^2 - s^2) (t^2 q^2 - 1) (st - 1).$$
Since $H_K(s,t,q)$ has $q$-$\width =4$, it follows that $v(K) \geq 2.$ Comparing to Figure \ref{3-7}, we conclude that $K$ has virtual crossing number 2.
\end{example}

\begin{figure}[ht]
\centering
\includegraphics[scale=1.80]{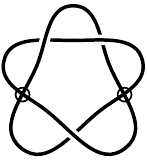}
\caption{The virtual knot $3.7$ has $v(K)=2$.}
\label{3-7}
\end{figure}

\begin{example}
Let $K$ be the virtual knot depicted in Figure \ref{4-62}. A straightforward computation shows that 
$$H_K(s,t,q) =   
-(q^3 t + q^2 st + q^2 + 3qs + s^2)(tq - 1)(st - 1)(q - s).$$
Since $H_K(s,t,q)$ has $q$-$\width=5$, it follows that $v(K) \geq 3.$ Comparing to Figure \ref{4-62}, we conclude that $K$ has virtual crossing number 3.
\end{example}

\begin{figure}[ht]
\centering
\includegraphics[scale=1.00]{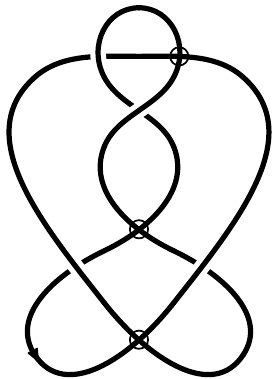}
\caption{The virtual knot $4.62$ has $v(K)=3$.}
\label{4-62}
\end{figure}

The arrow polynomial of Dye and Kauffman \cite{DK} also gives an effective lower bound on the virtual crossing number $v(K)$, see for instance \cite{BDK}. The following examples shows that the lower bound on $v(K)$ from $H_K(s,t,q)$ sometimes gives an improvement over that obtained from the arrow polynomial.

\begin{example}
According to \cite{BDK}, there are four virtual knots with 4 crossings having trivial arrow polynomial. They are the knots 4.46, 4.72, 4.98, and 4.107. Both 4.72 and 4.98 have $H_K(s,t,q)=0$, but for the other two virtual knots, we use the virtual Alexander polynomial to get a lower bound on $v(K)$. 

For $K = 4.46,$ one computes that
$$H_K(s,t,q) = (q - s) (tq - 1) (st - 1)^2$$
and  has $q$-$\width  = 2.$ It follows that $v(K) \geq 1.$ Thus $v(K)=1$ or 2.

For $K = 4.107,$ one computes that
$$H_K(s,t,q) = (q^2 - s^2) (t^2 q^2 - 1) (st - 1)^2$$
and has $q$-$\width = 4.$ It follows that $v(K) \geq 2.$ Thus $v(K)=2$ or 3.
\end{example}

\begin{figure}[ht]
\centering
\includegraphics[scale=1.60]{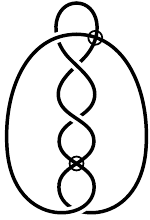} \hspace{2cm}
\includegraphics[scale=1.60]{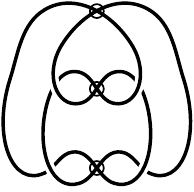}
\caption{The virtual knots $4.46$ (left) and $4.107$ (right).}
\label{4-46and4-107}
\end{figure}

Next, we relate the virtual Alexander polynomial to the generalized Alexander polynomial defined by Sawollek \cite{Sawollek}, Kauffman--Radford \cite{KR}, and Silver--Williams \cite{SW-Alexander, SW-03}.
Consider the presentation of $\VG_K$ obtained from the virtual knot diagram of $K$. This presentation has one meridional generator for each short arc and two relations for each crossing. The Alexander matrix one obtains from this presentation coincides, after setting $q=1$, with the presentation matrix one gets from the Alexander biquandle. To see this, take the Fox derivatives of the relations coming from the right-handed crossing in Figure \ref{Relations}.  Let $z=xysx^{-1}s^{-1}$ and
$w = sxs^{-1}$, then 
$$dw = \left(\frac{\partial w }{\partial x}\bigg|_{x,y=t}\right) x  + \left(\frac{\partial w}{\partial y}\bigg |_{x,y=t}\right)   y =\,\,\,  sx$$ 
and
$$dz = \left(\frac{\partial z }{\partial x}\bigg|_{x,y=t}\right) x  + \left(\frac{\partial z}{\partial y}\bigg |_{x,y=t}\right)   y =\,\,\,  (1-st)x+ty.$$
Similar formulas hold for left-handed crossings, and the resulting equations are identical to the ones coming from the Alexander biquandle, see   \cite[Definition 2.4]{CHN}.

The generalized Alexander polynomial $G_K(s,t)$ of a virtual knot or link is defined as the determinant of the  presentation matrix of the Alexander biquandle, and together with the above observations, this immediately implies the following result.

\begin{proposition} \label{HdeterminesG}
For any virtual knot or link K,  the generalized Alexander polynomial is determined by the
virtual Alexander polynomial via the formula
$$G_K(s,t) = H_K(s,t,1).$$
\end{proposition}

 If $K$ is a classical knot or link, then it follows that the virtual Alexander polynomial vanishes. To see this, suppose $K$ is a virtual knot with nonzero virtual Alexander polynomial. Then Proposition \ref{H_K-divisibility} implies that $H_K(s,t,q)$ has nonzero $q$-width, and Theorem \ref{v-bound} applies to show that $v(K)\geq 1.$ Thus $K$ is not classical.  
 
For virtual knots or links $K$ with trivial virtual Alexander polynomial, one can often use the virtual Alexander invariant to conclude that $K$ is not classical (cf. \cite[Corollary 4.8]{SW-Crowell}).
The idea is based on the observation that the virtual Alexander invariant of a classical
knot is determined from the classical Alexander invariant under replacing $t$ by $st$. 
Thus, if $K$ is classical, then its $k=1$ elementary ideal $\cE_1$ must be principal with a generator that is symmetric under $s \leftrightarrow s^{-1}$ and $t \leftrightarrow t^{-1}$. 

The next example illustrates this point for the virtual knots $4.99$ and $4.105$ depicted in Figure \ref{almostclassical}. 

\begin{example} \label{ex-ac}
Using the short arcs labelled $a$ and $b$ in Figure \ref{almostclassical}, we determine labels for all the other short arcs, and consequently we see can deduce presentations for the virtual knot groups $\VG_K$ with generators $a,b,s,q.$

For $K_1 = 4.99,$ we find that
$$ \VG_{K_1} = \langle a,b,s, q \mid  [s,q], \; 
a b s a^{-1} s^{-1} b  s a^{-1} s^{-2}b^{-1} s  \rangle.
$$
Notice that the variable $q$ does not appear in the last two relations. Taking the Fox derivatives, one can show that this knot has trivial virtual Alexander polynomial and that its $k=1$ elementary ideal is given by $\cE_1=(1-2st)$. Although $\cE_1$ is principal, the generator is not symmetric, and  it follows that $4.99$ is not classical. 

We claim that the virtual knot group splits as a free product, $\VG_{K_1} \cong G_{K_1} * \ZZ^2$. To see this, cyclically permute the relations and make the substitutions $x=sa, y=bs$. The new presentation of $\VG_{K_1}$ is given by 
\begin{eqnarray*}  \VG_{K_1} &=& \langle x,y,s, q \mid  [s,q], \; 
x y x^{-1} y  x^{-1} y^{-1}  \rangle \\
&\cong&  \langle x,y \mid   
x y x^{-1} y  x^{-1} y^{-1}  \rangle *  \langle s, q \mid  [s,q]  \rangle.
\end{eqnarray*}
It is a simple exercise to check the first presentation yields the knot group $G_{K_1}$.

For $K_2 = 4.105,$ we find that
\begin{eqnarray*}  \VG_{K_2} = \langle a,b,s, q & \mid & [s,q], \; a^{-1}s^{-2} b^{-1}a^{-1} s^{-1} b s^2 a b s a^{-1} s^{-2} b^{-1} s a b s, \\   
&& \hspace{1.2cm} s^{-1} b^{-1}a^{-1} s^{-1} b s a^{-1} s^{-2} b^{-1} s a b s^2 a s^{-1} b^{-1} s a  \rangle.
\end{eqnarray*}
Notice that again the variable $q$ does not appear in the last two relations. Taking the Fox derivatives, one can show that this knot has trivial virtual Alexander polynomial and that its $k=1$ elementary ideal is given by $\cE_1=( 2-2st+s^2t^2)$. Although $\cE_1$ is principal, the generator is not symmetric, and it follows that $4.105$ is not classical. 

We claim that the virtual knot group splits as a free product, $\VG_{K_2} \cong G_{K_2} * \ZZ^2$. To see this, substitute $x=sa, y=bs$. The new presentation of $\VG_{K_2}$ is given by 
\begin{eqnarray*}  \VG_{K_2} &=& \langle x,y,s, q  \mid   [s,q], \; x^{-1} y^{-1}x^{-1} y x y x^{-1} y^{-1} x y, \; y^{-1}x^{-1} y  x^{-1} y^{-1}  x y x y^{-1}  x  \rangle \\
 &\cong&  \langle x,y \mid   x^{-1} y^{-1}x^{-1} y x y x^{-1} y^{-1} x y, \; y^{-1}x^{-1} y  x^{-1} y^{-1}  x y x y^{-1}  x    \rangle *  \langle s, q \mid  [s,q]  \rangle.
\end{eqnarray*}
We leave it as an exercise to check that the first presentation yields the knot group $G_{K_2}$.
\end{example}

\begin{figure}[ht]
\centering
\includegraphics[scale=1.40]{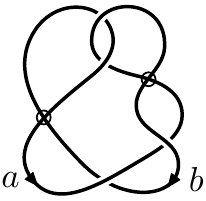} \hspace{1.6cm}
\includegraphics[scale=1.40]{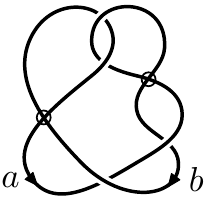}
\caption{The virtual knots $4.99$  (left) and $4.105$ (right).}
\label{almostclassical}
\end{figure}

\begin{remark}
Both virtual knots in the previous example are \emph{almost classical}.
This means that they possess diagrams that admit {\it Alexander numberings} (see \cite[Definition 4.3]{SW-Crowell}).
Alexander proved that every classical knot admits an Alexander numbering, thus every classical knot is almost classical.
Furthermore, since a classical knot can be realized by a knot diagram without virtual crossings, it is evident that one can eliminate the variable $q$ from the Wirtinger relations of $\VG_K$. 
If the virtual knot $K$ is almost classical, then by an Alexander numbering argument, it can be shown that $\VG_K \cong G_K \ast \ZZ^2$.
\end{remark} 
 

\section{The virtual braid group and the virtual Burau representation} \label{section-4}
In this section, we introduce the virtual braid group and its Burau representation and show that they are closely related to the Alexander polynomial of a virtual knot or link obtained as a braid closure. 

The {\it virtual braid group on $k$ strands}, denoted $\VB_{k}$,
is the group is generated by symbols $\si_{1}, \ldots, \si_{k-1}, \tau_1,\ldots, \tau_{k-1}$ subject to the relations given below in \eqref{classical-rel}, \eqref{virtual-rel}, \eqref{mixed-rel}. Here, $\si_{i}$ represents a classical crossing  and $\tau_i$ represents a virtual crossing involving the $i$-th and $(i+1)$-th strands as in Figure \ref{generators}. Virtual braids are drawn from top to bottom, and the group operation is given by stacking the diagrams, and the closure of a virtual braid represents a virtual link.

\begin{figure}[ht]
\centering
\includegraphics[scale=1.00]{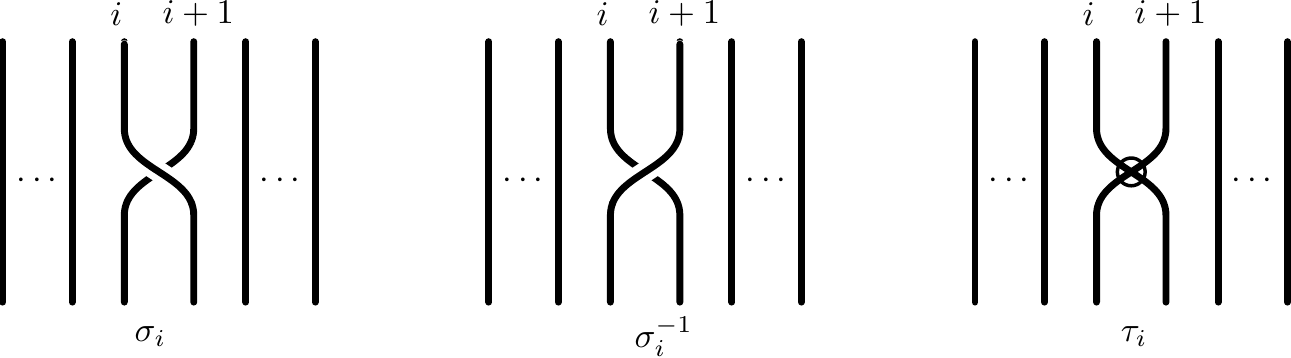} 
\caption{Generators of $\VB_{k}$.}
\label{generators}
\end{figure}  

The relations in $\VB_{k}$ are given by three families of relations, the first involve only classical crossings $\si_i$ and are the same as in classical braid group $B_k$, the second involve only virtual crossings $\tau_i$ and are the same as in the permutation group $S_k$, and the third are the mixed relations and involve both classical and virtual crossings:
\begin{equation} \label{classical-rel}
\begin{array}{rcl}
\si_{i}\si_{j}&=&\si_{j}\si_{i} \hspace{1.5cm} \text{ if $|i-j|>1$,} \\
\si_{i}\si_{i+1}\si_{i}&=&\si_{i+1}\si_{i}\si_{i+1},
\end{array}
\end{equation}
\begin{equation} \label{virtual-rel}
\begin{array}{rcl}
\tau_{i}\tau_{j}&=&\tau_{j}\tau_{i} \hspace{1.5cm} \text{ if $|i-j|>1$,} \\
\tau_{i}\tau_{i+1}\tau_{i} &=& \tau_{i+1}\tau_{i}\tau_{i+1}, \\
\tau_{i}^{2}&=&1,
\end{array}
\end{equation} 
\begin{equation}\label{mixed-rel}
\begin{array}{rcl}
\si_{i}\tau_{j}&=&\tau_{j}\si_{i} \hspace{1.5cm} \text{ if $|i-j|>1$,} \\
\tau_i \si_{i+1}\tau_i & = & \tau_{i+1} \si_i \tau_{i+1}.
\end{array}
\end{equation}

In Section \ref{section-5}, we will use Kamada's generalization of the Markov theorem to give a natural normalization for the virtual Alexander polynomial $H_K(s,t,q)$. This will be achieved by interpreting the virtual Alexander invariant in terms of the {\it virtual Burau representation}, introduced below in terms of a more general construction of the Magnus nonabelian 1-cocycle.

\subsection*{The fundamental representation of $\VB_k$}
Let $F_{k+2} = \langle x_1, \ldots, x_k, s, q \rangle$ and let \linebreak
$\Aut(F_{k+2} \, \rel \, \{s,q\})$ be the subgroup of $\Aut(F_{k+2})$ consisting of those automorphisms of $F_{k+2}$ that fix $s$ and $q$.
We view $\th \in  \Aut(F_{k+2})$ as acting on the right on $F_{k+2}$ so for  $\th_1, \th_2 \in  \Aut(F_{k+2})$ the product $\th_1 \th_2$ is left to right
composition, that is,   $(w)(\th_1 \th_2) = \left( (w) \th_1\right) \th_2$.

For $i=1,\ldots, k-1$, define automorphisms $\Phi(\si_i), \Phi(\si_i^{-1})$ and $\Phi(\tau_i)$ of $F_{k+2}$  fixing $s$ and $q$ as follows.
For $j=1,\ldots, k$, 
\begin{equation} \label{eq-fundrep}
\begin{split}
 (x_j) \Phi(\si_i) =&  \begin{cases} 
      s x_{i+1} s^{-1} &  \text{if $j=i$,} \\
      x_{i+1}x_{i} (s x_{i+1}^{-1}s^{-1}) & \text{if $j=i+1$,}\\
      x_{j} & \text{otherwise},
   \end{cases} \\
  (x_j) \Phi(\si_i^{-1}) =&  \begin{cases} 
      (s^{-1} x_{i}^{-1} s) x_{i+1} x_{i} &  \text{if $j=i$,} \\
      s^{-1} x_{i} s & \text{if $j=i+1$,}\\
      x_{j} & \text{otherwise},
   \end{cases} \\
 (x_j) \Phi(\tau_i) =&  \begin{cases} 
      q x_{i+1} q^{-1} &  \text{if $j=i$,} \\
      q^{-1}x_{i}q &  \text{if $j=i+1$,} \ \\
      x_{j} & \text{otherwise}.
   \end{cases}
\end{split}
\end{equation}
It is straightforward to show that this assignment of generators of $\VB_k$ to automorphisms of $F_{k+2}$
respects the relations \eqref{classical-rel}, \eqref{virtual-rel}, and \eqref{mixed-rel}. 
Hence $\Phi$ extends to a representation 
$\Phi \colon \VB_{k} \to \Aut(F_{k+2} \, \rel \, \{s,q\})$
that we call the {\it fundamental representation of the virtual braid group}.
We will often suppress $\Phi$ and simply write $x_i^\be$ instead of $(x_i)\Phi(\be)$.

Notice that under the fundamental representation, a virtual braid $\be \in \VB_k$ acts on $x_1, \ldots, x_k$ as depicted in Figure
\ref{Braid-Action}

\begin{figure}[ht]
\centering
 \includegraphics[scale=1.20]{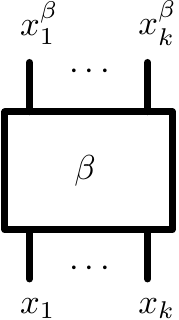}
\caption{The braid automorphism $\be$.}
\label{Braid-Action}
\end{figure}
Notice further that if $K= \wh{\be}$ is the closure of $\be \in \VB_k$, then the virtual knot group $\VG_K$ admits the presentation 
\begin{equation} \label{present2}
 \langle x_1, \ldots, x_k, s,q \mid  x_1=x_1^\be, \ldots, x_k=x_k^\be, [s,q]=1 \rangle.
\end{equation} 

For two virtual braids $\al, \be \in \VB_k$, the product $\al \be$ acts on $x_1, \ldots, x_k$  as depicted in Figure
\ref{Braid-Composition}. 
 
\begin{figure}[hb]
\centering
\includegraphics[scale=1.20]{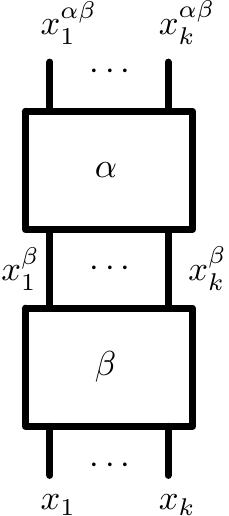}
\caption{The composition $\al \be$ of two braid automorphisms.}
\label{Braid-Composition}
\end{figure}

\subsection*{The Magnus non-abelian 1-cocycle}
Given a matrix $A $ over $\ZZ F_{k+2}$ and $\theta \in \Aut(F_{k+2})$, let $A^\theta$ be  the matrix obtained by applying $\theta$ to each entry of $A$
(note that an automorphism of $F_{k+2}$ extends to an automorphism of $\ZZ F_{k+2}$ denoted by the same symbol).

For $\theta_1, \theta_2 \in \Aut(F_{k+2} \, \rel \, \{s,q\})$,
the Chain Rule for Fox derivatives \cite[(2.6)]{Fox1953} yields
\begin{equation}
\label{chainrule}
\frac{\partial ~}{\partial x_{j}}  \, (x_{i}^{\theta_1})^{\theta_2} =
 \sum_{\ell=1}^k  \left( \frac{\partial (x_{i}^{\theta_1})}{\partial x_{\ell}} \right)^{\theta_2} \,\, \frac{\partial (x_{\ell}^{\theta_2})}{\partial x_{j}} 
\end{equation} 
(note that   $\frac{\partial (s^{\theta_2})}{\partial x_{j}} = \frac{\partial s }{\partial x_{j}}   = 0$ and $\frac{\partial (q^{\theta_2})}{\partial x_{j}} = \frac{\partial q}{\partial x_{j}}   = 0$).

For $\theta \in  \Aut(F_{k+2} \, \rel \, \{s,q\})$ define the {\it Fox Jacobian} of $\theta$ with respect to the variables $x_1, \ldots, x_k$ to be the $k \times k$ matrix over $\ZZ F_{k+2}$ given by
\[
D(\theta)  = \left( \frac{\partial (x_{i}^{\theta})}{\partial x_{j}} \right).
\]
It follows from \eqref{chainrule} that
\begin{equation}
\label{cocycleproperty}
D(\theta_1 \theta_2) = {D(\theta_1)}^{\theta_2} \,\, D(\theta_2).
\end{equation} 
(Recall that  $\theta_1 \theta_2$ denotes left to right composition.)
Applying  \eqref{cocycleproperty}  in the case \linebreak
$\theta_1 = \theta_2^{-1}$ shows
that $D(\theta)$ is invertible for any $\theta \in  \Aut(F_{k+2} \, \rel \, \{s,q\})$.
Hence the map $D \colon  \Aut(F_{k+2} \, \rel \, \{s,q\}) \to GL_k(\ZZ F_{k+2})$ is a (right) non-abelian $1$-cocycle
that we call the {\it Magnus non-abelian $1$-cocycle}.

The following properties of the Magnus non-abelian $1$-cocycle are straightforward consequences of \eqref{cocycleproperty}.

\begin{proposition}
\label{cocycle-properties}
\begin{enumerate}[\hspace{12pt}]
\item[(i)]    ${D(\be^{-1})}^\be  = {D(\be)}^{-1}$ for any $\be \in  \Aut(F_{k+2} \, \rel \, \{s,q\}).$
\item[(ii)]   ${D(\ga \be \ga^{-1})}^\ga = {D(\ga)}^\be \,\, D(\be) \, {D(\ga)}^{-1}$ for all $\be,\ga \in  \Aut(F_{k+2} \, \rel \, \{s,q\}).$
\end{enumerate}
\end{proposition}

\subsection*{The virtual Burau representation}
Let $\cL = \ZZ[s^{\pm 1}, t^{\pm 1}, q^{\pm1}]$ and
let  \hbox{$\al \colon \ZZ F_{k+2} \to \cL$} be the ring homomorphism determined by $\al(x_i) =t$, $\al(s)=s$ and $\al(q)=q$.
Let $\al_* \colon GL_k(\ZZ F_{k+2} ) \to GL_k( \cL)$ be the change of rings homomorphism, that is,
$\al_*\left( (a_{ij})  \right) = \left( \al(a_{ij}) \right)$ for $ (a_{ij})  \in GL_k(\ZZ F_{k+2} )$.

\begin{definition}
\label{def_virtual_Burau_rep}
The {\it virtual Burau representation} is the map $\Psi \colon \VB_k \to GL_k(\cL)$ given by
$
\Psi(\be) =   \al_*(D(\Phi(\be)))  
$
for $\be \in \VB_k$.
\end{definition}

The map $\Psi \colon \VB_k \to GL_k(\cL)$ is a homomorphism:
For $\be, \ga \in \VB_k$,
\begin{eqnarray*}
\Psi(\be \ga)  &=&  \al_*(D(\Phi(\be \ga))) =   \al_*\left(D(  \Phi(\be) \Phi(\ga) ) \right)  \\
                                &=&  \al_*\left( D(  \Phi(\be) )^{\Phi(\ga)}  \,     D(  \Phi(\ga) )         \right)   \qquad \qquad \text{ by  \eqref{cocycleproperty}}\\
                                &=&  \al_*\left( D(  \Phi(\be) )^{\Phi(\ga)} \right) \, \al_*\left(D(  \Phi(\ga) ) \right)  
\end{eqnarray*}
For any  $w \in F_{k+2}$, $\al(w^\ga) =  \al(w)$ and thus
$\al_*\left( D(  \Phi(\be) )^{\Phi(\ga)}) \right)  = \al_*\left( D(  \Phi(\be) ) \right)$.
It follows that $\Psi(\be \ga) = \Psi(\be) \, \Psi(\ga)$.

By direct calculation, the images of the generators $\si_i, \si_i^{-1},$ and $\tau_i$ under the virtual Burau representation $\Psi$ are given by the matrices
\begin{equation} \label{Burau-matrix}
\begin{split}
\Psi(\si_{i}) &=
\begin{bmatrix}
I_{i-1} &  & \cdots& {\bf 0}\\
  & 0 & s & \vdots  \\
 \vdots& t  & 1-st &    \\
{\bf 0} & \cdots &  & I_{k-i-1} \\
\end{bmatrix}, \\[1em] 
\Psi(\si_{i}^{-1}) &=
\begin{bmatrix}
I_{i-1} & & \cdots  & {\bf 0} \\
 &  1-s^{-1}t^{-1} & t^{-1} & \vdots\\
\vdots & s^{-1} & 0  &  \\
{\bf 0} & \cdots &  & I_{k-i-1} \\
\end{bmatrix},  \\[1em]
\Psi(\tau_{i}) &=
\begin{bmatrix}
I_{i-1} &  &  \cdots & {\bf 0}\\
 & 0 & q & \vdots\\
\vdots & q^{-1} & 0  & \\
{\bf 0} & \cdots &  & I_{k-i-1} \\
\end{bmatrix}
\end{split}
\end{equation}
where  $I_m$ is the $m\times m$ identity matrix. 

\subsection*{The twisted virtual Burau representation}
Let $R$ be a commutative ring and let
 $\cL_R = R[s^{\pm 1}, t^{\pm 1}, q^{\pm 1}]$ be  the ring of Laurent polynomials over $R$.
Let  $\al \colon \ZZ F_{k+2} \to \cL_R$ be the ring homomorphism determined by $\al(x_i) =t$, $\al(s)=s$ and $\al(q)=q$.
Given any representation $\varrho  \colon F_{k+2}  \to GL_n(R)$,
let ${\wt \varrho} = \varrho \otimes \al \colon \ZZ F_{k+2}  \to GL_n(\cL_R)$. 
For any $k \times k$ matrix  $A=(a_{ij})$ over $\ZZ F_{k+2}$,
let ${\wt \varrho}_*(A)$ be the block matrix  obtained by replacing $a_{ij}$ by ${\wt \varrho}(a_{ij})$. 
We view ${\wt \varrho}_*(A)$ as a $nk \times nk$ matrix over $\cL_R$.  
This construction yields a homomorphism ${\wt \varrho}_* \colon GL_k(\ZZ F_{k+2}) \to GL_{nk}(\cL_R)$.

\begin{definition}
\label{def_twisted_virtual_Burau_rep}
The {\it $\varrho$-twisted virtual Burau representation} is the map
\[
\Psi_{\varrho} \colon \VB_k \lto GL_{nk}(\cL_R)
\]
given by
$
\Psi_{\varrho}(\be) =    {\wt \varrho}_{*}\left(  D( \Phi(\be) )   \right)  
$
for $\be \in \VB_k$.
\end{definition}

In general, $\Psi_{\varrho}$ is not a homomorphism of groups but rather has property \eqref{psipropertyone} below.
In the special case $\varrho$ is the trivial representation,
we recover the (untwisted) virtual Burau representation $\Psi$.

If $\theta$ is an automorphism  of $F_{k+2}$ then the composite
$F_{k+2} \overset{\theta}\to  F_{k+2} \overset{\varrho}\to GL_n(R)$  is also a representation
that we denote by $\theta_*\varrho$.
In the case of an automorphism of the form $\Phi(\be)$, where $\be\in \VB_k$, we abbreviate $\Phi(\be)_*\varrho$ by $\be_* \varrho$.
Note that ${\wt {\be_*\varrho}} = \be_*{\wt \varrho}$.
Since we view automorphisms of $F_{k+2}$ as acting on the right, 
for $\be, \, \ga \in \VB_k$ our convention gives  $(\be \ga)_*\varrho = \be_*(\ga_* \varrho)$.

It follows from \eqref{cocycleproperty} that the map $\Psi_{\varrho}$ has the property
\begin{equation}
\label{psipropertyone}
\Psi_{\varrho}(\be \ga)  = \Psi_{\ga_*\varrho}(\be )  \, \Psi_{\varrho}(\ga). 
\end{equation}

Iterating \eqref{psipropertyone} yields
\begin{equation}
\label{psipropertyonemultiple}
\Psi_{\varrho} \left(a_1 a_2 \cdots a_m \right) =
\Psi_{(a_2 \cdots a_m)_*{\varrho} } \left(a_1\right) \,
\Psi_{(a_3\cdots a_m)_*{\varrho} } \left(a_2\right)  \,
\cdots
\Psi_{\varrho} \left(a_m\right). 
\end{equation}

As a consequence of Proposition \ref{cocycle-properties}(ii),  for $\be, \, \ga \in \VB_k$ we have
\begin{equation}
\label{psipropertytwo}
\Psi_{\ga_*\varrho}(\ga \be \ga^{-1})  =  \Psi_{\be_*\varrho}(\ga)  \, \Psi_{\varrho}(\be) \, \Psi_{\varrho}(\ga)^{-1}.
\end{equation}

\subsection*{The virtual Alexander invariants via braids.}
We now explain the relationship between the virtual Burau representation of $\be \in \VB_k$ and the Alexander polynomial $H_K(s,t,q)$ of the closure $K =\wh{\be}$. We will use this relationship to derive some interesting properties of $H_K(s,t,q)$.

\begin{theorem} \label{Burau}
Suppose $\be \in \VB_{k}$ is a virtual braid and $K= \wh{\be}$ is the virtual knot or link obtained from its closure. Then the $\ell$-th virtual Alexander polynomial $\Delta^\ell_{K}$ 
is given by the $\gcd$ of the $(k-\ell)\times (k-\ell)$ minors of the matrix $\Psi(\be)-I_k.$
In particular, the virtual Alexander polynomial is determined by the virtual Burau representation via the formula
$$H_K(s,t,q) = \det\left(\Psi(\be)-I_k\right).$$
\end{theorem}

\begin{proof} Label the braid strands $x_{1} \ldots, x_{k}$ and observe that labelling of braid strands according to the automorphisms given by $\si_{i}, \si_{i}^{-1}$, and $\tau_{i}$ is consistent with labelling according to the relations of $\VG_{K}$. As we use right automorphisms, the braid $\be_{1}\be_{2}$ corresponds to the left-to-right composition of automorphisms. The result is obtained by taking Fox derivatives of the relations in the presentation for $\VG_K$ in \eqref{present2}.
\end{proof}

The next result, first proved by Bartholomew and Fenn \cite[Theorem 7.1]{BF} (the case of the ``Alexander switch''),
will be used to show that the generalized Alexander polynomial determines the virtual Alexander polynomial.

\begin{theorem}
\label{theorem-H-K}
If $K$ is a virtual knot or  link,  then  the virtual Alexander polynomial satisfies 
$$ 
H_K(s,t,q) = H_K(sq^{-1}, tq, 1).
$$

\end{theorem}

\begin{remark} \label{rem-identity}
Using the above identity and  writing $H_K(s,t,q) = \sum_{ijk} a_{ijk} s^i t^j q^k$, it follows that $a_{ijk}=0$ unless $i+k=j$. Thus, it follows that the virtual Alexander polynomial also satisfies
$$ 
H_K(s,t,q) = H_K(1, st, qs^{-1}) = H_K(st, 1, qt).
$$
\end{remark}
\begin{proof}
Let $\be \in \VB_k$ be a virtual braid such that $K = \wh \be$.  
Write $\be$ as a product of generators $\be = \th_1 \cdots \th_m$, where $\th_j \in \{ \si_i^{\pm 1},  \tau_i  \mid   1\leq i \leq k-1\}$.
By Theorem \ref{Burau},  $H_K(s,t,q) = \det\left(\Psi(\be)-I_k\right)= \det\left(\Psi(\th_1) \cdots \Psi(\th_m)-I_k\right)$.
Let $\CC^*$ denote the multiplicative group of non-zero complex numbers.
A $k \times k$ matrix $A$ over $\ZZ[s^{\pm 1}, t^{\pm 1}, q^{\pm 1}]$ can be regarded as
function from $(\CC^*)^3$ to $k \times k$ matrices over $\CC$ via evaluation at the variables $s, t, q$
and, when viewed as such, we write  $A(s,t,q)$. 
Consider the diagonal  $k \times k$ matrix
$$\La= \begin{bmatrix} 1 &&& 0 \\ & q  \\ && \ddots \\ 0 &&& q^{k-1} \end{bmatrix}.$$
A straightforward calculation yields:
\begin{eqnarray*}
\left(\La\Psi(\si_i^{\pm 1})\La^{-1} \right)(s,t,q) &=& \Psi(\si_i^{\pm 1})(sq^{-1}, tq, 1)  \quad \text{ and } \\
\left(\La\Psi(\tau_i)\La^{-1} \right)(s,t,q) &=&  \Psi(\tau_i)(1, 1, 1) =  \Psi(\tau_i)(sq^{-1}, tq, 1).
\end{eqnarray*}
For the last equality, note that $\Psi(\tau_i)(s, t, q)$ is independent of $s$ and $t$.
Hence
\begin{eqnarray*}
H_K(s,t,q) &=& \det\left(\Psi(\th_1) \cdots \Psi(\th_m)-I_n\right) = \det\left(\La\left(\Psi(\th_1) \cdots \Psi(\th_m)-I_k\right) \La^{-1} \right)  \\
                 &=& \det\left(\La\Psi(\th_1)\La^{-1} \cdots \La\Psi(\th_m)\La^{-1} - I_k  \right)  = H_K(sq^{-1}, tq, 1).
\end{eqnarray*}
 \end{proof}

\begin{theorem}
\label{theorem-H-K-higher}
If $K$ is a virtual knot or  link, then  the higher virtual Alexander invariants satisfy $\Delta^{\ell}_K(s,t,q) = \Delta^{\ell}_K(sq^{-1}, tq, 1)$
for all $0\leq \ell < k$.
\end{theorem}

By the same reasoning as in Remark \ref{rem-identity}, it follows that the higher virtual Alexander invariants also satisfy $\Delta^{\ell}_K(s,t,q) = \Delta^{\ell}_K(1, ts, qs^{-1}) = \Delta^{\ell}_K(st, 1, qt)$ for all $0\leq \ell < k$.

\begin{proof}
We make use of the proof of Theorem \ref{theorem-H-K} and the notation used there.
Let $A = \Psi(\be)-I_k$.
Recall that
 $\left( \La\, A\, \La^{-1} \right)(s,t,q) = A(sq^{-1}, tq, 1)$.
For an $k \times k$ matrix $B$ and an integer $0\leq \ell < k$,
let ${\cE}_{\ell}(B)$ denote the ideal generated by the $(k-\ell)\times (k-\ell)$ minors of $B$, and let $\cI_\ell(B)$ denote the smallest principal ideal containing ${\cE}_{\ell}(B)$.

By  \cite[{Theorem 2, \S 1.4} ]{Northcott},    
${\cE}_{\ell}(A) = {\cE}_{\ell}\left( \La\, A\, \La^{-1} \right)$  and
hence ${\cE}_{\ell}\left(A\right) = {\cE}_{\ell}\left(A(sq^{-1}, tq, 1)\right)$ 
(the right side of the last equality having the obvious meaning).
It follows that we can choose generators $f_1, \ldots, f_m$ of $\cE_\ell(A)$ satisfying
$f_i(s,t,q) = f_i(sq^{-1},tq,1)$ for $1 \leq i \leq m.$

The higher Alexander invariant $\Delta^{\ell}_K$ is given as the greatest common divisor of the set of generators of $\cE_\ell$, and it is well-defined up to units in the ring of Laurent polynomials. Thus, we have $\Delta^{\ell}_K = \gcd(f_1, \ldots, f_m).$
 
Since any irreducible factor in $\ZZ[s^{\pm 1}, t^{\pm 1}, q^{\pm 1}]$ of a polynomial of the form 
$f(sq^{-1},tq)$ also has this form, it follows that $\Delta^{\ell}_K$ satisfies the condition
$\Delta^{\ell}_K(s,t,q) = \Delta^{\ell}_K(sq^{-1},tq,1)$.
\end{proof}

The following corollary is a direct consequence of Proposition \ref{HdeterminesG} and Theorem \ref{theorem-H-K}.
\begin{corollary} \label{GdeterminesH}
For any virtual knot or  link $K$, the virtual Alexander polynomial is determined by the generalized Alexander polynomial via the formula 
$$H_K(s,t,q) = G_K(sq^{-1}, tq).$$
\end{corollary}

\begin{lemma}
\label{labelthislemma}
Let $K$ be a virtual knot or  link.  Then $H_K(s,t,1)$ is divisible by $1-s$. 
\end{lemma}

\begin{proof}
Let $A$ be an Alexander matrix for $K$ and
let $A_1, \ldots, A_n$ be the columns of $A$.
Let $B$ be the matrix obtained by replacing the first column of $A$ with the left side of the relation \eqref{fundamental-Fox-two}.
Then $\det(B)=0$ and the  linearity of the determinant as a function of the first column yields the identity
$$\det(A)(t-1) + \det(A')(s-1) + \det(A'')(q-1) =0,$$
where $A'$ is obtained  from $A$ by replacing its first column with $\left(\frac{\partial r^{}}{\partial s_{}} \right)$ and 
$A''$ is obtained  from $A$ by replacing its first column with $\left(\frac{\partial r}{\partial q} \right)$.
Evaluating at $q=1$ gives
$$H_K(s,t,1)(t-1) =  \left. \det(A')\right|_{q=1}(1-s),$$
and so $H_K(s,t,1)$ is divisible by $1-s$.
\end{proof}

The next proposition shows that $H_K(s,t,q)$ is always divisible by $(q-s)(1-qt)(1-st)$ if $K$ is a virtual knot.
 
\begin{proposition} 
\label{H_K-divisibility}
\begin{itemize}
\item[(i)] If $K$ is a virtual knot or  link then $H_K(s,t,q)$ is divisible by $q-s$. 
\item[(ii)] If $K$ is a virtual knot or link then  $H_K(s,t,q)$ is divisible by $1-tq$.
\item[(iii)] If $K$ is a virtual knot then  $H_K(s,t,q)$ is divisible by $1-st$.
\end{itemize}
\end{proposition}

\begin{proof}
Firstly, (i) is an immediate consequence of combining Theorem \ref{theorem-H-K} and Lemma \ref{labelthislemma}.

Secondly, by direct calculation, we see that for $1\leq i \leq k-1$,  
\begin{eqnarray*}
{[t^{k-1}, \cdots, t, 1 ]} \, \Psi(\si_i)(s,t,t^{-1}) &=& {[t^{k-1}, \cdots,t, 1]} \quad \text{and} \\ 
{[t^{k-1}, \cdots, t, 1]}  \, \Psi(\tau_i)(s,t,t^{-1}) &=& {[t^{k-1}, \cdots,t, 1]}.
\end{eqnarray*}
Writing a given virtual braid $\be \in \VB_k$ as a word in the generators $\si_i, \tau_i,$ it follows that ${[t^{k-1}, \cdots,t, 1 ]}$ is always a left eigenvector of $\Psi(\be)(s,t,t^{-1})$ with eigenvalue $\la=1$.

If $K$ is a virtual knot or link, and $\be \in \VB_k$ is a braid with closure $\wh{\be} =K$, then $H_K(s,t,t^{-1}) = \det\left(\Psi(\be)(s,t,t^{-1})-I_k\right)=0$ by Theorem \ref{Burau}. Thus $1-qt$ is a factor of $H_K(s,t,q)$, and this shows (ii).

\begin{figure}[ht]
\centering
\includegraphics[scale=0.64]{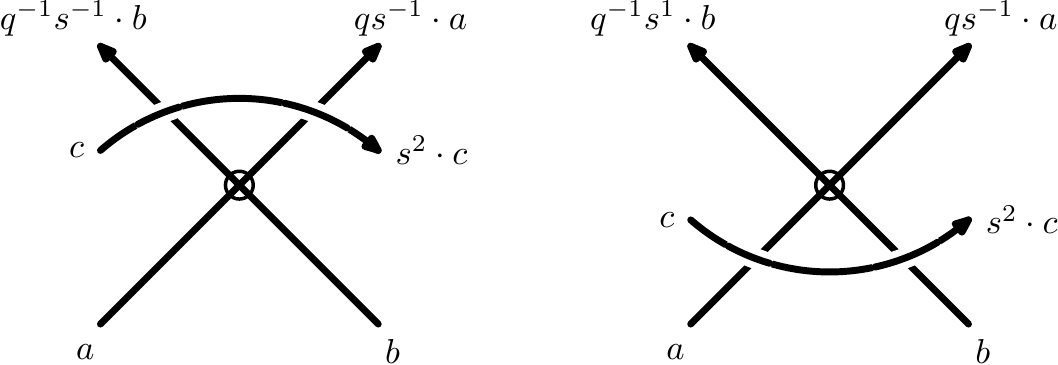} \qquad \includegraphics[scale=0.64]{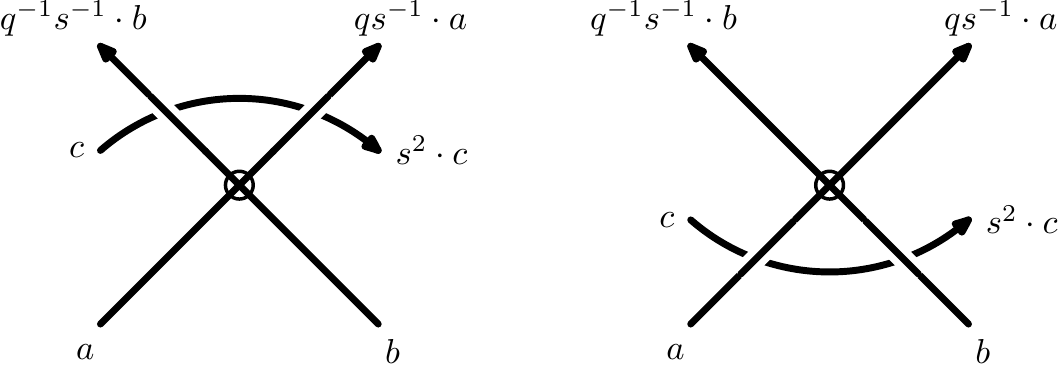}
\caption{Invariance of $H_K(s,s^{-1},q)$ under $(f1)$ and $(f2)$.}
\label{forb}
\end{figure}      


Thirdly, notice that $H_K(s,t,q)$ is invariant under $(f1)$ and $(f2)$, the first and second forbidden moves, upon setting $t=s^{-1}$. This can be seen using the diagrammatic  argument depicted in Figure \ref{forb}.
From this it follows that $H_K(s,s^{-1},q)=0$ for any virtual knot $K$, because virtual knots  can unknotted using $(f1)$ and $(f2)$ (see \cite{Kanenobu, Nelson-for}).   
 Thus  if $K$ is a virtual knot, then $1-st$ is a factor of $H_K(s,t,q)$, and this gives  (iii). Note that the same argument does not hold for virtual links, as it is not generally true that virtual links can be trivialized using $(f1)$ and $(f2)$.  (In \cite{Okabayashi}, Okabayashi gives a classification of virtual links up to forbidden moves.)
\end{proof}


\section{A normalization for the virtual Alexander polynomial} \label{section-5} 
 
In this section, we give a natural normalization for the virtual Alexander polynomial $H_K(s,t,q)$ that is well defined up to multiplication by $(st)^i.$ This is achieved by
interpreting the virtual Alexander polynomial in terms of the virtual Burau representation (see Theorem \ref{Burau}). Key ingredients of the proof are the results of Kamada in \cite{Kamada, KL} giving virtual analogues of the Alexander and Markov theorems.
The virtual Alexander theorem states that every virtual link is obtained as the closure of a virtual braid.
The virtual Markov theorem, recalled below, gives necessary and sufficient conditions for two virtual braids $\be \in \VB_{k}$ and $\be' \in \VB_{k'}$ to represent the same virtual link.  

\begin{theorem}[Kamada \cite{Kamada}] \label{virtualMarkov}
Two virtual braids have equivalent closures as virtual links if and only if they are related to each other by a sequence of the following
{\rm VM1, VM2}  and {\rm VM3} moves:
\begin{enumerate}[hspace{0pt}]
\item[VM1:] a conjugation in the virtual braid group,
\item[VM2:] a positive, negative, or virtual right stabilization, and their inverse operations,
\item[VM3:] a right or left virtual exchange move.
\end{enumerate}
\end{theorem}

\begin{figure}[ht]
\centering
\includegraphics[scale=1.00]{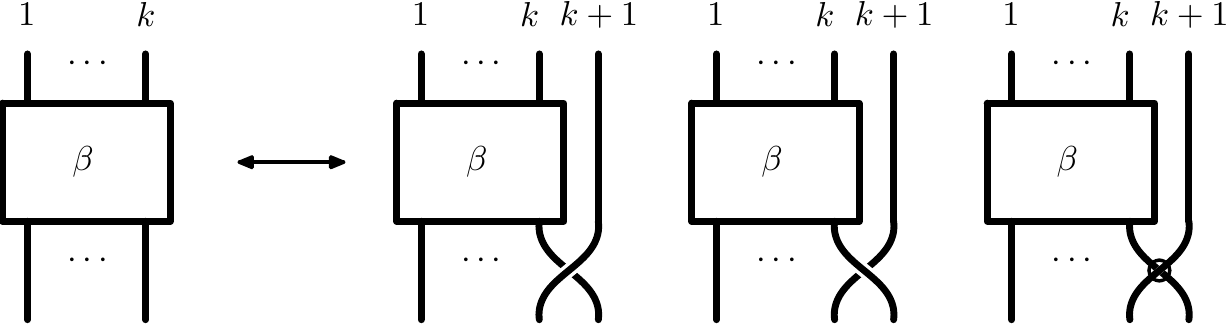}
\caption{The right stabilization moves: positive, negative, virtual types.}
\label{kamada-stabilization}
\end{figure}   

\begin{figure}[ht]
\centering
\includegraphics[scale=0.90]{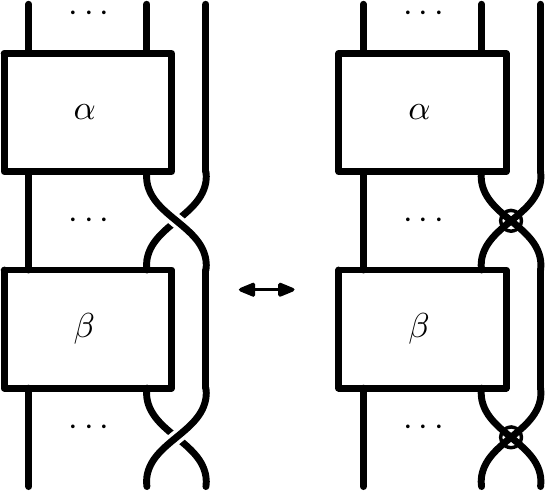}\qquad \qquad   \includegraphics[scale=0.90]{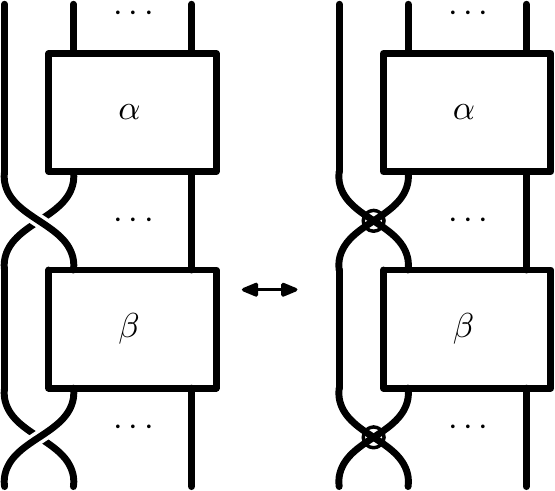}
\caption{The right and left virtual exchange moves.}
\label{Kamada-exchange}
\end{figure}   
Kamada showed that, unlike in the classical case, a virtual exchange move is not a consequence of the other moves, see \cite{Kamada-04}.
 
Let $\cL = \ZZ[s^{\pm 1}, t^{\pm 1}, q^{\pm1}]$.
We use the virtual Burau representation $\Psi\colon \VB_k \to GL_k(\cL)$ (Definition \ref{def_virtual_Burau_rep})
to define an invariant $H_\be(s,t,q)$ of  a  virtual braid  $\be \in \VB_k$ via the formula
 $$H_{\be}(s,t,q) = \det\left(\Psi(\be) \, - \, I_k\right).$$
 Note that Theorem \ref{Burau} implies that if $K =\wh \be$ is the closure of $\be$ then
 $H_K(s,t,q) =H_{\be}(s,t,q)$ up to units in $\cL$.
 
The normalized virtual Alexander polynomial is obtained by studying the effect of the virtual Markov moves on $H_\be(s,t,q).$ 
 
We show that  $H_{\be}$ is invariant under a VM1 move.  For $\ga \in \VB_k$, 
\begin{eqnarray*}
H_{\ga \be \ga^{-1}}(s,t,q) &=& \det\left(\Psi(\ga \be \ga^{-1}) \, - \, I_k\right) =   \det\left(\Psi(\ga) \Psi( \be)  {\Psi(\ga)}^{-1} \, - \, I_k  \right) \\
                                             &=&   \det\left(\Psi(\be) \, - \, I_k\right) = H_{\be}(s,t,q).
\end{eqnarray*}

Next, we examine how $H_\be$ changes under a VM2 move.

There is a natural injective stabilization homomorphism
$\mu_k \colon \VB_k \hookrightarrow \VB_{k+1}$ given 
by taking the generators $\si_i$,  $\tau_i$,  $1 \leq i \leq k-1$,
of $\VB_k$ to the generators designated by the same symbols in
$\VB_{k+1}$.
We write $\be$ for $\mu_k(\be)$ when the meaning is clear from the context.

\begin{proposition}
\label{VM2-inv}
Let $\be \in \VB_k$.   If  $\be'$ is obtained from $\be$ by a right stabilization move of virtual type or of positive type then $H_{\be'}(s,t,q) =  -H_\be(s,t,q)$.
If  $\be'$ is obtained from $\be$ by a right stabilization move of negative type then $H_{\be'}(s,t,q) =  -(st) H_\be(s,t,q)$.
\end{proposition}
\begin{proof} Let $\Psi \colon \VB_k \to GL_k(\cL)$ and $\Psi'\colon \VB_{k+1} \to GL_{k+1}(\cL)$ be the virtual Burau representations,
and set $A = \Psi(\be)$ and $A' = \Psi'(\be')$. 
We consider separately the case of a virtual, positive, and negative stabilization.


\noindent
{\it Virtual stabilization.} ~We have $\be' =\be \tau_k$.
Thus, the matrices $A$ and $A'$ are related by:
$$A' = \Psi'(\be \tau_k) = \Psi'(\be) \Psi'(\tau_k) = 
\begin{bmatrix}
A      &\hspace{0.3em} {\bf 0}  \\[0.3em]
{\bf 0} &\hspace{0.3em} 1
\end{bmatrix} 
\begin{bmatrix}
I_{k-1} &  \cdots  & {\bf 0} \\ 
\vdots & 0 & q \\
{\bf 0} & {\,q^{-1}} & 0
\end{bmatrix}.$$
Let $A_k$ denote the $k$-th column of $A$ and let $E$ be the elementary matrix
$$ E = \begin{bmatrix}
I_{k-1} &  \cdots  & {\bf 0} \\ 
\vdots & 1 & 0 \\
{\bf 0} & {\,q^{-1}} & 1
\end{bmatrix}.$$
Observe that
$$ \left( A' - I_{k+1} \right) E= 
\begin{bmatrix}
A  - I_k      &  A_k q  \\[0.5em]
{\bf 0} & -1
\end{bmatrix}.$$
Since $\det(E)=1$, It follows that
\begin{eqnarray*}
H_{\be'}(s,t,q) &=& 
\det\left(A' - I_{k+1}\right) =   \det\left(\left(A' - I_{k+1}\right)E\right)  \\
&=&   \det\left(  \begin{bmatrix}  A  - I_k      &  A_k q  \\[0.5em]  
{\bf 0} & -1 \end{bmatrix} \right) = -\det\left( A  - I_k \right) = -H_{\be}(s,t,q).
\end{eqnarray*}

\smallskip


\noindent
{\it Positive stabilization.} ~We have $\be' = \be  \si_k^{-1}$.
Thus the matrices $A$ and $A'$ are related by:
$$ A' = \Psi'(\be \si_k^{-1}) = \Psi'(\be) \Psi'(\si_k^{-1}) =\begin{bmatrix}
A & \hspace{0.3em} {\bf 0}  \\[0.3em]
{\bf 0} & \hspace{0.3em} 1
\end{bmatrix} 
\begin{bmatrix}
I_{k-1} &  \cdots  & {\bf 0} \\ 
\vdots & 1- s^{-1} t^{-1}& t^{-1} \\
{\bf 0} & s^{-1} & 0
\end{bmatrix}.$$
Let $A_k$ denote the $k$-th column of $A$ and let $E$ be the elementary matrix
$$E = \begin{bmatrix}
I_{k-1} &  \cdots  & {\bf 0} \\ 
\vdots & 1 & 0 \\
{\bf 0} & {\,s^{-1}} & 1
\end{bmatrix}.$$
Observe that
$$ \left( A' - I_{k+1} \right) E= 
\begin{bmatrix}
A  - I_k      &  A_k t^{-1}  \\[0.5em]
{\bf 0} & -1
\end{bmatrix}.$$
Since $\det(E)=1$, it follows that
\begin{eqnarray*}
H_{\be'}(s,t,q) &=& \det\left(A' - I_{k+1}\right) 
= \det\left(\left(A' - I_{k+1}\right)E\right)  \\
&=&   \det\left(  \begin{bmatrix}  A  - I_k &  A_k t^{-1} \\[0.5em]  
{\bf 0} & -1 \end{bmatrix} \right) = -\det\left( A  - I_k \right) = -H_{\be}(s,t,q).
\end{eqnarray*}


\noindent
{\it Negative stabilization.} ~We have $\be' = \be  \si_k$.
Thus, the matrices $A$ and $A'$ are related by:
$$ A' = \Psi'(\be \si_k) = \Psi'(\be) \Psi'(\si_k) =
\begin{bmatrix}
A & \hspace{0.3em} {\bf 0}  \\[0.3em]
{\bf 0} & \hspace{0.3em} 1
\end{bmatrix} 
\begin{bmatrix}
I_{k-1} &  \cdots  & {\bf 0} \\ 
\vdots & 0 & s \\
{\bf 0} & t & 1-st
\end{bmatrix}.$$
Let $A_k$ denote the $k$-th column of $A$ and let $E$ be the elementary matrix
$$ E = \begin{bmatrix}
I_{k-1} & \cdots & {\bf 0} \\ 
\vdots & 1 & 0 \\
{\bf 0} & {\,s^{-1}} & 1
\end{bmatrix}.$$
Observe that
$$ \left( A' - I_{k+1} \right) E= 
\begin{bmatrix}
A  - I_k      &  A_k s  \\[0.5em]
{\bf 0} & -st
\end{bmatrix}.$$
Since $\det(E)=1$, it follows that
\begin{eqnarray*}
H_{\be'}(s,t,q) &=& \det\left(A' - I_{k+1}\right) 
=   \det\left(\left(A' - I_{k+1}\right)E\right) \\
&=&   \det\left(  \begin{bmatrix}  A  - I_k      &  A_k s \\[0.5em]  
{\bf 0} & -st \end{bmatrix} \right) = -(st)\det\left( A  - I_k \right) = -(st)H_{\be}(s,t,q).
\end{eqnarray*}
\end{proof}


Next, we show that $H_\be$ is unchanged by a VM3 move.

\begin{proposition}
\label{VM3-inv}
Let $\be_1 \in \VB_{k+1}$.   If  $\be_2$ is obtained from $\be_1$ by a right or a left exchange move then $H_{\be_1}(s,t,q) =  H_{\be_2}(s,t,q)$.
\end{proposition}
\begin{proof} 
For $\al$ and $\be$ as in Figure \ref{Kamada-exchange},
let  $A = \Psi(\al)$ and $B= \Psi(\be)$  (these are $k \times k$ matrices).
We consider separately the cases of right and left virtual exchange moves.


\noindent
{\it Right virtual exchange move.} ~Let
$$ A'= \begin{bmatrix}
A   & \hspace{0.3em} {\bf 0}  \\[0.3em]
{\bf 0}   & \hspace{0.3em} 1
\end{bmatrix},
\qquad
B'= \begin{bmatrix}
B   & \hspace{0.3em} {\bf 0}  \\[0.3em]
{\bf 0}   & \hspace{0.3em} 1
\end{bmatrix},$$

$$
P = \begin{bmatrix}
I_{k-1} &  \cdots  & {\bf 0} \\ 
\vdots & 0 & s \\
{\bf 0} & t & 1-st
\end{bmatrix},
\qquad
Q = \begin{bmatrix}
I_{k-1} &  \cdots & {\bf 0} \\ 
\vdots & 0 & q \\
{\bf 0} & q^{-1} & 0
\end{bmatrix},$$

$$ E = \begin{bmatrix}
I_{k-1} &  \cdots  & {\bf 0} \\ 
\vdots   & 1 & 0 \\
{\bf 0} & -t+s^{-1} & t
\end{bmatrix},
\qquad
F_1= \begin{bmatrix}
I_k        & {\bf 0}  \\[0.3em]
{\bf 0}     & t
\end{bmatrix},
\qquad
F_2 = \begin{bmatrix}
I_k  & \hspace{0.3em} {\bf 0}  \\[0.3em]
{\bf 0}    & \hspace{0.3em} q
\end{bmatrix}.$$

Since $\be_1 = \al \si_k \be \si_k^{-1}$ and  $\be_2 = \al \tau_k \be \tau_k$, we have
$$
H_{\be_1} =  \det\left(A'P B' P^{-1} - I_{k+1}\right)   \quad  \text{ and  }  \quad
H_{\be_2}  = \det\left(A'Q B' Q- I_{k+1}\right).
$$


Consider the matrices
$$
M_1 = F_1^{-1} \left(A' P B' P^{-1}  - I_{k+1}\right) E   \quad \text{   and   }   \quad
M_2 =  F_2 \left(A' Q B'Q - I_{k+1}\right) F_2^{-1}.
$$

Since $\det(E)=t=\det(F_1)$, it follows that $\det(M_1) = H_{\be_1}$.   Also, $\det(M_2) = H_{\be_2}$.

\noindent{\bf Claim:} $M_1 = M_2$. \\
\noindent
Let $A_k$ and $B_k$ denote the $k$-th column of $A$ and $B$, respectively, and let
$\wh A_k$ and $\wh B_k$ be the $k\times (k-1)$ matrices obtained  by removing the $k$-th column from $A$ and $B$, respectively.

One easily finds that
\begin{eqnarray*}
\left( F_1^{-1} A' \right)P &=&
\begin{bmatrix}
A     & \hspace{0.3em}  {\bf 0}  \\[0.3em]
{\bf 0}     & \hspace{0.3em} t^{-1}
\end{bmatrix}
\begin{bmatrix}
I_{k-1} &  \cdots  & {\bf 0} \\ 
\vdots & 0 & s \\
{\bf 0} & t & 1-st
\end{bmatrix} = 
\begin{bmatrix}
{\wh A}_k & {\bf 0}   &  A_k \, s\\[0.5em] 
{\bf 0} & 1 & t^{-1}-s
\end{bmatrix} \quad \text{ and } \\
B'\left(P^{-1}E\right)  
&= &
\begin{bmatrix}
B   & \hspace{0.3em}  {\bf 0}  \\[0.3em]
{\bf 0}   & \hspace{0.3em} 1
\end{bmatrix}
\begin{bmatrix}
I_{k-1} &  \cdots  & {\bf 0} \\ 
\vdots & 0 & 1 \\
{\bf 0} & s^{-1} & 0
\end{bmatrix} 
=
\begin{bmatrix}
{\wh B}_k &{\bf 0}   &  B_k\\[0.5em] 
{\bf 0} & s^{-1} & 0
\end{bmatrix}.
\end{eqnarray*}

It follows that
\begin{equation} \label{M-1}
\begin{split}
M_1 &=  F_1^{-1} A' P \cdot B'P^{-1} E ~-~ F_1^{-1} E \\
&= \begin{bmatrix}
{\wh A}_k & {\bf 0}   &  A_k \, s \\[0.5em] 
{\bf 0} & 1 & t^{-1}-s
\end{bmatrix}
\begin{bmatrix}
{\wh B}_k & {\bf 0}   &  B_k\\[0.5em] 
{\bf 0} & s^{-1} & 0
\end{bmatrix} ~-~
\begin{bmatrix}
I_{k-1} &  \cdots  & {\bf 0} \\ 
\vdots   & 1 & 0 \\
{\bf 0} & -1+(st)^{-1} & 1
\end{bmatrix}.
\end{split}
\end{equation}

An easy computation shows that 
\begin{eqnarray*}
B' \left(Q F_2^{-1}\right) &=& \begin{bmatrix}
B    & \hspace{0.3em} {\bf 0}  \\[0.3em]
{\bf 0}    & \hspace{0.3em} 1
\end{bmatrix}
\begin{bmatrix}
I_{k-1} &  \cdots & {\bf 0} \\ 
\vdots & 0 & 1 \\
{\bf 0} & q^{-1} & 0
\end{bmatrix}
=
\begin{bmatrix}
{\wh B}_k & {\bf 0}   & B_k\\[0.5em] 
{\bf 0} & q^{-1} & 0
\end{bmatrix} \quad \text{ and } \\
\left( F_2 A'\right) Q &=&
\begin{bmatrix}
A   & \hspace{0.3em} {\bf 0}  \\[0.3em]
{\bf 0}   & \hspace{0.3em} q
\end{bmatrix}
\begin{bmatrix}
I_{k-1} &  \cdots & {\bf 0} \\ 
\vdots & 0 & q \\
{\bf 0} & q^{-1} & 0
\end{bmatrix}
=
\begin{bmatrix}
{\wh A}_k &\hspace{0.3em} {\bf 0}   & \hspace{0.3em} A_k \, q \\[0.5em] 
{\bf 0} &\hspace{0.3em} 1 &\hspace{0.3em} 0
\end{bmatrix}.
\end{eqnarray*}

Hence
\begin{equation} \label{M-2}
\begin{split}
M_2 &= 
 \begin{bmatrix}
{\wh A}_k &\hspace{0.3em} {\bf 0}   & \hspace{0.3em}  A_k \, q\\[0.5em] 
{\bf 0} & \hspace{0.3em} 1 & \hspace{0.3em} 0
\end{bmatrix}
\begin{bmatrix}
{\wh B}_k & {\bf 0}   &  B_k\\[0.5em] 
{\bf 0} & q^{-1} & 0
\end{bmatrix} ~-~ I_{k+1}.
\end{split}
\end{equation}

The claim now follows by direct comparison of Equations \eqref{M-1} and \eqref{M-2}, and this shows that 
$H_{\be_1} = \det(M_1) = \det(M_2) = H_{\be_2} $.

\smallskip


\noindent
{\it Left virtual exchange move.} ~Let
$$ A'= \begin{bmatrix}
1 & \hspace{0.3em} {\bf 0}  \\[0.3em]
{\bf 0}     & \hspace{0.3em}  A
\end{bmatrix},
\qquad
B'= \begin{bmatrix}
1  & \hspace{0.3em} {\bf 0}  \\[0.3em]
{\bf 0} & \hspace{0.3em} B 
\end{bmatrix},$$

$$ P = \begin{bmatrix}
0 & s  & \hspace{0.3em} {\bf 0} \\ 
t  & 1-st & \hspace{0.3em}\vdots  \\[0.4em]
{\bf 0} & \cdots & \hspace{0.3em} I_{k-1}
\end{bmatrix},
\qquad
Q = \begin{bmatrix}
0 & q  & \hspace{0.3em} {\bf 0} \\ 
q^{-1}  & 0 & \hspace{0.3em} \vdots \\[0.4em]
{\bf 0}& \cdots & \hspace{0.3em} I_{k-1}
\end{bmatrix},$$

$$ E = \begin{bmatrix}
t^{-1}  & s - t^{-1} & \hspace{0.3em} {\bf 0} \\ 
0  & 1 & \hspace{0.3em}\vdots \\[0.4em]
{\bf 0} &  \cdots & \hspace{0.3em} I_{k-1}
\end{bmatrix},
\qquad
F_1= \begin{bmatrix}
t       & \hspace{0.3em} {\bf 0}  \\[0.3em]
{\bf 0}     & \hspace{0.3em} I_k 
\end{bmatrix},
\qquad
F_2 = \begin{bmatrix}
q      & \hspace{0.3em} {\bf 0}  \\[0.3em]
{\bf 0}     & \hspace{0.3em} I_k 
\end{bmatrix}.$$

Since $\be_1 = \al \si_1 \be \si_1^{-1}$ and  $\be_2 = \al \tau_1 \be \tau_1$, we have
\begin{eqnarray*}
H_{\be_1} &=&  \det\left(A'PB' P^{-1} - I_{k+1}\right) =    \det\left(B' P^{-1} A'P  - I_{k+1}\right) \quad  \text{ and }\\
H_{\be_2}  &=& \det\left(A'Q B' Q- I_{k+1}\right) =    \det\left(B' Q A' Q  - I_{k+1}\right).
\end{eqnarray*}
(Note that  $\det(XY - I_k) =  \det(YX - I_k)$ for any  $k \times k$ matrices $X$, $Y$  over a commutative ring.)

Consider the matrices
$$M_1 = F_1 \left(B' P^{-1} A'P  - I_{k+1}\right) E   \quad \text{   and   } \quad
M_2 =  F_2^{-1} \left(B' Q A' Q - I_{k+1}\right) F_2.$$
Since $\det(F_1)=t$ and $\det(E)=t^{-1}$, we have $\det(M_1) = H_\be$.   Also, $\det(M_2) = H_{\be'}$.
As in the case of the right exchange move, one can show that $M_1 = M_2$, and 
this implies that $H_{\be_1}(s,t,q) = \det(M_1) = \det(M_2) = H_{\be_2}(s,t,q)$.
\end{proof}


We apply these results to define a preferred normalization for the virtual Alexander polynomial $H_K(s,t,q)$ as follows.
Let $K$ be a virtual knot or link  represented as the closure of a braid $\be \in \VB_n,$ and let $\writhe(\be)$ be the writhe  and $v(\be)$ the virtual crossing number of closure $\wh \be$. Note that $v(\be)$ depends on the braid word for $\be$, whereas $\writhe(\be)$ is  invariant under the virtual braid relations \eqref{classical-rel}, \eqref{virtual-rel}, and \eqref{mixed-rel}. In the definition below, we only need to correct by $(-1)^{\writhe(\be)+v(\be)}$, which can be easily calculated from the braid word for $\be$ and which depends only on $\be \in \VB_n$ and not on the braid word.
 
\begin{definition} \label{Defnorm}
The {\it normalized  virtual Alexander polynomial} is given by setting $${\wh H}_K(s,t,q) = (-1)^{\writhe(\be) + v(\be)} H_\be(s,t,q).$$
It is an invariant of virtual knots and links that is well-defined up to a factor of $(st)^i.$
\end{definition}

Invariance follows from Theorem \ref{virtualMarkov}, Proposition \ref{VM2-inv} and Proposition \ref{VM3-inv}, which show that $(-1)^{\writhe(\be) + v(\be)} H_\be(s,t,q)$
is independent of braid representative, up to an overall factor of $(st)^i.$

Notice that $(-1)^{\writhe(\be) + v(\be)} = (-1)^{|\be|},$ where $|\be|$ denotes the length of $\be$ as a word in $\si_i, \si_i^{-1}, \tau_i$.
As with $H_K(s,t,q),$ we will suppress the inherent indeterminacy of $\wh H_K(s,t,q)$, namely we write ${\wh H}_K(s,t,q) = f$ whenever $f \in \cL$ is a Laurent polynomial such that ${\wh H}_K(s,t,q) = (st)^i \cdot f$ for some $i \in \ZZ.$

The next result is proved in essentially the same way as Theorem \ref{theorem-H-K}, and details are left to the industrious reader.

\begin{proposition} If $K$ is a virtual knot or link, then ${\wh H}_K(s,t,q) = {\wh H}_K(sq^{-1}, tq, 1).$
\end{proposition}

Using the normalized virtual Alexander polynomial, we obtain the following improvement of Theorem \ref{v-bound}. We define $\deg_q$ and $\deg_{q^{-1}}$ for the Laurent polynomial ${\wh H}_K(s,t,q)$ by regarding it as a polynomial in $q$ and $q^{-1}.$
 
\begin{theorem} \label{better-v-bound}
Given a virtual knot or link $K$,  then  
$$\deg_{q^{-1}} {\wh H}_K(s,t,q)~\leq~v(K) \quad \text{ and } \quad \deg_q {\wh H}_K(s,t,q)~\leq~v(K).$$  
\end{theorem}

We present two examples where one obtains sharp bounds on the virtual crossing numbers from the normalized Alexander polynomial. In both cases, these bounds are better than those obtained by the arrow polynomial or the unnormalized Alexander polynomial.

\begin{figure}[ht]
\centering
\includegraphics[scale=1.90]{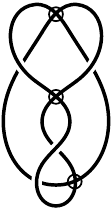}\qquad \qquad \qquad \qquad
\includegraphics[scale=0.45]{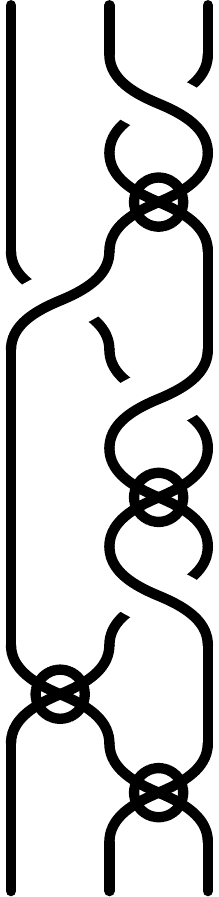}
\caption{The virtual knot $4.42$ has $v(K)=3$ and is the closure of the virtual braid  $\be = \si_2 \tau_2 \si_1^{-1} \si_2^{-1} \tau_2 \si_2 \tau_1 \tau_2$.}
\label{4-42}
\end{figure}

\begin{example}
Let $K$ be the virtual knot 4.42 depicted in Figure \ref{4-42}. Using the arrow polynomial, it has been shown that $v(K) \geq 2$, see \cite{BDK}. On the other hand, using the fact that $K$ is the closure of the braid $\be = \si_2 \tau_2 \si_1^{-1} \si_2^{-1} \tau_2 \si_2 \tau_1 \tau_2 \in \VB_3$, one can  compute  the normalized Alexander polynomial 
$${\wh H}_K(s,t,q) =  1 - s^{-1} t^{-1} + (s^{-2} t^{-1}- s^{-1})q^{1} + (s^{-1} t -t^2) q^{2}  + (s^{-1}t^2   - s^{-2} t)q^{3} .$$
Since $\deg_{q} {\wh H}_K(s,t,q) =3,$ Theorem \ref{better-v-bound} implies that $v(K)=3.$
\end{example}

\begin{figure}[ht]
\centering
\includegraphics[scale=1.70]{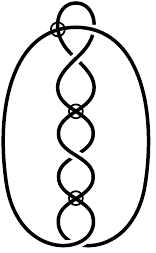}\qquad \qquad \qquad \qquad \includegraphics[scale=0.45]{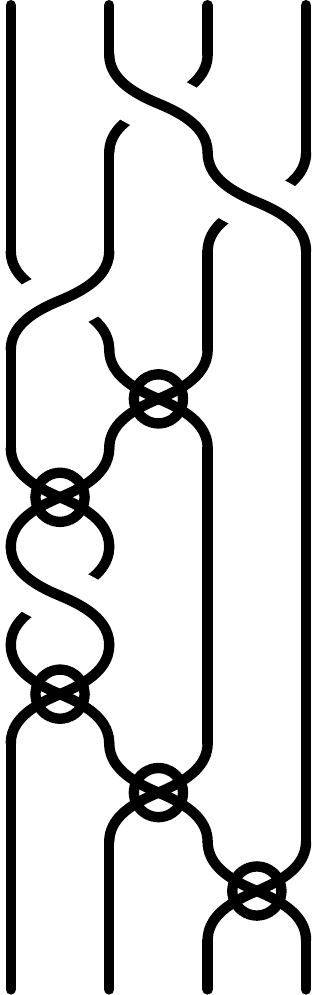}
\caption{The virtual knot $4.45$ has $v(K)=3$ and is the closure of the virtual braid $\be = \si_2 \si_3 \si_1^{-1} \tau_2 \tau_1 \si_1 \tau_1 \tau_2 \tau_3$.}
\label{4-45}
\end{figure}

\begin{example}
Let $K$ be the virtual knot 4.45 depicted in Figure \ref{4-45}. Using the arrow polynomial, it has been shown that $v(K) \geq 1$, see \cite{BDK}. On the other hand, using the fact that  $K$ is the closure of the braid $\be = \si_2 \si_3 \si_1^{-1} \tau_2 \tau_1 \si_1 \tau_1 \tau_2 \tau_3 \in \VB_4$, one can   compute the normalized Alexander polynomial 
$${\wh H}_K(s,t,q) = (s^2 t-s^3 t^2 + t^{-1} - s) q^{-1} + s^3 t^3 - s^2 t^2 - s^{-1} t^{-1}+1 + (s t^2-t) q^{1} + (s^{-1} t^2-t^3) q^{3}.$$
Since $\deg_{q} {\wh H}_K(s,t,q) =3,$ Theorem \ref{better-v-bound} implies that $v(K)=3.$
\end{example}

\section{A skein formula for the virtual Alexander polynomial} \label{section-6}
In \cite{Sawollek} and \cite{KR}, a skein formula for the generalized Alexander polynomial $G_K(s,t)$ is established, and the goal of this section is to establish a skein formula for the virtual Alexander polynomial $H_K(s,t,q)$. This result is similar to the one obtained by Sawollek in \cite{Sawollek}, and we provide an independent and elementary proof. Throughout this section, to avoid the indeterminacy inherent in $H_K$ and ${\wh H}_K,$ we will work with virtual knot and link diagrams.

To begin, we give a formula for expanding the determinant of an $n \times n$ matrix along the first two rows, following the well-known principle of multi-row Laplace expansion.

Suppose $A$ is an $n \times n$ matrix given by
$$A=\begin{bmatrix} 
a_1 & a_2 & \cdots & a_n \\
b_1 & b_2 &\cdots & b_n \\[0.3em]   
&&P \\[0.3em]
\end{bmatrix},$$
where $P =(p_{ij})$ is an $(n-2) \times n$ matrix.

By multi-row Laplace expansion along the first and second rows of $A$,
\begin{equation} \label{eq:star}
\det A = \sum_{i=1}^{n-1} \sum_{j=i+1}^n (-1)^{i+j-1} (a_ib_j-a_jb_i) \det(\wh{P}_{ij}). 
\end{equation}

In this formula, $\wh{P}_{ij}$ denotes the square $(n-2) \times (n-2)$ matrix obtained by removing the $i$-th and $j$-th columns from $P.$

Now suppose $K_+, K_-,$ and $K_0$ are three oriented virtual knot or link diagrams which are identical everywhere except near one crossing, where $L_+$ has a positive crossing, $K_-$ has a negative crossing, and  $K_0$ has the crossing removed, see Figure \ref{skein}. 

\begin{figure}[ht]
\centering
\includegraphics[scale=0.90]{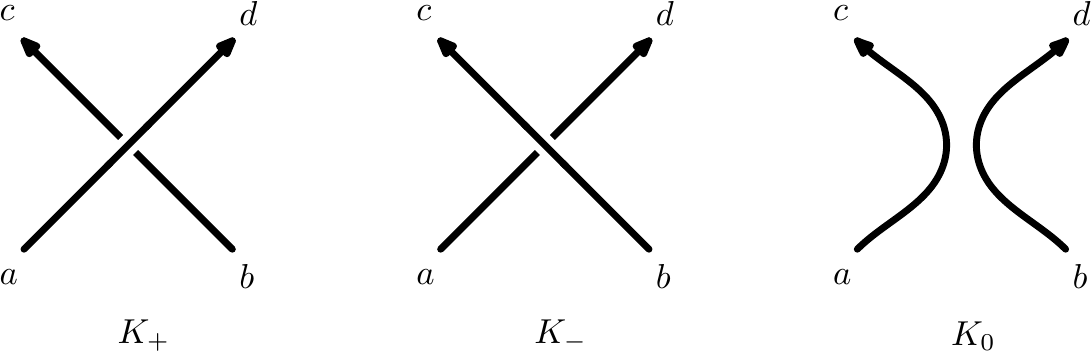}
\caption{The skein triple $K_+, K_-$ and $K_0.$}
\label{skein}
\end{figure}    

Label the short arcs of $K_+$ so that at the crossing in question, the arcs are as in Figure \ref{skein}, namely $a$ for the incoming over-crossing, $b$ for the incoming under-crossing, $c$ for the outgoing under-crossing and $d$ for the outgoing over-crossing and likewise for $K_-$ and $K_0$. 

Thus, 
\begin{eqnarray*} 
H_{K_+}(s,t,q) &=&
\det \begin{bmatrix} 1-st& t & -1 & 0 & 0 & \cdots  & 0 \\ 
s & 0 & 0 & -1  & 0 & \cdots & 0 \\[0.3em] 
& &&  P & & \\[0.3em]\end{bmatrix}, \\ \\
H_{K_-}(s,t,q)&=&
\det \begin{bmatrix}
0& s^{-1} & -1 & 0 & 0 & \cdots &  0 \\ 
t^{-1} & 1-(st)^{-1} & 0 & -1  & 0 & \cdots & 0 \\[0.3em] 
& & P & & \\[0.3em]\end{bmatrix}, \\ \\
H_{K_0}(s,t,q) &=&
\det \begin{bmatrix} 1 & 0 & -1 & 0 & 0 & \cdots & 0 \\ 
0 & 1 & 0 & -1  & 0 & \cdots & 0 \\[0.3em] 
& & & P & &  \\[0.3em] \end{bmatrix}. \\
\end{eqnarray*}

Applying formula \eqref{eq:star} to $H_{K_+}(s,t,q)$ and $H_{K_-}(s,t,q)$, we get that 
\begin{eqnarray*}
H_{K_+}(s,t,q) &=&-st \wh P_{12} - s\wh P_{13} -(1-st) \wh P_{14} +t\wh P_{24} + \wh P_{34}\\
H_{K_-}(s,t,q) & =& -(st)^{-1}\wh P_{12} -t^{-1}\wh P_{13}+(1-(st)^{-1})\wh P_{23} + s^{-1}\wh P_{24}  +\wh P_{34}
\end{eqnarray*}

Thus, 
\begin{eqnarray*}
 H_{K_+} - (st) H_{K_-} &=& \left( 1-st \right) \wh P_{12} + \left(st-1\right) \wh P_{14} +\left(1-st\right) \wh P_{23}+\left(1-st\right) \wh P_{34} \\
&=& \left(1-st\right) \left(\wh P_{12} - \wh P_{14} + \wh P_{23} + \wh P_{34}\right).
\end{eqnarray*}

On the other hand, expanding $H_{K_0}(s,t,q)$ along the first two rows reveals that
$$H_{K_0}(s,t,q)=\wh P_{12} - \wh P_{14}+ \wh P_{23}  + \wh P_{34},$$
and it follows that
$$H_{K_+} - (st) \; H_{K_-} = \left( 1-st \right) H_{K_0}.$$

Since the entire calculation can be performed on three braids $\be_+, \be_-, \be_0$ that are identical everywhere outside of one crossing, 
this argument applies equally well to the normalized invariant of Definition \ref{Defnorm}.
One only needs to note that the virtual crossing number and writhe satisfy 
$v(\be_+) = v(\be_-)=v(\be_0)$ and $\writhe(\be_+) = \writhe(\be_-)+2=\writhe(\be_0)+1$.
This causes a sign change in the coefficient of $\wh H_{\be_0}$, and the following theorem summarizes our discussion.  
\begin{theorem} \label{thm-skein}
Let $K_+, K_-,$ and $K_0$ be three oriented virtual knot or link diagrams obtained from the closures of braids $\be_+, \be_-,$ and $\be_0$, respectively, which are identical everywhere except near one crossing, where they are as pictured in Figure \ref{skein}. Then the normalized virtual Alexander polynomial satisfies the  skein formula
$$\wh H_{K_+}(s,t,q) - (st)\; \wh H_{K_-}(s,t,q) = \left(st - 1\right) \wh H_{K_0}(s,t,q).$$
Note that in the above formula, there is no indeterminacy in $\wh H_{K_+}, \wh H_{K_+},$ and $\wh H_{K_0}$  since they are all computed with respect to a given virtual knot or link diagram.

\end{theorem}


\section{Twisted Alexander invariants for virtual knots} \label{section-7}

In this section, we introduce twisted versions of virtual Alexander invariants for virtual knots and links.  We begin by recalling the definition of the twisted Alexander polynomial for finitely presented groups from \cite{Wada}.      

Given a group $\Ga$ with presentation 
$P= \langle \, a_1, \ldots, a_k \mid  r_1, \ldots, r_{\ell} \, \rangle$,
a surjection $\al \colon \Ga \to \ZZ^s$, and a representation $\varrho \colon \Ga \lto GL_n(R)$,
where $R$ is a unique factorization domain, we will construct 
the twisted Alexander polynomial $\Delta_{\Ga, \al}^{\varrho}(t_1, \ldots, t_s)$. 
Let $F_k$ denote the free group of rank $k$. 
The presentation $P$ determines a surjection $\phi \colon F_k \lto \Ga$, which induces a ring 
 homomorphism $\wt \phi \colon \ZZ F_k \lto \ZZ \Ga$ on the integral group-rings.  Likewise,
 the surjection $\al\colon \Ga \to \ZZ^s$ induces a ring homomorphism $ \wt \al \colon \ZZ \Ga \lto \ZZ[t^{\pm1}_1, \ldots, t^{\pm1}_s]$, the integral group-ring of $\ZZ^s$; and the representation $\varrho\colon \Ga \lto GL_n(R)$ induces a ring homomorphism $\wt \varrho \colon \ZZ \Ga \lto M_n(R)$, the algebra of $n \times n$ matrices over $R$.

Let $\cL_R= R[t^{\pm1}_1, \ldots, t^{\pm1}_s]$ be the ring of Laurent polynomials with coefficients in $R$.
The Jacobian  of the presentation $P$ is the $\ell \times k$ matrix $J=\left(\frac{\partial r_i}{\partial a_j}\right)$.
Let
$$
\Om = (\wt \varrho \, \otimes \, \wt \al) \circ \wt \phi \colon \ZZ F_k \lto M_n\left(\cL_R\right)  
$$
be the composition of $\wt \phi \colon \ZZ F_k \lto \ZZ \Ga$ with the  tensor product $\wt \varrho \, \otimes \, \wt \al \colon \ZZ \Ga \lto M_n\left(\cL_R\right)$.
The  twisted presentation matrix $M =\Om \left(\frac{\partial r_i}{\partial a_j}\right)$ is an $\ell \times k$ matrix with entries in $M_n(\cL_R)$, that is, a block matrix.

\begin{remark}\label{rho-tensor-alpha}
If $\sum m_i w_i \in \ZZ \Gamma$, where $m_i \in \ZZ$ and $w_i \in \Gamma$, then 
$$\wt \varrho \, \otimes \, \wt \al \left( \sum m_i w_i \right) = \sum m_i \, \varrho(w_i) \cdot \al(w_i)  $$
where $\al(w_i)$ acts by scalar multiplication on $\varrho(w_i)$.   
\end{remark}

Let $\wh{M}_j$ denote the matrix obtained by removing the $j$-th column from $M$, and we can regard $\wh{M}_j$ as an $n\ell \times n(k-1)$ matrix with entries in $\cL_R$.  If $\ell < k-1,$ then $\wh{M}_j$ has more columns than rows, and its determinant is zero.  If $\ell > k-1,$ then $\wh{M}_j$ has more rows than columns, and then we consider multi-indices $I$ of length $|I|=n(k-1)$ and use $\wh{M}_j^I$ to denote the square matrix consisting of the $I$ rows from $\wh{M}_j$.

The following result is essential to the construction of the twisted Alexander polynomials, and it is proved in \cite[Lemmas 2 and 3]{Wada}. 

\begin{lemma}\label{Nonzero}
\begin{enumerate}[\hspace{18pt}]
\item[(i)] For some $a_j$, $\det(\Om(1-a_j))$ is nonzero.
\item[(ii)] If $\det(\Om(1-a_j))$ and $\det(\Om(1-a_{k}))$ are both nonzero, then for any multi-index $I$ of length $|I| = n(k-1)$,
$$
\det({\wh M}_j^I)\det(\Om(1-a_{k})) = \pm \det({\wh M}_{k}^I)\det(\Om(1-a_j)).
$$  
The sign in the above formula is always positive if the representation $\varrho$ has even degree. 
\end{enumerate}
\end{lemma}

We consider the set of determinants $\{\det (\wh{M}_j^I) \mid |I|= n(k-1)\}$ taken over all possible indices $I$ and denote their greatest common divisor
by $Q_j(t_1, \ldots, t_s)$. (Note that because $R$ is a unique factorization domain, $\cL_R$ is too, hence $\cL_R$ is a gcd domain.) Thus $Q_j$ is well-defined up to a factor of $\ep t_1^{e_1} \cdots t_s^{e_s}$, where $\ep$ is a unit of $R$ and $e_1, \ldots, e_s\in \ZZ$.   
Lemma \ref{Nonzero} implies that the quotient
\begin{equation}\label{Deftwist}
\Delta_{\Ga, \al}^{\varrho}(t_1, \ldots, t_s) = \frac{Q_j(t_1, \ldots, t_s)}{\det(\Om(1-a_j))},
\end{equation}
is independent, up to possibly a sign, of the column $j$ chosen to remove from $M$, and Equation
\eqref{Deftwist} defines the twisted Alexander polynomial associated to the group $\Ga,$ presentation $P$, surjection $\al \colon \Ga \to \ZZ^s$ and representation $\varrho \colon \Ga\lto GL_n(R).$

Notice that $\Delta_{\Ga, \al}^{\varrho}(t_1, \ldots, t_s)$ is not generally a polynomial, but rather a quotient of two Laurent polynomials.  For virtual knots, we will show  that the twisted Alexander polynomial is a Laurent polynomial over $R$, see  Theorem \ref{tHispoly} below (see  \cite[Proposition 8]{Wada} for a similar but weaker result for classical knots, as well as  \cite[Theorem 3.1]{Kitano-Morifuji} by Kitano and Morifuji).

The definition \eqref{Deftwist} of the twisted Alexander polynomials requires a choice of presentation $P$ for the group $\Ga.$ 
The Tietze transformation theorem implies that any two presentations $P$ and $P'$ of $\Ga$ are related by a sequence of Tietze moves and their inverses:

\begin{itemize}
\item[{\bf T1:}] Add a consequence $r$ of the relators $r_i$.
\item[{\bf T2:}] Add a new generator $a$ and a new relator $aw^{-1}$, where $w$ is a word in the $a_j$.  
\end{itemize}

\noindent
Wada studies the effect of the Tietze transformations on the twisted Alexander polynomials,  and in \cite[Theorem 1]{Wada} he proves that $\Delta_{\Ga, \al}^{\varrho}(t_1, \ldots, t_s)$ is independent of the choice of presentation $P$ for $\Ga$, up to a factor of  $\ep t_1^{e_1} \cdots t_s^{e_s}$.

Given a virtual knot or link $K$, we define the twisted virtual Alexander polynomial by applying this construction to the virtual knot group $\VG_K$. We will show that the  twisted Alexander polynomials associated to representations of the virtual knot group $\VG_K$ enjoy some special properties by showing that the presentations of $\VG_K$ one gets from diagrams of $K$ are strongly Tietze equivalent. 
 
Recall that two presentations $P$ and $P'$ of a group are said to be \emph{strongly Tietze equivalent} if they are related by the following moves and their inverses:

\begin{itemize}
\item[{\bf S1:}] Replace one relator $r_i$ by its inverse $(r_i)^{-1}$.
\item[{\bf S2:}] Replace one relator $r_i$ by a conjugate $w r_i w^{-1}$.
\item[{\bf S3:}] Replace one relator $r_i$ by the product $r_i r_j$.
\item[{\bf S4:}]  Add a new generator $a$ and the relator $aw^{-1}$.
\end{itemize}
 
Note that {\bf S4} is the same as {\bf T2}.

Given a virtual knot diagram $D$ of $K$, one can construct a presentation $P_D$ of the virtual knot group $\VG_K$ as follows.  It has one meridional generator for each short arc and two relations for each classical and virtual crossing as in Figure \ref{Relations}. In addition, it has two auxilliary generators $s$ and $q$, along with the commutation relation $[s,q]=1.$ We call $P_D$ the\emph{Wirtinger-like} presentation associated with the diagram $D$ for $K$,
and the next result shows that any two such presentations are strongly Tietze equivalent.
The argument is similar to the proof of Lemma 6 in \cite{Wada}.

\begin{lemma}\label{WirtingertStrongTietze}
If $D$ and $D'$ are two diagrams for the virtual knot $K$, then the 
presentations $P_D$ and $P_{D'}$ of its virtual knot group $\VG_K$ are strongly Tietze equivalent.
\end{lemma}

\begin{proof}
Let $D$ and $D'$ be two diagrams of $K$ that are identical outside of the small neighborhood indicated in the left of Figure \ref{Reidemeister1}.
Let $P$ and $P'$ be the corresponding presentations
$$
P = \langle a_1, \ldots a_{k},x, s,q \mid  r_1, \ldots, r_{k},xs^{-1}a_k^{-1}s, [s,q]  \rangle
$$ 
and  
$$
P' = \langle a_1, \ldots a_{k},s,q \mid  r_1', \ldots, r_{k}', [s,q]  \rangle.
$$
The relator $r_i'$ is obtained from $r_i$ by replacing occurrences of $x$ by $s^{-1}a_ks$.

We can transform $r_i$ to $r_i'$ using strong Tietze moves as follows. Suppose that $r_i$ contains an occurrence of $x$.  We invert $r_i$ so that it now contains an occurrence of $x^{-1}$.  Let $r_i = u x^{-1} v$ where $u$ and $v$ are words in the generators of $P$.  We make the following sequence of transformations using moves {\bf S2} and {\bf S3}
$$
r_i \leadsto vr_iv^{-1} \leadsto vr_iv^{-1}xs^{-1}a_k^{-1}s \leadsto v^{-1}vr_iv^{-1}xs^{-1}a_k^{-1}sv = us^{-1}a_k^{-1}sv .
$$ 

If $D$ and $D'$ are as indicated in the right of Figure \ref{Reidemeister1}, then we need to replace occurrences of $x$ by $q^{-1}a_kq$. This can be done using a sequence of moves similar to the above sequence.  

This shows that the first two generalized Reidemeister moves $(r1)$ and $(v1)$ can be realized by a sequence of strong Tietze moves.
The other generalized Reidemeister moves $(r2), (v2), (r3), (v3),$ and $(v4)$ of Figure \ref{VRM} can be similarly realized by the strong Tietze move {\bf S4}, and this completes the argument. 
\end{proof}

\begin{figure}[ht]
\centering
\def\svgwidth{14cm}
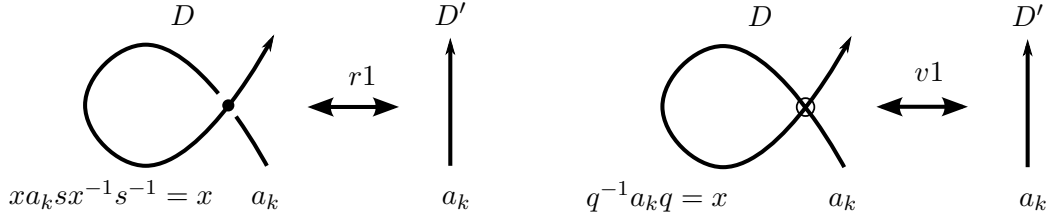
\caption{The first Reidemeister move.}
\label{Reidemeister1}
\end{figure}

Associated to any Wirtinger-like presentation 
$$P=\langle a_1,\ldots, a_k, s,q \mid r_1, \ldots, r_k, [s,q]\rangle$$
is a surjection $\al \colon \VG_K \to \ZZ^3$, where $\ZZ^3$ is generated by $\{s,t,q\}$.  The homomorphism $\al$ sends each meridional generator
$a_j$ to $t$ and the other generators $s,q$ identically to themselves. 
\begin{definition}\label{DefTwisted}
Let $K$ be a virtual knot or link and  $\varrho \colon \VG_K \lto GL_n(R)$ a representation of the virtual knot group. Then the {\it twisted virtual Alexander polynomial} is defined to be $H_K^{\varrho}(s,t,q) = \Delta_{\VG_K, \al}^{\varrho}(s,t,q)$. 
\end{definition}

Note that if $K$ is classical, then $H_K^{\varrho}(s,t,q) =0.$ Thus the twisted virtual Alexander polynomial is an obstruction to $K$ being classical.

Given a representation $\varrho \colon \VG_K \lto GL_n(R)$, we define the twisted Alexander matrix $A = \Om \left(\frac{\partial r_i}{\partial a_j}\right)$ to be the $k \times k$ block of the twisted presentation matrix $M$ associated to the above presentation $P$.

\begin{theorem} \label{tHispoly}
For a virtual knot or link $K$ and representation $\varrho \colon \VG_K \lto GL_n(R)$, the twisted virtual Alexander polynomial satisfies 
$H_K^{\varrho}(s,t,q) = \det(A)$.  Thus $H_K^{\varrho}(s,t,q)$ is a Laurent polynomial with coefficients in $R$.
\end{theorem}

\begin{proof}
Using the above presentation $P$ of  
the virtual knot group  $\VG_K$, we obtain the twisted presentation matrix
$$
M = \begin{bmatrix} A & \Om \left( \frac{\partial r}{\partial s} \right) & \Om \left(\frac{\partial r}{\partial q}\right) \\ {\bf 0} & \Om(1-q) & \Om(s-1) \end{bmatrix}.
$$
Clearly $\det(\Om(1-q))$ is non-zero, and since ${\wh M}_{n+2}$ is square, equation \eqref{Deftwist} implies that
\begin{equation} \label{eqnone}
\Delta_{\VG_K, \al}^{\varrho}(s,t,q) = \frac{\det({\wh M}_{n+2})}{\det(\Om(1-q))}.
\end{equation}
Since ${\wh M}_{n+2}$ is block upper triangular, we have $\det({\wh M}_{n+2}) = \det(A)  \det(\Om(1-q))$. Now equation \eqref{eqnone} implies that $\Delta_{\VG_K, \al}^{\varrho}(s,t,q) = \det(A)$, and the statement of the theorem follows.
\end{proof}

We say that two representations $\varrho_1, \varrho_2 \colon \VG_K \to GL_n(R)$ are conjugate if there exists $P \in GL_n(R)$ such that $\varrho_1(g)= P \varrho_2(g) P^{-1}$
 for all $g \in \VG_K$. It is straightforward to verify the following elementary result. 
 
\begin{lemma}   If $\varrho_1\colon \VG_K \to GL_n(R)$ and $\varrho_2\colon \VG_K \to GL_n(R)$ are conjugate representations, then
$H_K^{\varrho_1}(s,t,q) =H_K^{\varrho_2}(s,t,q).$
 \end{lemma}

In \cite[Theorem 2]{Wada}, Wada shows that the twisted Alexander polynomials of classical knots have less indeterminacy than the invariants for arbitrary groups. The next theorem proves the same result for virtual knots, and we include a proof for the convenience of the reader. 

\begin{theorem} \label{twisted-def}
Let $\varrho\colon \VG_K \lto GL_n(R)$ be a representation. As an invariant of the oriented virtual link type of $K$, the twisted virtual Alexander polynomial  
$H_K^{\varrho}(s,t,q)$
is well-defined up to a factor of $\ep (s^{e_1}t^{e_2}q^{e_3})^n$, where $\ep$ is a unit of $R$ and $e_1, e_2, e_3 \in \ZZ$.

If $\varrho\colon \VG_K \lto SL_n(R)$ is unimodular, then the twisted virtual Alexander polynomial $H_K^{\varrho}(s,t,q)$ is well-defined up to a factor of  $\pm (s^{e_1}t^{e_2}q^{e_3})^n$ if $n$ is odd and 
$(s^{e_1}t^{e_2}q^{e_3})^n$ if $n$ is even. 

\end{theorem}

\begin{proof}
The oriented virtual knot type of $K$ consists of the set of oriented virtual knot diagrams of $K$ modulo the generalized Reidemesiter moves, see Figure \ref{VRM}.  Each knot diagram of $K$ is related by a sequence of strong Tietze moves as was shown in Lemma \ref{WirtingertStrongTietze}.

Let $K$ have the following Wirtinger-like presentation of its virtual knot group
$\VG_K =\langle a_1, \ldots a_k, s,q \mid  r_1, \ldots, r_k, [s,q]  \rangle$ and let 
$$
M = \begin{bmatrix} A & \Om \left( \frac{\partial r}{\partial s} \right) & \Om \left(\frac{\partial r}{\partial q}\right) \\ {\bf 0} & \Om(1-q) & \Om(s-1) \end{bmatrix}
$$ 
be the associated twisted presentation matrix.

Under the strong Tietze move {\bf S1}, we replace a relator $r_i$ by its inverse $r_i^{-1}$.  The corresponding twisted presentation matrix $M'$ 
is then obtained from $M$ by replacing the matrix row $\Om\left(\frac{\partial r_i}{\partial a_j}\right)$ with $\Om\left(-\frac{\partial r_i}{\partial a_j}\right)$.  If $A'$ is the twisted Alexander matrix of $M'$, then $\det (A') = (\pm 1)^n \det(A)$.  

Under the move {\bf S2}, we replace a relator $r_i$ by a conjugate $w r_i w^{-1}$. Since
$\frac{\partial w^{-1}}{\partial a_j} = - w^{-1} \frac{\partial w}{\partial a_j},$ we see
that
\begin{equation*}
\begin{split}
\frac{\partial(w r_i w^{-1})}{\partial a_j} &=\frac{\partial w}{\partial a_j} + w \frac{\partial r_i}{\partial a_j} + w r_i \frac{\partial w^{-1}}{\partial a_j} \\
&=  (1-wr_i w^{-1})\frac{\partial w}{\partial a_j} + w \frac{\partial r_i}{\partial a_j}.
\end{split}
\end{equation*}
Applying $\wt \varrho \otimes \wt \al$ and using the fact that $\varrho(w r_i w^{-1}) = I_n$, it follows that the twisted Alexander matrix $M'$ is obtained from $M$ by replacing the matrix row $\Om\left(\frac{\partial r_i}{\partial a_j}\right)$ by $\Om(w) \Om\left(\frac{\partial r_i}{\partial a_j}\right)$.    We  compute that 
$$
\det(\Om(w)) = \det(\varrho (w)) \det(\al(w))^n = \ep (s^{e_1}t^{e_2}q^{e_3})^n. 
$$
If $\varrho$ is unimodular, then $\det(\varrho (w))= 1.$  

Under the  move {\bf S3}, we replace a relator $r_i$ by the product $r_i r_k$.  The twisted presentation matrix  $M'$ is then obtained from $M$ by replacing the matrix row $\Om \left(\frac{\partial r_i}{\partial a_j}\right)$ by $\Om\left(\frac{\partial r_i}{\partial a_j}\right) + \Om(r_i)\Om\left(\frac{\partial r_k}{\partial a_j}\right)$.  Since each row of $A'$ is a linear combination of the rows of $A$, we have that $\det(A')= \det(A)$.

Notice that the last strong Tietze move {\bf S4} coincides with the Tietze move {\bf T2}.  In proof  of \cite[Theorem 1]{Wada}, Wada shows that the twisted Alexander polynomial of a finitely presentable group $\Ga$ is independent of second Tietze move {\bf T2}. More precisely, if $P'$ is obtained from $P$ by the move {\bf T2}, and if $M$ and $M'$ are the respective twisted presentation matrices, then Wada shows that $\det(M_j^I)=\det(M_j^{'J})$. 
\end{proof}

As in the untwisted case, we suppress the indeterminacy of the twisted virtual Alexander polynomial and simply write $H^\varrho_K(s,t,q) = f$ for any $f \in \cL_R$ if they are equal up to the indeterminacy indicated in Theorem \ref{twisted-def}.

The following result shows that twisted virtual Alexander polynomial also carries virtual crossing number information, compare with Theorem \ref{v-bound}.

\begin{theorem} \label{twist-v-bound}
If $K$ is a virtual knot or link, $\varrho \colon \VG_K \lto GL_n(R)$ is a representation and $H_K^{\varrho}(s,t,q)$ is its twisted virtual Alexander polynomial, then  
$$\text{$q$-$\width$}\; H_K^{\varrho}(s,t,q) \leq 2n \, v(K).$$
\end{theorem}
\begin{proof}
Suppose that $K$ admits a diagram with $m$ virtual crossings.  Then the corresponding 
twisted virtual Alexander matrix $A$ has $2n\cdot m$ rows with a $q$ or $q^{-1}$.  It follows that after a suitable normalization, the $q$-$\width$ of $\det (A)$ is at most $2n m$. 
\end{proof}

In Section \ref{section-3} we used the virtual Alexander polynomial to give lower bounds on the virtual crossing number $v(K)$ for several virtual knots. We will now show how to improve those bounds using twisted Alexander polynomials associated with representations
$\varrho \colon \VG_K \lto SL_2(\FF_2)$, where $\FF_2$ denotes the field with two elements. The computations in the following examples were performed using {\tt sage} \cite{Sage}.

\begin{example}
Consider the virtual knot $K=3.4$. A simple computation shows that
$H_K(s,t,q) = (s - q)(1-st)(1-qt)$. Since $H_K(s,t,q)$ has $q$-$\width$ equal to 2, Theorem \ref{v-bound} implies that $v(K) \geq 1.$ This bound is not sharp.

Using the twisted virtual Alexander polynomials, we can obtain a sharp bound for $v(K)$.
From the diagram of $K$ in Figure \ref{3-4}, we find the following presentation for $\VG_K$, which can be simplified to a presentation with only one meridional generator as follows:
\begin{eqnarray*}
\VG_{K} &=& \langle a,b,c,s,q   \mid  [s,q], \; a^{-1}q s^{-2}a^{-1}s^2q^{-2}cq^2s^{-1}asq^{-1},
b^{-1}s^{-1}a^{-1}q^2s^{-2}as^3q^{-2}a, \\
&& \hspace{3.4cm} c^{-1}q^2s^{-3}a^{-1}s^3q^{-3}bq^3s^{-2}as^2q^{-2} \rangle, \\
&=& \langle a,s, q   \mid  [s,q], \; a^{-1}qs^{-2}a^{-1}s^{-1}a^{-1}s^2q^{-3}a^{-1}q^2s^{-2}as^3q^{-2}aq^3s^{-2}asasq^{-1}  \rangle. 
\end{eqnarray*}
Setting $r = a^{-1}qs^{-2}a^{-1}s^{-1}a^{-1}s^2q^{-3}a^{-1}q^2s^{-2}as^3q^{-2}aq^3s^{-2} asasq^{-1},$ we  compute
the Fox derivative 
\begin{eqnarray*}
\frac{\partial r}{\partial a} &=& 
-a^{-1} - a^{-1}qs^{-2}a^{-1} - a^{-1}qs^{-1}(s^{-1}a^{-1})^2 - a^{-1}qs^{-1}(s^{-1}a^{-1})^2s^2q^{-3}a^{-1}  \\ 
&&+\; a^{-1}qs^{-1}(s^{-1}a^{-1})^2s^2q^{-3}a^{-1}q^2s^{-2}
+ a^{-1}qs^{-1}(s^{-1}a^{-1})^2s^2q^{-3}a^{-1}q^2s^{-2}as^3q^{-2} \\
&&+ \; a^{-1}qs^{-1}(s^{-1}a^{-1})^2s^2q^{-3}a^{-1}q^2s^{-2}as^3q^{-2}aq^3s^{-2} \\
&&+ \; a^{-1}qs^{-1}(s^{-1}a^{-1})^2s^2q^{-3}a^{-1}q^2s^{-2}as^3q^{-2}aq^3s^{-2}as.
\end{eqnarray*}

Let $\varrho \colon \VG_K \lto SL_2(\FF_2)$ be the representation defined by setting
$$
a=
\begin{bmatrix}
0 & 1 \\
1 & 0
\end{bmatrix}, \; 
s=
\begin{bmatrix}
0 & 1 \\
1 & 1
\end{bmatrix}, \; 
q=
\begin{bmatrix}
0 & 1 \\
1 & 1
\end{bmatrix},
$$
we compute the determinant of $\Om\left(\frac{\partial r}{\partial a}\right)$ and deduce that the twisted virtual  Alexander polynomial is given by

$$
H_K^{\varrho}(s,t,q)= (s-q)^2(tq-1)^2(st-1)^2(stq-s-q)^2.$$
   
Since $H_K^{\varrho}(s,t,q)$ has $q$-$\width=6$, it follows that $v(K) \geq 2$. Comparing to Figure \ref{3-4}, we see this bound is sharp.

\begin{figure}[ht]
\centering
\includegraphics[scale=1.0]{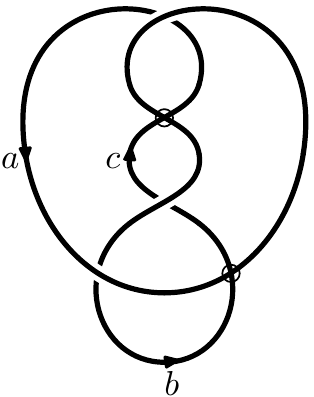}  
\caption{The virtual knot $3.4$ has $v(K)=2$.}
\label{3-4}
\end{figure}
\end{example}

\begin{example}

Let $K$ be the virtual knot $4.71$.  A straightforward  computation shows that $H_K(s,t,q)=0$.
From the diagram of $K$ in Figure \ref{4-71}, we derive the following presentation for $\VG_K$:
\begin{eqnarray*}  \VG_{K} &=& 
\langle a,b,s, q \mid [s,q], \; 
a^{-1}bs^{-1}q^2a q^{-2}s^2 b^{-1}as^{-2}q^2b^{-1}q^{-2}s, \\
&& \hspace{3.2cm} b a^{-1}q^{-2}s^2 b s^{-1} q^2 a b^{-1} q^{-2} s a^{-1}s^{-2}q^2 \rangle.
\end{eqnarray*}
Setting 
\begin{eqnarray*}
r_1&=&a^{-1}bs^{-1}q^2aq^{-2}s^2b^{-1}as^{-2}q^2b^{-1}q^{-2}s, \\
r_2&=& ba^{-1}q^{-2}s^2bs^{-1}q^2ab^{-1}q^{-2}sa^{-1}s^{-2}q^2,
\end{eqnarray*}
 we compute
the Fox derivatives 
\begin{eqnarray*}
\frac{\partial r_1}{\partial a} 
&=& 
-a^{-1} + a^{-1}bs^{-1}q^2 + a^{-1}bs^{-1}q^2aq^{-2}s^2b^{-1},  \\
\frac{\partial r_1}{\partial b} &=& a^{-1} - a^{-1}bs^{-1}q^2aq^{-2}s^2b^{-1} - a^{-1}bs^{-1}q^2aq^{-2}s^2b^{-1}as^{-2}q^2b^{-1},\\
\frac{\partial r_2}{\partial a} &=&
-ba^{-1} + ba^{-1}q^{-2}s^2bs^{-1}q^2 - ba^{-1}q^{-2}s^2bs^{-1}q^2ab^{-1}q^{-2}sa^{-1}, \\ 
\frac{\partial r_2}{\partial b} &=&
1 + ba^{-1}q^{-2}s^2 - ba^{-1}q^{-2}s^2bs^{-1}q^2ab^{-1}.
\end{eqnarray*}

Let $\varrho \colon \VG_K \lto SL_2(\FF_2)$ be the representation defined by setting
$$
a=
\begin{bmatrix}
0 & 1 \\
1 & 0
\end{bmatrix}, \; 
b=
\begin{bmatrix}
1 & 0 \\
1 & 1
\end{bmatrix}, \; 
s=
\begin{bmatrix}
0 & 1 \\
1 & 1
\end{bmatrix}, \; 
q=
\begin{bmatrix}
1 & 0 \\
0 & 1
\end{bmatrix}.  
$$
Applying $\Om$ to the twisted Alexander matrix,
we compute that the twisted virtual  Alexander polynomial is given by
$$H_K^{\varrho}(s,t,q)= (1-st)^2 (s^2 + sq + q^2)^2  (t q^2 - s)^2.$$
 Since $H_K^{\varrho}(s,t,q)$ has $q$-$\width =8$, it follows that $v(K) \geq 2$. Since $K$ has a diagram with three virtual crossings, we conclude that $v(K)=2$ or 3.
\begin{figure}[ht]
\centering
\includegraphics[scale=1.0]{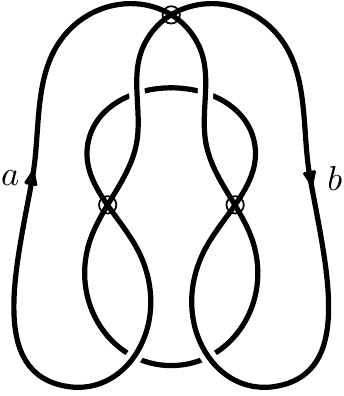}  
\caption{The virtual knot $4.71$.}
\label{4-71}
\end{figure}

\end{example}

\begin{example}
Let $K$ be the virtual knot $4.98$.  An easy computation shows that $H_K(s,t,q)=0$.
From the diagram of $K$ in Figure \ref{4-98}, we derive the following presentation for $\VG_K$:
\begin{eqnarray*}  
\VG_{K} &=& \langle a,b,s, q \mid [s,q], \; 
q^{-2}a^{-1}sbaq^4s^{-3}as^4q^{-4}a^{-1}b^{-1}s^{-1}aq^2s^{-1}b^{-1}, \\   
&& \hspace{3cm} a^{-1}sbaq^2s^{-2}a^{-1}s^3q^{-4}a^{-1}sbaq^4s^{-3}as^2q^{-2}a^{-1}b^{-1}s^{-2}  \rangle. 
\end{eqnarray*}
Setting 
\begin{eqnarray*}
r_1&=&q^{-2}a^{-1}sbaq^4s^{-3}as^4q^{-4}a^{-1}b^{-1}s^{-1}aq^2s^{-1}b^{-1}, \\
r_2&=& a^{-1}sbaq^2s^{-2}a^{-1}s^3q^{-4}a^{-1}sbaq^4s^{-3}as^2q^{-2}a^{-1}b^{-1}s^{-2},
\end{eqnarray*}
 we compute
the Fox derivatives 
\begin{eqnarray*}
\frac{\partial r_1}{\partial a} \!\!\!
&=& \!\!\!
-q^{-2}a^{-1} + q^{-2}a^{-1}sb + q^{-2}a^{-1}sbaq^4s^{-3} - q^{-2}a^{-1}sbaq^4s^{-3}as^4q^{-4}a^{-1} \\
&&+ q^{-2}a^{-1}sbaq^4s^{-3}as^4q^{-4}a^{-1}b^{-1}s^{-1},  \\
\frac{\partial r_1}{\partial b} 
\!\!\!&=& \!\!\! q^{-2}a^{-1}s - q^{-2}a^{-1}sbaq^4s^{-3}as^4q^{-4}a^{-1}b^{-1} \\ 
&&- q^{-2}a^{-1}sbaq^4s^{-3}as^4q^{-4}a^{-1}b^{-1}s^{-1}aq^2s^{-1}b^{-1},\\
\frac{\partial r_2}{\partial a} \!\!\!&=&\!\!\!
-a^{-1} + a^{-1}sb - a^{-1}sbaq^2s^{-2}a^{-1} - a^{-1}sbaq^2s^{-2}a^{-1}s^3q^{-4}a^{-1}  \\
&& + a^{-1}sbaq^2s^{-2}a^{-1}s^3q^{-4}a^{-1}sb + a^{-1}sbaq^2s^{-2}a^{-1}s^3q^{-4}a^{-1}sbaq^4s^{-3} \\
&&- a^{-1}sbaq^2s^{-2}a^{-1}s^3q^{-4}a^{-1}sbaq^4s^{-3}as^2q^{-2}a^{-1}, \\
\frac{\partial r_2}{\partial b} \!\!\!&=&\!\!\!
a^{-1}s + a^{-1}sbaq^2s^{-2}a^{-1}s^3q^{-4}a^{-1}s - a^{-1}sbaq^2s^{-2}a^{-1}s^3q^{-4}a^{-1}sbaq^4s^{-3}as^2q^{-2}a^{-1}b^{-1}.
\end{eqnarray*}

Let $\varrho \colon \VG_K \lto SL_2(\FF_2)$ be the representation defined by setting
$$
a=
\begin{bmatrix}
0 & 1 \\
1 & 0
\end{bmatrix}, \; 
b=
\begin{bmatrix}
1 & 0 \\
1 & 1
\end{bmatrix}, \; 
s=
\begin{bmatrix}
1 & 0 \\
0 & 1
\end{bmatrix}, \; 
q=
\begin{bmatrix}
1 & 0 \\
0 & 1
\end{bmatrix}.  
$$

\begin{figure}[ht]
\centering
\includegraphics[scale=1.0]{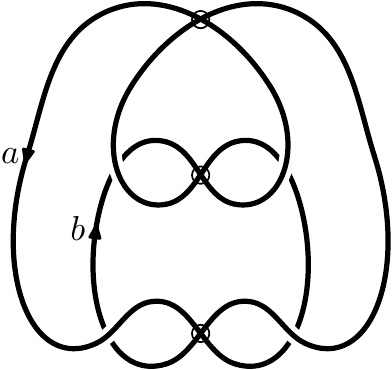}  
\caption{The virtual knot $4.98$.}
\label{4-98}
\end{figure}
\end{example}

Applying $\Om$ to the twisted Alexander matrix,
we compute that the twisted virtual  Alexander polynomial is given by
$$H_K^{\varrho}(s,t,q)= (s - q)^4  (1-s t)^2  (s-t q^2)^2.$$
 
Since $H_K^{\varrho}(s,t,q)$ has $q$-$\width =8$, it follows that $v(K) \geq 2$. Since $K$ has a diagram with three virtual crossings, we conclude that  $v(K)=2$ or 3.
 

\section{Twisted virtual Alexander polynomials via braids}  \label{section-8}
In this section, we relate the twisted Alexander polynomial of a virtual knot or link to the twisted virtual Burau representation, and we use this to establish twisted analogues of the results from Section \ref{section-4}.
Throughout this section, we assume $R$ is a unique factorization domain and 
 $\cL_R = R[s^{\pm 1}, t^{\pm 1}, q^{\pm 1}]$ will denote the ring of Laurent polynomials over $R$.

Let $K$ be a virtual knot or link obtained as the closure of a virtual braid $\be \in \VB_k$.  
Recall that the virtual knot group of $K$ has a presentation
\[
\VG_\be = \langle x_1, \ldots, x_k, s,q \mid  x_1=x_1^\be, \ldots, x_k=x_k^\be, [s,q]=1 \rangle,
\]
see \eqref{present2}.
Here, the notation $\VG_\be$ indicates that the virtual knot group $\VG_K$ is equipped with the above presentation.
A representation ${\varrho} \colon \VG_\be \to GL_n(R)$ extends
to a representation 
$\varrho \colon  F_{k+2} = \langle x_1, \ldots, x_k, s,q \rangle ~ \to~ GL_n(R)$,
where we reuse the symbol $\varrho$ to avoid an excess of notation,
given by composition with the homomorphism $\pi_\be \colon F_{k+2} \to \VG_\be$ that takes a generator to a generator of the same name.

The first result is the twisted analogue of Theorem \ref{Burau}, 
interpreting the twisted virtual Alexander polynomial in terms of the twisted Burau representation.

\begin{theorem} \label{thm:twisted-Burau}
Let  $\be \in \VB_{k}$ be a virtual braid and  let $K$  be the virtual knot or link obtained from the closure of $\be$.
Let ${\varrho} \colon \VG_\be \lto GL_n(R)$ be a representation.
Then the virtual twisted Alexander polynomial of $K$ as in Definition {\rm \ref{DefTwisted}} is given by
\[
H^{\varrho}_K(s,t,q) = \det\left(\Psi_{\varrho}(\be)-I_{nk}\right)
\]
where $\Psi_{\varrho}$ is the ${\varrho}$-twisted Burau representation as in Definition {\rm \ref{def_twisted_virtual_Burau_rep}}.
\end{theorem}

\begin{proof}
It is not difficult to verify that the given presentation of $\VG_\be$ is strongly Tietze equivalent to the Wirtinger-like presentation obtained from the diagram of $K = \wh \be$ as the braid closure. In fact, this can be achieved using exclusively {\bf S4} moves.  The conclusion then follows  from Theorem \ref{tHispoly}. 
 \end{proof}

Define the polynomial  invariant  $H^\varrho_\be(s,t,q) \in \cL_R$ of a virtual braid $\be \in \VB_k$ and an arbitrary representation $\varrho \colon F_{k+2} \to GL_n(R)$ by
\begin{equation}\label{eq-twist-H}
H^\varrho_\be(s,t,q) = \det \left( \Psi_{\varrho}(\be)-I_{nk}\right).
\end{equation}

We will use this approach to extend certain results to the twisted setting.  The following observation, which can be verified by direct calculation, shows that the images of the generators $\si_{i}, \, \si_{i}^{-1}, \, \tau_i  \, \in \VB_k$ under  
$\Psi_{\varrho}$ are given by
the following $nk \times nk$ matrices:

\begin{equation} \label{eq:twisted-Burau}
\begin{split}
\Psi_{\varrho}(\si_{i}) &=
\begin{bmatrix}
I_{n(i-1)} &  & \cdots& {\bf 0}\\
  & 0 & \varrho(s)s & \vdots  \\
 \vdots& \varrho(x_{i+1}) t  & I_n-\varrho(x_{i+1} x_{i} sx_{i+1}^{-1})st &    \\[0.6em]
{\bf 0} & \cdots &  & \hspace{0.3em} I_{n(k-i-1)} 
\end{bmatrix}, \\[0.5em] 
\Psi_{\varrho}(\si_{i}^{-1}) &=
\begin{bmatrix}
I_{n(i-1)} &\hspace{0.3em} & \cdots  & {\bf 0} \\
 &\hspace{0.3em}  \varrho(s^{-1}x_i^{-1}s x_{i+1}) -\varrho(s^{-1}x_i^{-1}) s^{-1}t^{-1} & \varrho(s^{-1}x_i^{-1}s) t^{-1} & \vdots\\
\vdots & \hspace{0.3em}\varrho(s^{-1}) s^{-1} & 0  &  \\[0.4em]
{\bf 0} & \hspace{0.3em} \cdots &  & \hspace{0.3em} I_{n(k-i-1)} 
\end{bmatrix},  \\[0.5em]
\Psi_{\varrho}(\tau_{i}) &=
\begin{bmatrix}
I_{n(i-1)} &  &  \cdots & {\bf 0}\\
 & 0 & \varrho(q) q & \vdots\\
\vdots & \varrho(q^{-1}) q^{-1} & 0  & \\[0.6em]
{\bf 0} & \cdots &  & \hspace{0.3em} I_{n(k-i-1)} 
\end{bmatrix}. 
\end{split}
\end{equation}

Recall from Section \ref{section-4} that for $\varrho \colon F_{k+2} \to GL_n(R)$
and $\be \in \VB_k$, we write $\be_*\varrho$  for
$\Phi(\be)_*\varrho$, where $\Phi \colon F_{k+2} \to \Aut(F_{k+2} \, \rel \, \{s,q\})$ is the fundamental representation.
We will see how to determine $\Psi_{\be_* \varrho}(g_i)$ for a generator $g_i$ of $\VB_k$.

\begin{remark} \label{rem-mult}
Suppose $g_i \in \{\si_i, \si_i^{- 1}, \tau_i\}$ is a generator of $\VB_k$. By equation \eqref{psipropertyone}, it follows that $\Psi_\varrho(g_i \be) = \Psi_{\be_* \varrho}(g_i) \Psi_\varrho(\be)$. 
Then $\Psi_{\be_* \varrho}(g_i)$is obtained by making a small modification to
the matrices in \eqref{eq:twisted-Burau}.

For instance, if $g_i = \si_i$, then since $\be_* \varrho(x_i) = \varrho(x_i^\be)$,  \eqref{eq:twisted-Burau} implies that
$$\Psi_{\be_* \varrho}(\si_i) = 
\begin{bmatrix}
I_{n(i-1)} &  & \cdots& {\bf 0}\\
  & 0 & \varrho(s)s & \vdots  \\
 \vdots& \varrho(x^\be_{i+1}) t  & I_n-\varrho(x^\be_{i+1} x_{i}^\be s (x_{i+1}^\be)^{-1})st &    \\[0.4em]
{\bf 0} & \cdots &  & \hspace{0.3em} I_{n(k-i-1)} 
\end{bmatrix}.$$
A similar argument leads to corresponding formulas for $\Psi_{\be_* \varrho}(\si_i^{-1})$ and $\Psi_{\be_* \varrho}(\tau_i),$ though in the latter case, note that since $\be_*(q)=q,$ we have $\Psi_{\be_* \varrho}(\tau_i) = \Psi_\varrho(\tau_i).$
When applied inductively, these formulas determine $\Psi_\varrho(\be)$ for any braid $\be$ once it is written as a word in the generators.
\end{remark}


Next, we prove a twisted analogue of Theorem \ref{theorem-H-K}.

\begin{theorem} \label{theorem-H-K-twisted}
Let $K$ be a virtual knot or link and  $\varrho \colon \VG_K \lto GL_n(R)$ a representation. Then the twisted virtual Alexander polynomial
satisfies $H^\varrho_K(s,t,q) = H^\varrho_K(sq^{-1}, tq, 1)$.
\end{theorem}
\begin{proof}
Let $\be \in \VB_k$ be a virtual braid with $K = \wh \be$, and
let $\varrho \colon F_{k+2} \to GL_n(R)$ a representation satisfying $\be_* \varrho = \varrho$.
If 
$\be = \th_1 \cdots \th_\ell$, where $\th_j \in \{ \si_i^{\pm 1},  \tau_i  \mid   1\leq i \leq k-1\}$, 
then we set $| \be|=\ell,$ which is the length of $\be$ as a word in the standard generators. 
Consider the diagonal  $kn \times kn$ block matrix
$$\La= \begin{bmatrix} I_n &&\cdots& 0 \\ & qI_n&& \vdots  \\ \vdots && \ddots \\ 0 &\cdots&& q^{k-1}I_n \end{bmatrix}.$$  
We will argue by induction on $|\be|$ that 
\begin{equation} \label{induction}
\La \Psi_\varrho(\be) \La^{-1}(s,t,q)=\Psi_\varrho(\be)(sq^{-1}, tq, 1).
\end{equation}
If $|\be|=1$, 
then $\be=\th_1$ where $\th_1 \in \{ \si_i^{\pm 1},  \tau_i  \mid   1\leq i \leq k-1\}$.
Using equation \eqref{eq:twisted-Burau}, it is easy to check that
$$\La \Psi_\varrho(\th_1) \La^{-1}(s,t,q)=\Psi_\varrho(\th_1)(sq^{-1}, tq, 1).$$
This proves the base case of our induction.  Using Remark \ref{rem-mult}, one can similarly check that 
$$\La \Psi_{\be_* \varrho}(\th_1) \La^{-1}(s,t,q)=\Psi_{\be_* \varrho}(\th_1)(sq^{-1}, tq, 1).$$
This observation will be used in the inductive step of the argument below.

Now suppose, by induction, that equation \eqref{induction} holds for all braids $\be$ of length $\ell$.
Given a braid $\be'$ of length $\ell+1$, then we can write it as $\be'=\th_1 \be$, where $| \be | = \ell.$  By the inductive hypothesis and the above observation, it follows that
\begin{eqnarray*}
\La \Psi_\varrho(\be') \La^{-1}(s,t,q) 
&=& \La\Psi_{\be_* \varrho}(\th_1)\La^{-1}(s, t, q)\; \La \Psi_\varrho(\be)\La^{-1}(s,t,q)\\
&=& \Psi_{\be_* \varrho}(\th_1)(sq^{-1}, tq, 1) \; \Psi_\varrho(\be)(sq^{-1},tq,1) \\
&=& \Psi_\varrho(\be)(sq^{-1},tq,1).
\end{eqnarray*}
This completes the proof of \eqref{induction}, and
this shows that
\begin{eqnarray*}
H_\be^\varrho(s,t,q) 
&=& \det\left(\Psi_\varrho(\be)-I_{kn}\right) \\
&=& \det\left(\La \Psi_\varrho(\be)\La^{-1}-I_{kn}\right)  \\
&=& H_\be^\varrho(sq^{-1}, tq, 1).
\end{eqnarray*}                                         
It follows that $H_K^\varrho(s, t, q)= H_K^\varrho(sq^{-1}, tq, 1).$
\end{proof}

Just as in Remark \ref{rem-identity}, one can easily see that $H^\varrho_K(s,t,q) = H^\varrho_K(1, ts, qs^{-1})= H^\varrho_K(st, 1, qt)$. 

We now prove a twisted analogue of Lemma \ref{labelthislemma}. The proof requires us to assume that $\varrho(s) = \varrho(q)=I,$ but our computations suggest this result holds under the weaker assumption that $\varrho(s) = \varrho(q)$.
\begin{lemma}
\label{s-q-divides-twisted-lemma}
Let $K$ be a virtual knot or link and  $\varrho \colon \VG_K \lto GL_n(R)$ a representation of the virtual knot group such that $\varrho(s)=\varrho(q)=I_n$.
Then $H_K^{\varrho}(s,t,1)$ is divisible by $(s-1)^n$.
\end{lemma}


\begin{proof} Using a virtual knot diagram, we obtain a Wirtinger-like presentation of the virtual knot group $\VG_K =\langle a_1,\ldots, a_k, s,q \mid r_1, \ldots, r_k, [s,q] \rangle$. Let
$$
M = \begin{bmatrix} A & \Om \left( \frac{\partial r}{\partial s} \right) & \Om \left(\frac{\partial r}{\partial q}\right) \\ {\bf 0} & \Om(1-q) & \Om(s-1) \end{bmatrix}
$$ 
be the associated twisted presentation matrix and let
$A_1, \ldots, A_k$ be the block columns of twisted Alexander matrix $A$.
Applying $\Omega$ to the fundamental identity \eqref{fundamental-Fox-two}, after replacing $n$ by $k$ in \eqref{fundamental-Fox-two}, yields
\begin{equation}\label{fundamental-Fox-two-twisted}
\sum^k_{j=1}
 A_j\Omega(a_j-1) +
\Omega\left( \frac{\partial r^{}}{\partial s_{}} \right)\Omega(s - 1) +
\Omega \left( \frac{\partial r}{\partial q} \right)\Omega(q - 1) = {\bf 0}.
\end{equation}

For $1 \leq j \leq k$, let $B_j = \Omega\left(a_j -1\right) = \varrho(a_j)\, t - I_n$. 
Also let
$$
A' = A ~\begin{bmatrix}
B_1    &0 &\cdots &  { 0} \\ 
B_2 & I_n && \vdots  \\
\vdots& \vdots & \ddots &0 \\[0.3em]
B_k    &0   & \cdots&I_{n}
\end{bmatrix}.
$$
Observe that $\det(A') = \det(A)\det(B_1)= H_K^\varrho(s,t,q) \det \left(  \varrho(a_1)\, t - I_n \right)$.
We have,
$$
\Omega(s-1)=\varrho(s)s-I_n= (s-1) I_n
~\text{ and }~
\Omega(q-1)=\varrho(q)q-I_n = (q-1) I_n.
$$
By (\ref{fundamental-Fox-two-twisted}) and the above expressions for $\Omega(s-1)$ and $\Omega(q-1)$,
the first block column of $A'$ is
$$
\sum^k_{j=1} A_j \, B_j =  - \Omega\left( \frac{\partial r^{}}{\partial s_{}} \right) \,\, (s - 1) ~-~
\Omega \left( \frac{\partial r}{\partial q} \right) \,\, (q - 1) 
$$
Evaluating at $q=1$, 
we see that each of the first $n$ columns of $A'|_{q=1}$
is of the form $(s-1)$ times a column vector.
Hence $\det(A') |_{q=1} = (s -1)^n \, p(s,t)$ for some polynomial $p(s,t)$.
It follows that
$$
H_K^\varrho(s,t,1) \det \left(  \varrho(a_1)\, t - I_n \right)  =  (s -1)^n \, p(s,t)
$$
and so $H_K^\varrho(s,t,1)$ is divisible by $(s-1)^n$.
\end{proof}


Combining Theorem \ref{theorem-H-K-twisted} and Lemma \ref{s-q-divides-twisted-lemma} yields:

\begin{corollary}
If $K$ is a virtual knot or link and  $\varrho \colon \VG_K \lto GL_n(R)$ is a representation such that $\varrho(s)=\varrho(q)=I_n$, then $H_K^{\varrho}(s,t,q)$ is divisible by $(s-q)^n$.
\end{corollary}


\section{A normalization for the twisted virtual Alexander polynomial}  \label{section-9}
In this section, we give a normalization for the twisted Alexander polynomial of a virtual knot or link.
Throughout this section, we assume $R$ is a unique factorization domain and 
 $\cL_R = R[s^{\pm 1}, t^{\pm 1}, q^{\pm 1}]$ will denote the ring of Laurent polynomials over $R$.

In order to define the normalization, we consider the effect of the virtual Markov moves VM1, VM2, VM3 on $H^\varrho_\be(s,t,q).$

\begin{proposition}
\label{twist-VM1-inv}
Let $ \be, \ga  \in \VB_k$.   Assume $\be_*\varrho =\varrho$.  Then 
\[
H^{\ga_* \varrho}_{\ga\be \ga^{-1}}(s,t,q)  = H^{\varrho}_\be(s,t,q).
\]
\end{proposition}
\begin{proof} 
We have
\begin{eqnarray*}
H^{\ga_* \varrho}_{\ga\be \ga^{-1}}(s,t,q) &=&  \det \left( \Psi_{\ga_*{\varrho}}(\ga\be \ga^{-1}) \,-\, I_{nk}\right) \\
&=&  \det \left( \Psi_{\be_*{\varrho}}(\ga) \Psi_{\varrho}(\be) {\Psi_{\varrho}(\ga)}^{-1}  \,-\,  I_{nk}\right)  \qquad \text{ by \eqref{psipropertytwo} } \\
&=&  \det \left( \Psi_{{\varrho}}(\ga) \Psi_{\varrho}(\be) {\Psi_{\varrho}(\ga)}^{-1}  \,-\,  I_{nk}\right)  \qquad \text{ since  $\be_*{\varrho}= {\varrho}$} \\
&=&  \det \left( \Psi_{\varrho}(\be)  \,-\,  I_{nk}\right) = H^{\varrho}_\be(s,t,q).
\end{eqnarray*}
\end{proof}

Next, we examine how $H^\varrho_\be(s,t,q)$ changes under a VM2 move.
Recall from Section \ref{section-5} that there is a natural inclusion $\VB_k \hookrightarrow \VB_{k+1}$ that we use to identify $\be \in \VB_k$ with its image under this map.
Thus $\be' =  \be \tau_k$ for a stabilization of virtual type,  $\be' =  \be \si_k^{-1}$ for a stabilization of positive type
and $\be' =  \be \si_k$ for a stabilization of negative type. 

\begin{proposition}
\label{twist-VM2-inv}
Let $\be \in \VB_k$ and assume $\be' \in \VB_{k+1}$ is obtained from $\be$ by a right stabilization.
Let $\varrho\colon F_{k+2} \to GL_n(R)$ be a representation satisfying $\be_*\varrho = \varrho$, and define $\varrho' \colon F_{k+3} \to GL_n(R)$ by setting
$\varrho'(s)=\varrho(s), \varrho'(q)=\varrho(q), \varrho'(x_i) =\varrho(x_i)$ for $1 \leq i \leq k$ and 
$$\varrho'(x_{k+1})=\begin{cases} \varrho(q^{-1} x_k q) & \text{if $\be' = \be \tau_k$,} \\
 \varrho(s^{-1} x_k s) & \text{if $\be' = \be \si_k^{\pm 1}$.} 
 \end{cases}$$
Then  $\varrho'$ satisfies $\be' _*\varrho' = \varrho'$. 
Furthermore, if  $\be'$ is obtained from $\be$ by a  stabilization of virtual or of positive type then
$$H^{\varrho'}_{\be'}(s,t,q) =  (-1)^n H^\varrho_\be(s,t,q).$$
If $\be'$ is obtained from $\be$ by a stabilization of negative type
then $$H^{\varrho'}_{\be'}(s,t,q) =  (-st)^n \det \varrho(x_{k}s) H^\varrho_\be(s,t,q).$$
\end{proposition}

\begin{proof} Let $\Psi_\varrho \colon \VB_{k} \to GL_{nk}\left(\cL_R\right)$ and $\Psi'_{\varrho'}\colon \VB_{k+1} \to GL_{n(k+1)}\left(\cL_R\right)$ be the twisted Burau representations,
and set $A = \Psi_\varrho(\be)$ and $A' = \Psi'_{\varrho'}(\be')$. 
We consider separately the case of a virtual, positive, and negative stabilization.


\noindent
{\it Virtual stabilization.}~Applying the fundamental representation \eqref{eq-fundrep} to $\be' = \be \tau_k$, we have
$$x_i^{\be'}=\begin{cases} x_i^\be & \text{if $1 \leq i \leq k-1$,} \\
 q x_{k+1}^\be q^{-1} & \text{if $i=k$,} \\
 q^{-1} x_{k}^\be q & \text{if $i=k+1$.}
 \end{cases}$$
Assume $\be_*\varrho  = \varrho.$ That $\varrho'\left(x_i^{\be'}\right) =\varrho'(x_i)$
for $1\leq i \leq k-1$ follows by definition, so we only need to check $i=k$ and $i= k+1.$

For $i=k,$ we have
$$\varrho'\left(x_k^{\be'}\right) = \varrho'\left(q x_{k+1} q^{-1}\right) = \varrho(x_k) = \varrho'(x_k).$$ 
For $i=k+1,$ we have
$$\varrho'\left(x_{k+1}^{\be'}\right) = \varrho'\left(q^{-1} x_{k}^\be q \right) = \varrho\left(q^{-1} x_{k}^\be q\right) = \varrho\left(q^{-1} x_{k} q\right) = \varrho'(x_{k+1}).$$ 
Hence  $\be'_*\varrho'  = \varrho'$.
Also,  by direct calculation, $(\tau_k)_*\varrho' = \varrho' $.

Writing $A$ and $A'$ as block matrices with entries in $M_n\left(\cL_R\right),$ we see that they are related by:
\begin{equation*}
\begin{split}
A' &= \Psi'_{\varrho'}(\be  \tau_k) = \Psi'_{(\tau_k)_*\varrho'}(\be) \,\, \Psi'_{\varrho'}( \tau_k) = \Psi'_{\varrho'}(\be) \, \Psi'_{\varrho'}( \tau_k) \\
&=\begin{bmatrix}
A & \hspace{0.3em}  {\bf 0}  \\[0.3em]
{\bf 0} & \hspace{0.3em} I_n
\end{bmatrix} 
\begin{bmatrix}
 I_{n(k-1)} &  \cdots  & {\bf 0} \\ 
\vdots & 0 & \varrho(q)q \\
{\bf 0} & {\,\varrho(q^{-1})q^{-1}} & 0
\end{bmatrix}.
\end{split}
\end{equation*}


Let $A_k$ denote the $k$-th block  column of $A$ (so $A_k$ has size $kn \times n$ when viewed as matrix over $\cL_R$), and let $E$ be the elementary block matrix
$$
E = \begin{bmatrix}
I_{n(k-1)} &  \cdots  & {\bf 0} \\ 
\vdots & I_n & 0 \\
{\bf 0} & {\,\varrho(q^{-1}) q^{-1}} & I_n
\end{bmatrix}.
$$
Observe that
$$
 \left( A' - I_{n(k+1)} \right) E= 
\begin{bmatrix}
A  - I_{nk}      &  A_k \varrho(q) q  \\[0.5em]
{\bf 0} & -I_n
\end{bmatrix}.
$$
Viewing the above block matrix as a matrix in $M_{n(k+1)}\left(\cL_R\right)$
and using the fact that $\det(E)=1$, we  find that
\begin{eqnarray*}
H^{\varrho'}_{\be'}(s,t,q) &=& \det\left(A' - I_{n(k+1)}\right) =   \det\left(\left(A' - I_{n(k+1)}\right)E\right) \\
&=&   \det\left(  
\begin{bmatrix}  A  - I_{nk}      &  A_k \varrho(q) q  \\[0.5em]  
{\bf 0} & -I_n \end{bmatrix} \right) = (-1)^n\det\left( A  - I_{nk} \right)\\
& =& (-1)^n H^{\varrho}_{\be}(s,t,q).
\end{eqnarray*}
\smallskip


\noindent
{\it Positive stabilization.}~Applying the fundamental representation \eqref{eq-fundrep} to $\be' = \be \si_k^{-1}$, we have
$$x_i^{\be'}=\begin{cases} x_i^\be & \text{if $1 \leq i \leq k-1$,} \\
 s^{-1} \left(x_{k}^\be\right)^{-1} s x_{k+1} x_k^\be & \text{if $i=k$,} \\
 s^{-1} x_k^\be s  & \text{if $i=k+1$.}
 \end{cases}$$
Assume $\be_* \varrho  = \varrho$.
That $\varrho'\left(x_i^{\be'}\right) =\varrho'(x_i)$
for $1\leq i \leq k-1$ follows by definition, so we only need to check $i=k$ and $i= k+1.$

For $i=k,$ we have
\begin{equation*}
\begin{split}
\varrho'\left(x_k^{\be'}\right) &= \varrho'\left(s^{-1} \left(x_{k}^\be\right)^{-1} s x_{k+1} x_k^\be\right) \\
&= \varrho\left(s^{-1} \left(x_{k}^\be\right)^{-1} s\right) \varrho'\left( x_{k+1}\right) \varrho\left( x_k^\be\right) \\
&= \varrho\left(s^{-1} x_{k}^{-1} s\right) \varrho\left( s^{-1} x_{k}s \right) \varrho\left( x_k\right) 
=\varrho(x_k) = \varrho'(x_k).
\end{split}
\end{equation*}

For $i=k+1,$ we have
$$\varrho'\left(x_{k+1}^{\be'}\right) = \varrho'\left(s^{-1} x_{k}^\be s \right) 
= \varrho\left( s^{-1} x_{k}^\be s \right)  = \varrho'\left( s^{-1} x_{k} s\right) =\varrho'\left( x_{k+1}\right) .$$
Thus it follows that $\be' _*\varrho' = \varrho'$.
Also,  by direct calculation, $( \si_k^{-1})_* \varrho'  = \varrho'$.

Writing $A$ and $A'$ as block matrices with entries in $M_n\left(\cL_R\right),$ using equation \eqref{eq:twisted-Burau}, the fact that $\varrho'(x_{k+1})=\varrho(s^{-1} x_{k} s)$ and the assumption that $\varrho\left(x_k^\be\right) = \varrho(x_k),$ we see that

\begin{equation*}
\begin{split}
A' &= \Psi'_{\varrho'}(\be \si_k^{-1}) = \Psi'_{ (\si_k^{-1})_* \varrho'}(\be) \,\,  \Psi'_{\varrho'}(\si_k^{-1}) = \Psi'_{\varrho'}(\be) \,  \Psi'_{\varrho'}(\si_k^{-1}) \\
&=\begin{bmatrix}
A & \hspace{0.3em} {\bf 0}  \\[0.3em]
{\bf 0} & \hspace{0.3em} I_n
\end{bmatrix} 
\begin{bmatrix}
I_{n(k-1)} &  \cdots  & {\bf 0} \\ 
\vdots & I_n-\varrho(s^{-1}x_k^{-1})(st)^{-1}& \varrho(s^{-1}x_k^{-1}s)t^{-1} \\
{\bf 0} & \varrho(s^{-1}) s^{-1} & 0
\end{bmatrix}.
\end{split}
\end{equation*}

Let $A_k$ denote the $k$-th block column of $A$ and let $E$ be the elementary block matrix
$$
E = \begin{bmatrix}
I_{n(k-1)} &  \cdots  & {\bf 0} \\ 
\vdots & I_n & 0 \\
{\bf 0} & {\,\varrho(s^{-1}) s^{-1}} & I_n
\end{bmatrix}.
$$
Observe that
$$ \left( A' - I_{n(k+1)} \right) E
= \begin{bmatrix}
A  - I_{nk}      &  A_k \varrho(s^{-1}x_k^{-1}s)t^{-1}  \\[0.5em]
{\bf 0} & -I_n
\end{bmatrix}.$$
Viewing  these $(k+1) \times (k+1)$ block matrices with entries in $M_n\left(\cL_R\right)$ as matrices in $M_{n(k+1)}\left(\cL_R\right)$,
and using the fact that $\det(E)=1$, it follows that
\begin{eqnarray*}
H^{\varrho'}_{\be'}(s,t,q) &=& \det\left(A' - I_{n(k+1)}\right) =   \det\left(\left(A' - I_{n(k+1)}\right)E\right) \\
&=&   \det\left(  
\begin{bmatrix}  A  - I_{nk}      &  A_k  \varrho(x_k^{-1}) t^{-1} \\[0.5em]  
{\bf 0} & -I_n \end{bmatrix} \right) = (-1)^n\det\left( A  - I_{nk} \right) \\
 &=& (-1)^n H^{\varrho}_{\be}(s,t,q).
\end{eqnarray*}
\smallskip


\noindent
{\it Negative stabilization.}~Applying the fundamental representation \eqref{eq-fundrep} to $\be' = \be \si_k$, we have
$$x_i^{\be'}=\begin{cases} x_i^\be & \text{if $1 \leq i \leq k-1$,} \\
 s x_{k+1}s^{-1} & \text{if $i=k$,} \\
 x_{k+1} x_k^\be s x_{k+1}^{-1} s^{-1} & \text{if $i=k+1$.}
 \end{cases}$$
Assume $\be_*\varrho  = \varrho$. That $\varrho'\left(x_i^{\be'}\right) =\varrho'(x_i)$
for $1\leq i \leq k-1$ follows by definition, so we only need to check $i=k$ and $i= k+1.$

For $i=k,$ we have
$$\varrho'\left(x_k^{\be'}\right) = \varrho'\left(s x_{k+1} s^{-1}\right) = \varrho(x_k) = \varrho'(x_k).$$ 
For $i=k+1,$ we have
\begin{equation*}
\begin{split}
\varrho'\left(x_{k+1}^{\be'}\right) &= \varrho'\left(x_{k+1} x_k^\be s x_{k+1}^{-1} s^{-1} \right) \\
&= \varrho'\left( x_{k+1}\right) \varrho\left( x_k^\be\right) \varrho'\left( s x_{k+1}^{-1} s^{-1}\right) = \varrho'\left( x_{k+1}\right) \varrho\left( x_k \right) \varrho\left( x_{k}^{-1}\right)=\varrho'\left( x_{k+1}\right) .
\end{split}
\end{equation*}
Thus it follows that $\be'_* \varrho'  = \varrho'$.
Also,  by direct calculation, $(\si_k)_* \varrho'  = \varrho'$.

Writing $A$ and $A'$ as block matrices with entries in $M_n\left(\cL_R\right),$ using equation \eqref{eq:twisted-Burau} and the fact that $\varrho'(s x_{k+1}^{-1}) = \varrho(x_{k}^{-1}s),$ we see that  

\begin{equation*}
\begin{split}
A' &= \Psi'_{\varrho'}(\be \si_k) = \Psi'_{(\si_k)_*\varrho'} (\be) \, \Psi'_{\varrho'}(\si_k) = \Psi'_{\varrho'}(\be) \,  \Psi'_{\varrho'}(\si_k) \\
&=\begin{bmatrix}
A & \hspace{0.3em} {\bf 0} \\[0.3em]
{\bf 0} & \hspace{0.3em} I_n
\end{bmatrix} 
\begin{bmatrix}
I_{n(k-1)} &  \cdots  & {\bf 0} \\ 
\vdots & 0 & \varrho(s)s \\
{\bf 0} & \varrho'(x_{k+1}) t & I_n-\varrho'(x_{k+1} s)st
\end{bmatrix}.
\end{split}
\end{equation*}

Let $A_k$ denote the $k$-th block column of $A$ and let $E$ be the elementary block matrix
$$
E = \begin{bmatrix}
I_{n(k-1)} & \cdots & {\bf 0} \\ 
\vdots & I_n & 0 \\
{\bf 0} & {\,\varrho(s^{-1}) s^{-1}} & I_n
\end{bmatrix}.
$$
Observe that
$$
 \left( A' - I_{n(k+1)} \right) E= 
\begin{bmatrix}
A  - I_{nk}      &  A_k \varrho(s) s  \\[0.5em]
{\bf 0} & -\varrho'(x_{k+1} s)st
\end{bmatrix}.
$$
Viewing the above block matrix  as a matrix in $M_{n(k+1)}\left(\cL_R\right)$
and using the fact that $\det(E)=1$ and $\det \varrho'(x_{k+1}s) = \det \varrho(s^{-1}x_{k}s^2) =\det \varrho(x_{k}s),$ it follows that
\begin{eqnarray*}
H^{\varrho'}_{\be'}(s,t,q) &=& \det\left(A' - I_{n(k+1)}\right) =   \det\left(\left(A' - I_{n(k+1)}\right)E\right)     \\
&=&   \det\left(  
\begin{bmatrix}  A  - I_{n(k+1)}      &  A_n \varrho(s) \\[0.5em]  
{\bf 0} & -\varrho'(x_{k+1} s) st \end{bmatrix} \right) \\
&=& (-st)^n \det\left(\varrho'(x_{k+1}s)\right) \det\left( A  - I_{nk} \right) \\
&=& -(st)^n \det \varrho(x_{k}s) H^\varrho_{\be}(s,t,q).
\end{eqnarray*}
\end{proof}


Next, we show that $H^\varrho_\be(s,t,q)$ is invariant under VM3 moves.

\begin{proposition}
\label{twist-VM3-inv}
Suppose $\be_1 \in \VB_{k+1}$, and $\varrho_1  \colon F_{k+3} \to GL_n(R)$ is a representation such that  $(\be_1)_*\varrho_1   = \varrho_1$. 
\begin{enumerate}[\hspace{12pt}]
\item[(i)] If $\be_2 \in \VB_{k+1}$ is related to $\be_1$ by a right virtual exchange move,  then we define $\varrho_2  \colon F_{k+3} \to GL_n(R)$ by setting $\varrho_2(s)=\varrho_1(s), \varrho_2(q)=\varrho_1(q),\varrho_2(x_i)=\varrho_1(x_i)$ for all $1 \leq i \leq k$ and $\varrho_2(q x_{k+1}q^{-1}) =\varrho_1(s^{-1} x_{k}^{-1} s x_{k+1} x_k).$ 
\item[(ii)] If $\be_2 \in \VB_{k+1}$ is related to \; $\be_1$ \; by a left virtual exchange move,  then we define \; $\varrho_2  \colon F_{k+3} \to GL_n(R)$ by setting $\varrho_2(s)=\varrho_1(s), \varrho_2(q)=\varrho_1(q),\varrho_2(x_i)=\varrho_1(x_i)$ for all $2 \leq i \leq k+1$ and $\varrho_2(q^{-1} x_{1}q) =\varrho_1(s^{-1} x_{1} s).$
\end{enumerate}

In either case, we have $(\be_2)_*\varrho_2 = \varrho_2$ and  $H^{\varrho_1}_{\be_1}(s,t,q) =  H^{\varrho_2}_{\be_2}(s,t,q)$.
\end{proposition}

\begin{proof} 
Let $\al$ and $\be$ be as in Figure \ref{Kamada-exchange}.
We consider separately the cases of right and left virtual exchange moves.

\noindent
{\it Right virtual exchange move.}

We have  $\be_1 = \al  \, \si_k \, \be  \, \si_k^{-1}  $ and  $\be_2 = \al  \, \tau_k  \, \be  \, \tau_k$.

Let $\theta$ be the automorphism of $F_{k+3}$ given by
$s^\theta = s$, $q^\theta =q$,  $x_i^\theta = x_i$ for $1\leq i \leq k$, and
$x_{k+1}^\theta = q^{-1} s^{-1} x_{k}^{-1} s x_{k+1} x_k q$.
Note that $\theta_*\varrho_1 = \varrho_2$.
The assumption that $(\be_1)_*\varrho_1   = \varrho_1$ is equivalent to $\varrho_1$ factoring
through $\pi_{\be_1}  \colon F_{k+3} \to  \VG_{\be_1}$
(see the introduction to Section \ref{section-8} for this notation).
The automorphism $\theta$ induces an isomorphism ${\bar \theta} \colon \VG_{\be_2} \to  \VG_{\be_1}$
such that 
\begin{center}
\begin{tikzcd} 
F_{k+3} \arrow{r}{\theta} \arrow{d}{\pi_{\beta_2}}
&F_{k+3} \arrow{d}{\pi_{\beta_1}} \\
  \VG_{\beta_2}   \arrow{r}{\bar{\theta}} &\VG_{\beta_1}, 
\end{tikzcd} 
\end{center}
see Figure \ref{Braid-Exchange-Loop} for a geometric proof of this fact. The diagram in Figure \ref{Braid-Exchange-Loop} is obtained as a partial closure of either of the two braids involved in the right exchange move (see Figure \ref{Kamada-exchange}) after applying ($r2$) or ($v2$) of Figure \ref{VRM}, namely the real or virtual Reidemeister two move.
It follows that $\varrho_2$ factors through $\pi_{\be_2}$ and so $(\be_2)_*\varrho_2   = \varrho_2$.

\begin{figure}[ht]
\centering
\includegraphics[scale=1.20]{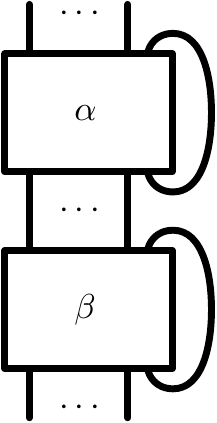}
\caption{The diagram obtained by closing up the two strands on the right in either braid appearing in the right virtual exchange move.}
\label{Braid-Exchange-Loop}
\end{figure}

An application of  \eqref{psipropertyonemultiple} yields
\begin{equation} \label{REM-beta1beta2}
\begin{split}
\Psi_{\varrho_1}(\be_1) &= \Psi_{(\si_k \be \si_k^{-1})_*{\varrho_1}} \left(\al \right) ~~ \Psi_{(\be \si_k^{-1})_*{\varrho_1}}\left(\si_k\right) ~~ \Psi_{(\si_k^{-1})_*\varrho_1}\left(\be \right) ~~ \Psi_{\varrho_1}\left(\si_k^{-1} \right),  \\
\Psi_{\varrho_2}(\be_2) &= \Psi_{(\tau_k \be \tau_k)_*{\varrho_2}}\left(\al\right)  ~ \Psi_{(\be \tau_k)_* {\varrho_2}}\left(\tau_k\right)  ~~ \Psi_{(\tau_k)_*\varrho_2 }\left(\be\right)  ~~\Psi_{\varrho_2}\left(\tau_k\right).
\end{split}
\end{equation}

Since $\si_k \be \si_k^{-1}= \al^{-1}\be_1 $ and $(\be_1)_* \varrho_1 =\varrho_1$, we have
$(\si_k \be \si_k^{-1})_*{\varrho_1} = (\al^{-1}\be_1)_* \varrho_1 = (\al^{-1})_*\varrho_1$.
Similarly, $(\be_2)_* \varrho_2 =\varrho_2$ implies $(\tau_k \be \tau_k)_*{\varrho_2} = (\al^{-1})_*\varrho_2$.
Hence
$
\Psi_{(\si_k \be \si_k^{-1})_*{\varrho_1}} \left(\al \right)
=  \Psi_{(\al^{-1})_* \varrho_1} \left( \al \right)
$
and 
$
\Psi_{(\tau_k \be \tau_k)_*{\varrho_2}} \left(\al \right)
=  \Psi_{(\al^{-1})_* \varrho_2} \left( \al \right)
$.

Consider the injective ``right stabilization'' homomorphism
$\mu \colon F_{k+2} = \langle x_1, \ldots, x_k, s,q \rangle \hookrightarrow F_{k+3}$ that
takes a generator to a generator of the same name.
Since $\varrho_1$ and $\varrho_2$ agree on the generators $x_j$  for $1 \leq j \leq k$ we
have $\mu_* \varrho_1 = \mu_* \varrho_2$  (where $\mu_* \varrho_j$ denotes the composite of $\varrho_j$ and $\mu$).
Let $\varrho = \mu_* \varrho_1$.
Note that $\mu_*(\al^{-1})_* \varrho_j = (\al^{-1})_* \varrho$ for $j=1,2$.
Let $A = \Psi_{(\al^{-1})_* \varrho} \left( \al \right)$.
Then 
\[
 \Psi_{(\si_k \be \si_k^{-1})_*{\varrho_1}} \left(\al \right) ~=~
\begin{bmatrix}
A   & \hspace{0.3em} {\bf 0}  \\[0.3em]
{\bf 0}   & \hspace{0.3em} I_n
\end{bmatrix}
~=~  \Psi_{(\tau_k \be \tau_k)_*{\varrho_2}} \left(\al \right).
\]
Call the above matrix $A'$.

The condition defining $\varrho_2$ asserts that $\left({\si_k^{-1}}\right)_*\varrho_1(x_j) = {\left(\tau_k\right)}_*\varrho_2(x_j)$ for $1 \leq j \leq k$
and so $\mu_*\left({\si_k^{-1}}\right)_*\varrho_1 =  \mu_*\left({\tau_k}\right)_*\varrho_2$.
Let ${\wh \varrho} = \mu_*\left({\si_k^{-1}}\right)_*\varrho_1$ and
let $B = \Psi_{{\wh \varrho}}(\be)$.  We conclude that
\[
 \Psi_{(\si_k^{-1})_*{\varrho_1}} \left(\be \right) ~=~
\begin{bmatrix}
B   & \hspace{0.3em} {\bf 0}  \\[0.3em]
{\bf 0}   & \hspace{0.3em} I_n
\end{bmatrix}
~=~  \Psi_{(\tau_k)_*{\varrho_2}} \left(\be \right).
\]
Call the above matrix $B'$.

Let $P' = \Psi_{\varrho_1}\left(\si_k^{-1} \right)$ and $P = \Psi_{(\be \si_k^{-1})_*{\varrho_1}}\left(\si_k\right)$.  
We have $x_k^{\be \si_k^{-1}} = s^{-1} x_k^{-1}s x_{k+1} x_k $ and $x_{k+1}^{\be \si_k^{-1}} = s^{-1}x_k s$, 
and applying Remark \ref{rem-mult} shows that $P'$ and $P$ are given by
$$
P' = \begin{bmatrix}
I_{n(k-1)} &  \cdots  & {\bf 0} \\ 
\vdots &  \varrho_1(s^{-1} x_k^{-1} s x_{k+1})-\varrho_1( s^{-1} x_{k}^{-1})s^{-1} t^{-1} & \varrho_1( s^{-1} x_{k}^{-1}s) t^{-1}  \\
{\bf 0} & \varrho_1(s^{-1})s^{-1} & 0
\end{bmatrix},
$$
and
$$
P  = \begin{bmatrix}
I_{n(k-1)} &  \cdots  & {\bf 0} \\ 
\vdots & 0 & \varrho_1(s) s  \\
{\bf 0} & \varrho_1(s^{-1} x_{k}s)t & I_n-\varrho_1(x_{k+1}s) st
\end{bmatrix}.
$$
Let $Q = \Psi_{\varrho_2}\left(\tau_k \right)$.  By Remark \ref{rem-mult}, we have $Q = \Psi_{(\be \tau_k)_*{\varrho_2}}\left(\tau_k\right)$  and so 
$$
Q = \begin{bmatrix}
I_{n(k-1)} &  \cdots & {\bf 0} \\ 
\vdots & 0 & \varrho_2(q) q \\
{\bf 0} & \varrho_2(q^{-1}) q^{-1} & 0
\end{bmatrix}.
$$
With the names given to the various matrices displayed above, we have by \eqref{REM-beta1beta2} that
$\Psi_{\varrho_1}(\be_1) = A' P B' P'$ and
$\Psi_{\varrho_2}(\be_2) = A' Q B' Q$.
Thus 
$
H^{\varrho_1}_{\be_1}  =  \det\left(A'P B' P' - I_{n(k+1)}\right)
$
and
$
H^{\varrho_2}_{\be_2}   = \det\left(A'Q B' Q- I_{n(k+1)}\right)
$.

Let
$$ E = \begin{bmatrix}
I_{n(k-1)} &  \cdots  & {\bf 0} \\ 
\vdots   & I_n & 0 \\
{\bf 0} & -\varrho_1(x_{k+1})t+\varrho_1(s^{-1})s^{-1} & \varrho_1(s^{-1}x_k s)t
\end{bmatrix},$$

$$F_1= \begin{bmatrix}
I_{nk}        & {\bf 0}  \\[0.3em]
{\bf 0}     & \varrho_1(s^{-1}x_k s)t
\end{bmatrix},
\quad \text{ and } \quad
F_2 = \begin{bmatrix}
I_{nk}  & \hspace{0.3em} {\bf 0}  \\[0.3em]
{\bf 0}    & \hspace{0.3em} \varrho_2(q) q
\end{bmatrix}.$$

Consider the matrices
$$
M_1 = F_1^{-1} \left(A' P B' P'  - I_{n(k+1)}\right) E   \quad \text{   and   }  \quad
M_2 =  F_2 \left(A' Q B' Q - I_{n(k+1)}\right) F_2^{-1}.
$$

Since $\det(E)=t^n \det \varrho_1(x_{k})=\det(F_1)$, we have $\det(M_1) = H^{\varrho_1}_{\be_1}$.  Also, $\det(M_2) = H^{\varrho_2}_{\be_2}$.


\noindent{\bf Claim:} $M_1 = M_2$. \\
\noindent
Let $A_k$ and $B_k$ denote the $k$-th block columns of $A$ and $B$, respectively, and let
$\wh A_k$ and $\wh B_k$ be the $nk\times n(k-1)$ matrices obtained  by removing the $k$-th block columns from $A$ and $B$, respectively.

One then finds that
\begin{eqnarray*}
\left( F_1^{-1} A' \right)P &=&
\begin{bmatrix}
A     & \hspace{0.3em}  {\bf 0}  \\[0.3em]
{\bf 0}     & \hspace{0.3em} \varrho_1(s^{-1} x_{k}^{-1}s)t^{-1}
\end{bmatrix}
\begin{bmatrix}
I_{n-1} &  \cdots  & {\bf 0} \\ 
\vdots & 0 & \varrho_1(s) s \\
{\bf 0} & \varrho_1(s^{-1} x_{k}s)t & I_n-\varrho_1(x_{k+1} s)st
\end{bmatrix} \\
&=& 
\begin{bmatrix}
{\wh A}_k & {\bf 0}   & A_k \varrho_1(s)s\\[0.5em] 
{\bf 0} & I_n & \varrho_1(s^{-1} x_{k}^{-1} s)t^{-1}-\varrho_1(s^{-1} x_{k}^{-1}s x_{k+1}s)s
\end{bmatrix} \quad \text{ and } \\ 
B'\left(P' E\right)  
&= &
\begin{bmatrix}
B   & \hspace{0.3em} {\bf 0}  \\[0.3em]
{\bf 0}   & \hspace{0.3em} I_n
\end{bmatrix}
\begin{bmatrix}
I_{n(k-1)} &  \cdots  & {\bf 0} \\ 
\vdots & 0 & I_n \\
{\bf 0} & \varrho_1(s^{-1})s^{-1} & 0
\end{bmatrix} 
=
\begin{bmatrix}
{\wh B}_k & {\bf 0}   & B_k\\[0.5em] 
{\bf 0} & \varrho_1(s^{-1}) s^{-1} & 0
\end{bmatrix}.
\end{eqnarray*}

Furthermore, we have
\begin{eqnarray*}F_1^{-1} E &=&
\begin{bmatrix}
I_{nk}        & {\bf 0}  \\[0.5em]
{\bf 0}     & \varrho_1(s^{-1} x_{k}^{-1}s) t^{-1}
\end{bmatrix}
\begin{bmatrix}
I_{n(k-1)} &  \cdots  & {\bf 0} \\ 
\vdots   & I_n & 0 \\
{\bf 0} & -\varrho_1(x_{k+1})t+\varrho_1(s^{-1})s^{-1} & \varrho_1(s^{-1}x_k s)t
\end{bmatrix} \\
&=&
\begin{bmatrix}
I_{n(k-1)} &  \cdots  & {\bf 0} \\ 
\vdots   & I_n & 0 \\
{\bf 0} & \varrho_1(s^{-1} x_{k}^{-1}) s^{-1}t^{-1}-\varrho_1(s^{-1} x_{k}^{-1} s x_{k+1}) & I_n
\end{bmatrix}.
\end{eqnarray*}

It follows that
\begin{equation} \label{t-M-1}
\begin{split}
M_1 &=  \left( F_1^{-1} A' P \right) \left( B' P' E\right) ~-~ F_1^{-1} E \\
&= \begin{bmatrix}
{\wh A}_k &0   & A_k \varrho_1(s)s\\[0.5em] 
{\bf 0} & I_n & \varrho_1(s^{-1} x_{k}^{-1} s)t^{-1}-\varrho_1(s^{-1} x_{k}^{-1}s x_{k+1}s)s
\end{bmatrix}
\begin{bmatrix}
{\wh B}_k &0   &  B_k\\[0.5em] 
{\bf 0} & \varrho_1(s^{-1}) s^{-1} & 0
\end{bmatrix} \\
& \qquad -
\begin{bmatrix}
I_{n(k-1)} &  \cdots  & {\bf 0} \\ 
\vdots   & I_n & 0 \\
{\bf 0} & \varrho_1(s^{-1} x_{k}^{-1}) s^{-1}t^{-1}-\varrho_1(s^{-1} x_{k}^{-1} s x_{k+1}) & I_n
\end{bmatrix}.
\end{split}
\end{equation}

An easy computation shows that 

\begin{eqnarray*}
\left( F_2 A'\right) Q &=&
\begin{bmatrix}
A  & \hspace{0.3em} {\bf 0}  \\[0.3em]
{\bf 0}   & \hspace{0.3em} \varrho_2(q) q
\end{bmatrix}
\begin{bmatrix}
I_{n(k-1)} &  \cdots & {\bf 0} \\ 
\vdots & 0 & \varrho_2(q) q \\
{\bf 0} & \varrho_2(q^{-1}) q^{-1} & 0
\end{bmatrix}
=
\begin{bmatrix}
{\wh A}_k &\hspace{0.3em} {\bf 0}   & \hspace{0.3em}A_k \varrho_2(q) q\\[0.5em] 
{\bf 0} &\hspace{0.3em} I_n &\hspace{0.3em} 0
\end{bmatrix}
 \quad \text{ and } \\
B' \left(Q F_2^{-1}\right) &=& \begin{bmatrix}
B    & \hspace{0.3em} {\bf 0}  \\[0.3em]
{\bf 0}    & \hspace{0.3em} I_n
\end{bmatrix}
\begin{bmatrix}
I_{n(k-1)} &  \cdots & {\bf 0} \\ 
\vdots & 0 & I_n \\
{\bf 0} & \varrho_2(q^{-1}) q^{-1} & 0
\end{bmatrix}
=
\begin{bmatrix}
{\wh B}_k & {\bf 0}   & B_k\\[0.5em] 
{\bf 0} & \varrho_2(q^{-1}) q^{-1} & 0
\end{bmatrix}.
\end{eqnarray*}

Thus, we find that
\begin{equation} \label{t-M-2}
\begin{split}
M_2 &=  \left( F_2 A' Q \right)  \left(  B' Q F_2^{-1}\right)  ~-~ I_{n(k+1)} \\
 &=\begin{bmatrix}
{\wh A}_k &\hspace{0.3em} {\bf 0}   & \hspace{0.3em}  A_k \varrho_2(q) q\\[0.5em] 
{\bf 0} & \hspace{0.3em} I_n & \hspace{0.3em} 0
\end{bmatrix}
\begin{bmatrix}
{\wh B}_k &{\bf 0}   & B_k\\[0.5em] 
{\bf 0} & \varrho_2(q^{-1}) q^{-1} & 0
\end{bmatrix} ~-~ I_{n(k+1)}.
\end{split}
\end{equation}

The claim now follows by direct comparison of Equations \eqref{t-M-1} and \eqref{t-M-2}, showing that
$$
H^{\varrho_1}_{\be_1} = \det(M_1) = \det(M_2) = H^{\varrho_2}_{\be_2}.
$$

\medskip


\noindent
{\it Left virtual exchange move.}

We have  $\be_1 = \al  \, \si_1 \, \be  \, \si_1^{-1}  $ and  $\be_2 = \al  \, \tau_1  \, \be  \, \tau_1$.
As in the case of the right virtual exchange move,  $(\be_2)_*\varrho_2   = \varrho_2$.

Applying  \eqref{psipropertyonemultiple}, it follows that
\begin{equation} \label{LEM-beta1beta2}
\begin{split}
\Psi_{\varrho_1}(\be_1) =& \Psi_{(\si_1 \be \si_1^{-1})_*{\varrho_1}} \left(\al \right) ~~ \Psi_{(\be \si_1^{-1})_*{\varrho_1}}\left(\si_1\right) ~~ \Psi_{(\si_1^{-1})_*\varrho_1}\left(\be \right) ~~ \Psi_{\varrho_1}\left(\si_1^{-1} \right),  \\
\Psi_{\varrho_2}(\be_2)  =& \Psi_{(\tau_1 \be \tau_1)_*{\varrho_2}}\left(\al\right) ~ \Psi_{(\be \tau_1)_* {\varrho_2}}\left(\tau_1\right)  ~~ \Psi_{(\tau_1)_*\varrho_2 }\left(\be\right)  ~~\Psi_{\varrho_2}\left(\tau_1\right).
\end{split}
\end{equation}

Since $\si_1 \be \si_1^{-1}= \al^{-1}\be_1 $ and $(\be_1)_* \varrho_1 =\varrho_1$, we have
$(\si_1 \be \si_1^{-1})_*{\varrho_1} = (\al^{-1}\be_1)_* \varrho_1 = (\al^{-1})_*\varrho_1$.
Similarly, $(\be_2)_* \varrho_2 =\varrho_2$ implies $(\tau_1 \be \tau_1)_*{\varrho_2} = (\al^{-1})_*\varrho_2$.
Hence
$\Psi_{(\si_1 \be \si_1^{-1})_*{\varrho_1}} \left(\al \right)
=  \Psi_{(\al^{-1})_* \varrho_1} \left( \al \right)$
and 
$\Psi_{(\tau_1 \be \tau_1)_*{\varrho_2}} \left(\al \right)
=  \Psi_{(\al^{-1})_* \varrho_2} \left( \al \right)$.

Consider the injective
``left stabilization'' homomorphism
$\mu \colon F'_{k+2} = \langle x_2, \ldots, x_{k+1}, s,q \rangle \hookrightarrow F_{k+3}$ that
takes a generator to a generator of the same name.
Since $\varrho_1$ and $\varrho_2$ agree on the generators $x_j$  for $2 \leq j \leq k+1$ we
have $\mu_* \varrho_1 = \mu_* \varrho_2$  (where $\mu_* \varrho_j$ denotes the composite of $\varrho_j$ and $\mu$).
Let $\varrho = \mu_* \varrho_1$.
Note that $\mu_*(\al^{-1})_* \varrho_j = (\al^{-1})_* \varrho$ for $j=1,2$.
Let $A = \Psi_{(\al^{-1})_* \varrho} \left( \al \right)$.
Then
\[
 \Psi_{(\si_1 \be \si_1^{-1})_*{\varrho_1}} \left(\al \right) ~=~
\begin{bmatrix}
I_n   & \hspace{0.3em} {\bf 0}  \\[0.3em]
{\bf 0}   & \hspace{0.3em} A
\end{bmatrix}
~=~  \Psi_{(\tau_1 \be \tau_1)_*{\varrho_2}} \left(\al \right).
\]
Call the above matrix $A'$.

The condition defining $\varrho_2$ asserts that $\left({\si_1}^{-1}\right)_*\varrho_1(x_j) = {\left(\tau_1\right)}_*\varrho_2(x_j)$ for $2 \leq j \leq k+1$
and so $\mu_*\left({\si_1}^{-1}\right)_*\varrho_1 =  \mu_*\left({\tau_1}\right)_*\varrho_2$.
Let ${\wh \varrho} = \mu_*\left({\si_1}^{-1}\right)_*\varrho_1$ and
let $B = \Psi_{{\wh \varrho}}(\be)$.  We conclude that
\[
 \Psi_{(\si_1^{-1})_*{\varrho_1}} \left(\be \right) ~=~
\begin{bmatrix}
I_n   & \hspace{0.3em} {\bf 0}  \\[0.3em]
{\bf 0}   & \hspace{0.3em} B
\end{bmatrix}
~=~  \Psi_{(\tau_1)_*{\varrho_2}} \left(\be \right).
\]
Call the above matrix $B'$.

Let $P' = \Psi_{\varrho_1}\left(\si_1^{-1} \right)$ and 
$P = \Psi_{(\be \si_1^{-1})_*{\varrho_1}}\left(\si_1\right)$.  
We have $x_1^{\be \si_1^{-1}} = s^{-1} x_1^{-1}s x_2 x_1$ and $x_2^{\be \si_1^{-1}} = s^{-1}x_1 s$,
and applying Remark \ref{rem-mult} shows that $P'$ and $P$ are given by
$$
P' = \begin{bmatrix}

\varrho_1(s^{-1} x_1^{-1} s x_{2})-\varrho_1( s^{-1} x_{1}^{-1})s^{-1} t^{-1} & \varrho_1( s^{-1} x_{1}^{-1}s) t^{-1} & {\bf 0} \\
\varrho_1(s^{-1})s^{-1} &  0 & \vdots  \\[0.3em]
{\bf 0} & \cdots & I_{n(k-1)}  
\end{bmatrix},
$$
and
$$
P  = \begin{bmatrix}
0 & \varrho_1(s) s & {\bf 0}  \\
\varrho_1(s^{-1} x_{1}s)t & I_n-\varrho_1(x_{2}s) st & \vdots  \\[0.3em]
{\bf 0} & \cdots & I_{n(k-1)}
\end{bmatrix}.
$$
Let $Q = \Psi_{\varrho_2}\left(\tau_1 \right)$.  By Remark \ref{rem-mult}, we have $Q = \Psi_{(\be \tau_1)_*{\varrho_2}}\left(\tau_1\right)$  and so 
$$
Q = \begin{bmatrix}
0 & \varrho_2(q) q & {\bf 0} \\
\varrho_2(q^{-1}) q^{-1} & 0 & \vdots\\[0.3em]
{\bf 0} & \cdots & I_{n(k-1)}
\end{bmatrix}.
$$
With the names given to the various matrices displayed above, we have by \eqref{LEM-beta1beta2} that
$\Psi_{\varrho_1}(\be_1) = A' P B' P'$ and
$\Psi_{\varrho_2}(\be_2) = A' Q B' Q$.

Let
$$ E = \begin{bmatrix}
\varrho_1(s^{-1} x_1^{-1}s)t^{-1} & \varrho_1(s^{-1}x_1^{-1}s x_2 s)s-\varrho(s^{-1}x_1^{-1}s)t^{-1} & {\bf 0} \\
0 & I_n & \vdots\\[0.3em]
{\bf 0} &  \cdots  &  I_{n(k-1)}  
\end{bmatrix},$$

$$F_1= \begin{bmatrix}
\varrho_1(s^{-1} x_1 s)t         & {\bf 0}  \\[0.3em]
{\bf 0}     & I_{nk}
\end{bmatrix},
\quad \text{ and } \quad
F_2 = \begin{bmatrix}
\varrho_2(q) q   & \hspace{0.3em} {\bf 0}  \\[0.3em]
{\bf 0}    & \hspace{0.3em} I_{nk}
\end{bmatrix}.$$

Since $\be_1 = \al \si_1 \be \si_1^{-1}$ and  $\be_2 = \al \tau_1 \be \tau_1$, we have
\begin{eqnarray*}
H^{\varrho_1}_{\be_1} &=&  \det\left(A'PB' P' - I_{n(k+1)}\right) =    \det\left(B' P' A'P  - I_{n(k+1)}\right) \quad  \text{ and }\\
H^{\varrho_2}_{\be_2}  &=& \det\left(A'Q B' Q- I_{n(k+1)}\right) =    \det\left(B' Q A' Q  - I_{n(k+1)}\right).
\end{eqnarray*}

Consider the matrices
$$M_1 = F_1 \left(B' P' A'P  - I_{n(k+1)}\right) E   \quad \text{   and   } \quad
M_2 =  F_2^{-1} \left(B' Q A' Q - I_{n(k+1)}\right) F_2.$$
Since $\det(E)=t^{-n} \det \varrho_1(x_{2})^{-1}$ and $\det(F_1)=t^n \det \varrho(x_2)$, we have $\det(M_1) = H^{\varrho_1}_\be$.   Also, $\det(M_2) = H^{\varrho_2}_{\be'}$.

\noindent{\bf Claim:} $M_1 = M_2$. \\
\noindent
Let $A_1$ and $B_1$ denote the first block columns of $A$ and $B$, respectively, and let
$\wh A_1$ and $\wh B_1$ be the $nk\times n(k-1)$ matrices obtained  by removing the first block columns from $A$ and $B$, respectively.

One then finds that
\begin{eqnarray*}
\left( F_1 B' \right)P' &=&
\begin{bmatrix}
\varrho_1(s^{-1} x_{1}s)t     & \hspace{0.3em}  {\bf 0}  \\[0.3em]
{\bf 0}     & \hspace{0.3em} B
\end{bmatrix}
\begin{bmatrix}
\varrho_1(s^{-1} x_1^{-1} s x_{2})-\varrho_1( s^{-1} x_{1}^{-1})s^{-1} t^{-1} & \varrho_1( s^{-1} x_{1}^{-1}s) t^{-1} & {\bf 0} \\
\varrho_1(s^{-1})s^{-1} &  0 & \vdots  \\[0.3em]
{\bf 0} & \cdots \hspace{0.3em} & I_{n(k-1)}  
\end{bmatrix}\\
&=& 
\begin{bmatrix}
\varrho_1(x_2)t - \varrho_1(s^{-1})s^{-1} &\hspace{0.3em} I_n   &\hspace{0.3em} {\bf 0}\\[0.5em] 
B_1 \varrho_1(s^{-1})s^{-1} & {\bf 0} & {\wh B}_1
\end{bmatrix} \quad \text{ and } \\ 
A'\left(P E\right)  
&= &
\begin{bmatrix}
I_n   & \hspace{0.3em}  {\bf 0}  \\[0.3em]
{\bf 0}   & \hspace{0.3em} A
\end{bmatrix}
\begin{bmatrix}
0 &  \varrho_1(s)s & {\bf 0} \\ 
I_n & 0 & \vdots \\[0.3em]
{\bf 0} & \cdots \hspace{0.3em} & I_{n(k-1)}
\end{bmatrix} 
=
\begin{bmatrix}
0 &\varrho_1(s)s   & {\bf 0} \\[0.5em] 
A_1 & {\bf 0} & {\wh A}_1
\end{bmatrix}.
\end{eqnarray*}

Furthermore, we have
\begin{eqnarray*}F_1 E &=&
\begin{bmatrix}
\varrho_1(s^{-1} x_1 s)t         & {\bf 0}  \\[0.3em]
{\bf 0}     & I_{nk}
\end{bmatrix}
\begin{bmatrix}
\varrho_1(s^{-1} x_1^{-1}s)t^{-1} & \varrho_1(s^{-1}x_1^{-1}s x_2 s)s-\varrho(s^{-1}x_1^{-1}s)t^{-1} & {\bf 0} \\
0 & I_n & \vdots\\[0.3em]
{\bf 0} &  \cdots  &  I_{n(k-1)}  
\end{bmatrix} \\
&=&
\begin{bmatrix}
I_n &  \varrho_1(x_2s)st -I_n  & {\bf 0} \\ 
0   & I_n & \vdots \\[0.5em]
{\bf 0} & \cdots \hspace{0.3em} &  I_{n(k-1)}
\end{bmatrix}.
\end{eqnarray*}

It follows that
\begin{equation} \label{tt-M-1}
\begin{split}
M_1 &=  \left( F_1 B' P' \right) \left( A' P E\right) ~-~ F_1 E \\
&= \begin{bmatrix}
\varrho_1(x_2)t - \varrho_1(s^{-1})s^{-1} \hspace{0.3em} &I_n   \hspace{0.3em} & {\bf 0}\\[0.5em] 
B_1 \varrho_1(s^{-1})s^{-1} & 0 & {\wh B}_1
\end{bmatrix}
\begin{bmatrix}
0 &\varrho_1(s)s   & {\bf 0} \\[0.5em] 
A_1 \hspace{0.3em} & 0 \hspace{0.3em} & {\wh A}_1
\end{bmatrix} \\
& \qquad -
\begin{bmatrix}
I_n &  \varrho_1(x_2s)st -I_n  & {\bf 0} \\ 
0   & I_n & \vdots \\[0.5em]
{\bf 0} & \cdots \hspace{0.3em} &  I_{n(k-1)}
\end{bmatrix}.
\end{split}
\end{equation}

An easy computation shows that 

\begin{eqnarray*}
 \left(F_2^{-1} B'\right) Q  &=& \begin{bmatrix}
\varrho_2(q^{-1})q^{-1}    & \hspace{0.3em} {\bf 0}  \\[0.5em]
{\bf 0}    & \hspace{0.3em} B
\end{bmatrix}
\begin{bmatrix}
0 &  \varrho_2(q)q & {\bf 0} \\ 
\varrho_2(q^{-1})q^{-1} & 0 & \vdots\\[0.5em]
{\bf 0} & \cdots \hspace{0.3em} & I_{n(k-1)}
\end{bmatrix}
=
\begin{bmatrix}
0 &I_n   & {\bf 0}\\[0.5em] 
B_1 \varrho_2(q^{-1})q^{-1}  & {\bf 0} & \wh{B}_1
\end{bmatrix} \quad \text{ and } \\
A' \left( Q F_2\right) &=&
\begin{bmatrix}
I_n  & \hspace{0.3em} {\bf 0}  \\[0.3em]
{\bf 0}   & \hspace{0.3em} A
\end{bmatrix}
\begin{bmatrix}
0 &  \varrho_2(q)q &\hspace{0.3em} {\bf 0} \\ 
I_n & 0 & \vdots \\[0.5em]
{\bf 0} & \cdots & I_{n(k-1)}  
\end{bmatrix}
=
\begin{bmatrix}
0 &\hspace{0.3em} \varrho_2(q)q   & \hspace{0.3em} {\bf 0} \\[0.5em] 
A_1   &\hspace{0.3em} {\bf 0} &\hspace{0.3em} {\wh A}_1
\end{bmatrix}.
\end{eqnarray*}

Thus, we find that
\begin{equation} \label{tt-M-2}
\begin{split}
M_2 &=  \left( F_2^{-1} B' Q \right)  \left(  A' Q F_2\right)  ~-~ I_{n(k+1)} \\
 &=\begin{bmatrix}
0 &I_n   & {\bf 0}\\[0.5em] 
B_1 \varrho_2(q^{-1})q^{-1}  &{\bf 0} & \wh{B}_1
\end{bmatrix}
\begin{bmatrix}
0 &\hspace{0.3em} \varrho_2(q)q   & \hspace{0.3em} {\bf 0} \\[0.5em] 
A_1  &\hspace{0.3em} {\bf 0} &\hspace{0.3em} {\wh A}_1
\end{bmatrix} ~-~ I_{n(k+1)}.
\end{split}
\end{equation}

The claim now follows by direct comparison of Equations \eqref{tt-M-1} and \eqref{tt-M-2}, showing that
$$
H^{\varrho_1}_{\be_1} = \det(M_1) = \det(M_2) = H^{\varrho_2}_{\be_2}.
$$
\end{proof}


We apply these results to define a preferred normalization for the twisted virtual Alexander polynomial $H^\varrho_K(s,t,q)$ as follows.
Let $K$ be a virtual knot or link  represented as the closure of a braid $\be \in \VB_k,$ 
and suppose $\varrho  \colon F_{k+2}  \to GL_n(R)$ is a representation such that  $\be_* \varrho  = \varrho$.

\begin{definition} \label{t-Defnorm}
The {\it normalized twisted virtual Alexander polynomial} is given by setting $${\wh H}^\varrho_K(s,t,q) = (-1)^{n(\writhe(\be) + v(\be))} H^\varrho_\be(s,t,q),$$
where $\writhe(\be)$ is the writhe  and $v(\be)$  is the virtual crossing number of $\wh \be.$ 
Then ${\wh H}^\varrho_K(s,t,q)$ is an invariant of virtual knots and links that is well-defined up to a factor of $\ep (st)^{jn},$ where $\ep$ is a unit in $R$ and $j \in \ZZ.$ In fact, one can further assume $\ep$ lies in the image of $\det \varrho(F_{k+2}) \subset  R$, so if $\varrho \colon F_{k+2} \to SL_n(R)$ is unimodular, then ${\wh H}^\varrho_K(s,t,q)$ is well-defined up to $(st)^{jn}.$
\end{definition}

Invariance follows from Theorem \ref{virtualMarkov}, Proposition \ref{twist-VM2-inv} and Proposition \ref{twist-VM3-inv}, which show that $(-1)^{n(\writhe(\be) + v(\be))} H^\varrho_\be(s,t,q)$
is independent of the braid representative, up to an overall factor of $\ep (st)^{jn}.$
We will write ${\wh H}^\varrho_K(s,t,q) = f$ whenever $f \in \cL$ is a Laurent polynomial such that ${\wh H}_K(s,t,q) = \ep (st)^{jn} \cdot f$ for $\ep$ a unit in $R$ and $j \in \ZZ.$

Using the normalized twisted virtual Alexander polynomial, we obtain the following improvement of Theorem \ref{twist-v-bound}. Recall that $\deg_{q^{-1}}$ and $\deg_q $ are defined for the Laurent polynomial ${\wh H}^\varrho_K(s,t,q)$ by regarding it as a polynomial in $q$ and $q^{-1}.$

\begin{theorem} \label{better-twist-v-bound}
Given a virtual knot or link $K$,  then  
$$\deg_{q^{-1}} {\wh H}^\varrho_K(s,t,q)~\leq~n \, v(K) \quad \text{ and } \quad \deg_q {\wh H}^\varrho_K(s,t,q)~\leq~n \, v(K).$$  
\end{theorem}

\subsection*{Concluding Remarks}

In the classical case, twisted Alexander invariants have been used to characterize fibered knots \cite{Friedl-Vidussi}, to study sliceness \cite{Kirk-Livingston-Twisted, Herald-Kirk-Livingston} and periodicity  \cite{Hillman-Livingston-Naik}, and to understand a partial ordering on knots defined in terms of surjections between the associated knot groups \cite{Kitano-Suzuki-Wada, Kitano-Suzuki, Horie-Kitano-Matsumoto-Suzuki}.
It would be interesting to develop similar results for virtual knots, though it is not at all obvious how to define fibered and slice knots in the virtual category. On the other hand, there is a well-developed notion of periodicity for virtual knots, and it is expected that the Alexander invariants and twisted Alexander invariants will take a special form for periodic virtual knots. Further, just as in the classical case, one can construct a partial ordering on virtual knots in terms of surjections of the virtual knot groups, and the twisted Alexander invariants can be applied to those questions by the general results of Kitano, Suzuki, and Wada in \cite{Kitano-Suzuki-Wada}. 

For classical knots, the knot group $G_K$ has a topological interpretation as the fundamental group of the complement of the knot, and a natural question to ask is whether the virtual knot group $\VG_K$ is the fundamental group of some topological space naturally associated to $K$. A related problem is to characterize which groups occur as virtual knot groups $\VG_K$ for a virtual knot $K$, and we note that the corresponding problem for the knot group $G_K$ was solved by Kim in \cite{Kim00}. It would be interesting to develop the theory of virtual knot groups and Alexander invariants for virtual tangles and for long virtual knots. It would also be interesting to construct a categorification of the virtual Alexander polynomial $H_K(s,t,q)$. Finally, while the focus of this paper has been on the case of virtual knots, one can establish similar results for virtual links using multi-variable Alexander polynomials.  We hope to explore some of these questions in future work.

\bigskip
\noindent
{\it Acknowledgements.} H. Boden and A. Nicas were supported by grants from the Natural Sciences and Engineering Research Council of Canada.
E. Dies, A. Gaudreau and A. Gerlings were supported by Undergraduate Student Research Awards from  the Natural Sciences and Engineering Research Council of Canada.

\end{document}